\documentclass[10pt]{amsart}

\usepackage{amsmath}
\usepackage{amsthm}
\usepackage{amssymb}
\usepackage{tikz-cd}

\usepackage{xcolor}
\definecolor{winered}{rgb}{0.8,0,0}
\definecolor{deepblue}{rgb}{0,0,0.8}
\usepackage[colorlinks=true]{hyperref}
\hypersetup{linkcolor=winered, citecolor=deepblue}
\usepackage[color = blue!20, bordercolor = black, textsize = tiny]{todonotes}

\newtheorem{thm}{Theorem}[section]
\newtheorem{prop}[thm]{Proposition}
\newtheorem{cor}[thm]{Corollary}

\newtheorem{lem}[thm]{Lemma}
\theoremstyle{definition}
\newtheorem{df}[thm]{Definition}
\newtheorem{rmk}[thm]{Remark}
\newtheorem{exm}[thm]{Example}
\newtheorem{const}[thm]{Construction}
\newtheorem{recoll}[thm]{Recollection}

\numberwithin{equation}{section}

\newcommand{\E}{\mathbb{E}}
\newcommand{\F}{\mathbb{F}}

\renewcommand{\L}{\mathbb{L}}

\newcommand{\N}{\mathbb{N}}
\newcommand{\Q}{\mathbb{Q}}
\newcommand{\R}{\mathbb{R}}
\newcommand{\Sphere}{\mathbb{S}}

\newcommand{\Z}{\mathbb{Z}}

\newcommand{\cB}{\mathcal{B}}
\newcommand{\cC}{\mathcal{C}}
\newcommand{\cD}{\mathcal{D}}
\newcommand{\cE}{\mathcal{E}}
\newcommand{\cF}{\mathcal{F}}
\newcommand{\cG}{\mathcal{G}}

\newcommand{\cM}{\mathcal{M}}

\newcommand{\cP}{\mathcal{P}}

\newcommand{\cS}{\mathcal{S}}

\newcommand{\rB}{\mathrm{B}}

\newcommand{\rD}{\mathrm{D}}

\newcommand{\rH}{\mathrm{H}}

\newcommand{\rN}{\mathrm{N}}

\newcommand{\ul}{\underline}

\DeclareMathOperator{\im}{im}

\newcommand{\Alg}{\mathrm{Alg}}
\newcommand{\CAlg}{\mathrm{CAlg}}
\newcommand{\CRing}{\mathrm{CRing}}
\newcommand{\sCRing}{\mathrm{sCRing}}
\newcommand{\Fun}{\mathrm{Fun}}
\newcommand{\Mack}{\mathrm{Mack}}
\newcommand{\LMod}{\mathrm{LMod}}
\newcommand{\Mod}{\mathrm{Mod}}
\newcommand{\NAlg}{\mathrm{NAlg}}
\newcommand{\Sp}{\mathrm{Sp}}
\newcommand{\Ch}{\mathrm{Ch}}
\newcommand{\Hom}{\mathrm{Hom}}
\newcommand{\Map}{\mathrm{Map}}
\newcommand{\Poly}{\mathrm{Poly}}
\newcommand{\Set}{\mathrm{Set}}
\newcommand{\fib}{\mathrm{fib}}
\newcommand{\cofib}{\mathrm{cofib}}
\newcommand{\DF}{\mathrm{DF}}
\newcommand{\gr}{\mathrm{gr}}
\newcommand{\Fil}{\mathrm{Fil}}

\newcommand{\id}{\mathrm{id}}

\newcommand{\comp}{\mathrm{comp}}
\newcommand{\colim}{\mathop{\mathrm{colim}}}

\newcommand{\Fin}{\mathrm{Fin}}
\newcommand{\Spc}{\mathrm{Spc}}

\DeclareMathOperator{\sotimes}{\square}
\newcommand{\Assoc}{\mathrm{Assoc}}
\newcommand{\LM}{\mathcal{LM}}

\newcommand{\HH}{\mathrm{HH}}
\newcommand{\HR}{\mathrm{HR}}

\newcommand{\THH}{\mathrm{THH}}

\newcommand{\THR}{\mathrm{THR}}
\newcommand{\TC}{\mathrm{TC}}
\newcommand{\TCR}{\mathrm{TCR}}
\newcommand{\Bdi}{\mathrm{B}^\mathrm{di}}
\newcommand{\Ndi}{\mathrm{N}^\mathrm{di}}

\newcommand{\tr}{\mathrm{tr}}
\newcommand{\res}{\mathrm{res}}

\newcommand{\Spec}{\mathrm{Spec}}

\begin{document}
\author{Jens Hornbostel}
\address{Bergische Universit{\"a}t Wuppertal,
Fakult{\"a}t f\"ur Mathematik und Naturwissenschaften
\\
Gau{\ss}strasse 20, 42119 Wuppertal, Germany}
\email{hornbost@uni-wuppertal.de}
\author{Doosung Park}
\address{Bergische Universit{\"a}t Wuppertal,
Fakult{\"a}t f\"ur Mathematik und Naturwissenschaften
\\
Gau{\ss}strasse 20, 42119 Wuppertal, Germany}
\email{dpark@uni-wuppertal.de}

\title{Real Topological Hochschild Homology of Perfectoid Rings}
\begin{abstract}
We refine several results of Bhatt-Morrow-Scholze on $\THH$ to $\THR$. In particular, we compute $\THR$ of perfectoid rings.
This will be useful for establishing motivic filtrations on real topological Hochschild and cyclic homology of quasisyntomic rings.
We also establish a real refinement of the Hochschild-Kostant-Rosenberg theorem.
\end{abstract}
\maketitle

\section{Introduction}
This article 
may be considered as a continuation of \cite{HP23}.
It establishes both general properties and specific computations for real topological Hochschild homology, which as usual we abbreviate by  $\THR$.
Recently, there has been a lot of progress on $\THR$, see e.g.\ 
\cite{AKGH}, 
\cite{DMP},
\cite{DMPR21}, 
\cite{DO19}, \cite{HW21}, \cite{QS22}, \cite{Par23log},
and \cite{HP23}. We refer to the introduction of the latter article for further background. 

\medskip

Besides the importance of further $\THR$ computations for their own sake, there are at least two motivations for our investigations. 

First, Bhatt, Morrow, and Scholze \cite{BMS18}, \cite{BMS19}, and \cite{BS22} have presented several approaches to integral $p$-adic Hodge theory. The approach of \cite{BMS19} relies on computations of topological Hochschild (and cyclic) homology for perfectoid and more generally quasiregular semiperfectoid rings.  They moreover study certain motivic filtrations on $\THH$ and $\TC$.
The latter then leads to a suitable definition of syntomic cohomology, and the generalization of Fontaine's period rings $A_{inf}(R)$ appears as $\pi_0 \TC^{-}(R;\Z_p)$ for perfectoid rings $R$.
In this article, we refine several key results of \cite{BMS19} to the real setting, thus making a first step towards \emph{real integral $p$-adic Hodge theory}.
This has been continued by the second author in \cite{Par23},
where $\THR$ of quasiregular semiperfectoid and quasisyntomic rings has been investigated, and motivic filtrations on $\THR$ and its companions have been established.
(There is also very interesting recent work on more classical real Hodge theory, see e.g.\ \cite{BW20}, but that's another story.)

Second, some time ago, Harpaz, Nikolaus, and Shah announced the construction of a real refinement of the cyclotomic trace. This real trace should then satisfy a real refinement of the theorems of Dundas, Goodwillie, and McCarthy, and hopefully also of the recent results of Clausen-Mathew-Morrow \cite{CMM21} as quoted in \cite[Theorem 7.15]{BMS19}. This work on the real trace has not yet appeared on arXiv, but it has already been used by Land \cite{L23}
in his work on Gabber rigidity for hermitian $K$-theory. We expect that it will lead to many more interesting results about hermitian $K$-theory, and we hope that the results of this article will be useful for this. 

\medskip
Some theorems of this article only hold for commutative rings with trivial involution. Future research is necessary to improve our understanding in the case with involutions. To prepare the ground for this, we establish large parts of the  following in the more general setting of commutative rings with involution, and even for arbitrary Green functors. We also provide computations for two important families of commutative rings with involution,
see Propositions \ref{thra2} and \ref{HRsigma} as well as Remarks \ref{addinginvolutions} and \ref{perfectoid ring with nontrivial involution}.

\medskip

We now summarize some of our main results.
For a commutative ring $R$ and a simplicial commutative $R$-algebra $A$,
the \emph{real Hochschild homology $\HR(A/R)$} is defined to be the coproduct
\[
\HR(A/R)
:=
\THR(A)\wedge_{\THR(R)}R
\]
in the category of $\Z/2$-normed spectra,
see Definition \ref{def:thrandhr}.
We prove the following real refinement of the classical Hochschild-Kostant-Rosenberg theorem (see e.g.\ \cite[Theorem 9.4.7]{Weibel}):

\begin{thm}
[see Theorems \ref{HRfiltration} and \ref{smoothHKR}]
\label{thm1.1}
There exists a natural descending filtration $\Fil_\bullet \HR(A/R)$ on $\HR(A/R)$ whose $n$th graded piece
\[
\gr^n \HR(A/R)
\simeq
(\iota \wedge_{A}^n  \L_{A/R})[n\sigma]
\]
for every integer $n$. If $A$ is a smooth $R$-algebra or if $2$ is invertible in $R$,
then this filtration is complete.
\end{thm}

We refer to Construction \ref{construction of abelian iota} for the functor $\iota$ and to Construction \ref{alphaupperstar} for the twist denoted by $[n\sigma]$. In the smooth case, note that $\iota \wedge_A^n \L_{A/R}$ is just the constant Mackey functor $\ul{\Omega_{A/R}^n}$.

We expect that this theorem  generalizes to the case when $A$ is a commutative $R$-algebra with involution and $2$ is invertible in $R$,
see Remark \ref{addinginvolutions} for the details.

The next result refines \cite[Theorem 6.1]{BMS19}.
It is a fundamental ingredient of \cite{Par23},
which establishes motivic filtrations on $\THR$ and $\TCR$ and discusses their applications to computations of real and hermitian $K$-theories assuming the results announced by Harpaz, Nikolaus, and Shah.
Recall that any commutative ring
$R$ may be considered as a constant Mackey functor $\ul{R}$ and thus leads to an equivariant Eilenberg-MacLane spectrum $\rH \ul{R}$.

\begin{thm}[see Theorem \ref{perfectoid THR}]
\label{introtheorem}
Let $R$ be a perfectoid ring.
Then there is a natural equivalence of
associative $\Z/2$-equivariant ring spectra
\[
\THR(R;\Z_p)
\simeq
T_{\rH \ul{R}} (S^{1+\sigma})
:=
\bigoplus_{n=0}^\infty
\Sigma^{n+\sigma n} \rH \ul{R},
\]
where $T_{\rH \ul{R}}(S^{1+\sigma})$ denotes the free associative $\rH \ul{R}$-algebra on $S^{1+\sigma}$.
\end{thm}

In particular, 
for perfectoid rings $\THR(R;\Z_p)$ is a \emph{very even} $\Z/2$-spectrum in the sense explained in Definition \ref{def:even} below.
In a previous version, Theorem \ref{perfectoid THR} was only proved for $p \neq 2$, but thanks to \cite[Theorem 6.20]{Par23} we now have it for all primes.
We refer to \cite[Definition 3.5]{BMS18} for perfectoid rings.
By \cite[Example 3.15]{BMS18},
a perfectoid ring is a generalization of a perfect $\F_p$-algebra to mixed characteristic.
The $p$-completion of $\Z_p[\zeta_{p^\infty}]$ is an example of a perfectoid ring,
see \cite[Example 3.6]{BMS18}.
In the proof of Theorem \ref{perfectoid THR},
one of the key properties of perfectoid rings $R$ we are using is \cite[Proposition 4.19(2)]{BMS19}: The $p$-completed cotangent complex $(\L_{R/\Z_p})_p^\wedge$ is equivalent to $R[1]$.

The proof of Theorem \ref{introtheorem} relies on the computation of $\THR(\F_p)$ due to Dotto-Moi-Patchkoria-Reeh in \cite[Theorem 5.15]{DMPR21} as the computation of $\THH(R;\Z_p)$ for perfectoid ring $R$ in \cite[Theorem 6.1]{BMS19} relies on B{\"o}kstedt's computation of $\THH(\F_p)$ \cite{Boekstedt}.
On the other hand,
we do not directly use the computation of $\THR(R)$ for perfect $\F_p$-algebra $R$ due to Dotto-Moi-Patchkoria in \cite[Remark 5.15]{DMP} (see also \cite[Proposition 6.26]{QS22}). Rather, we need to generalize their base change argument to the perfectoid case, see Lemma \ref{lem:perfectoid basechange}.

To see why inverting $2$ makes life easier sometimes, observe that taking fixed points is not right exact, and in general produces $2$-torsion in cokernels, compare  e.g. the computations in Remark \ref{omegaone} and in Proposition \ref{inverting 2}.
We also point out that Theorem \ref{thm1.1} above is true without assuming $2$ invertible also in the non-smooth case when stated in the derived $\infty$-category of $R$-modules with involution. This corresponds to the derived evaluation at $(G/e)$ of the modules over the Green functor $\ul{R}$ considered in the theorem above. This in turn corresponds to considering homotopy fixed points rather than fixed points. These two are often the same (and always after $2$-adic completion) for hermitian $K$-theory, see \cite{BKSO15}, and it seems reasonable to expect a similar pattern for $\THR$.

\medskip

After this preprint was essentially finished, we discovered that Lucy Yang gave several talks with a ``real Hochschild-Kostant Rosenberg theorem'' in the title. She has now uploaded her preprint \cite{Yang}, whose introduction contains a short comparison with the arXiv preprint version of this article from November 2024.
We hope to obtain a more precise comparison of the filtrations and real HKR theorems in joint future research with her.
Also, Angelini-Knoll mentions on his homepage an ongoing project with Kong and Quigley on even slices and real syntomic cohomology. A few weeks before the final acceptance of this article, 
they uploaded their preprint in arXiv,
and we refer to \cite[Theorem 6.7, Proposition 6.8]{AKQ} for the comparisons between our filtration on $\HR$ and the filtration in \cite{Par23} with their strongly even filtrations. Note that they use ``strongly even'' for what in this article is called ``very even''.

\medskip

Throughout this article we let $G=C_2=\Z/2$, although some arguments and results on Mackey and Green functors obviously extend to other groups $G$. 

We use the notation from \cite{HP23}.
Following \cite{BMS19} most definitions and results are stated using $\infty$-categorical language.
We use model categories only at a very few places,  see e.g.\ Recollection \ref{model structures on A-modules} and Construction \ref{construction of abelian iota}.
We always use homological indexing, which is compatible with simplicial indexing. This is essentially compatible with \cite{BMS19}.
Beware however that when working with $\rD(A)$ rather than $\rH A$-modules, \cite{BMS19} sometimes switches to cohomological indexing. e.g.\ when writing ``Tor amplitude in $[-1,0]$''.
We also refer to \cite{mathew21} for a nice survey from an algebraic topology perspective.

\bigskip

From the above discussion, it is clear that we consider $\THR$ as the central object of our studies. The real refinement of Hochschild homology, denoted $\HR$, is introduced in Definition \ref{def:thrandhr} as something built out of $\THR$. We establish several new results about $\HR$, in particular Theorem \ref{HRfiltration}. Still, we admit we are somewhat less interested in $\HR$ for its own sake,
and $\HR$ is often rather a symbol for an object that appears in all kinds of arguments involving inductions and filtrations. Indeed, results and proofs switch frequently forth and back between $\THR$ and $\HR$ in this article.

Our article uses a few results from \cite{BMS19} about the cotangent complex, and also the classical Hochschild-Kostant-Rosenberg theorem. In combination with a computation of $\THR$ of a spherical monoid ring, this leads to Lemma \ref{HR(R[N])}, which is in some sense where the concrete computations start. 

The article obviously recalls and extends all kinds of technicalities. The end of section \ref{sec2} makes precise the idea that Mackey functors and derived completion interact nicely. The fact that slices for $G=C_2$ have a very simple form, recalled in Proposition \ref{slicesexplicit}, is absolutely crucial for us. We need to develop some theory about filtrations to make some inductive arguments running, e.g.\ in Proposition \ref{prop:perfectoid HR}. Finally, we need a few results about pseudocoherence in section \ref{sec5}, entering via Lemma \ref{lem:pseudo-coherence} and then in Lemma \ref{lem:induction}.

Another crucial ingredient in \cite{BMS19} is that $A \to \THH(A)$ is universal among maps of $E_{\infty}$-rings with a circle action on the target. This is originally due to McClure–Schw\"{a}nzl–Vogt, and was reproved in \cite[Proposition IV.2.2]{NS}. There is a real version of this result in 
\cite[Definition 5.2 and Remark 5.4]{QS22}, but we have not used this in our proofs below.

\emph{Acknowledgement}: This research was conducted in the framework of the DFG-
funded research training group GRK 2240: \emph{Algebro-Geometric Methods in Algebra,
Arithmetic and Topology}.
We thank to the referee for a very detailed and helpful report.

\section{Derived \texorpdfstring{$\infty$-}{infinity }categories of Mackey functors}
\label{sec2}

As noted above in the introduction, we let $G=C_2=\Z/2$.
We often use the notation $G$ instead of $C_2$ or $\Z/2$ in statements that easily generalize to arbitrary finite groups.

See e.g.\ \cite[Definition 2.1]{LM06} for the definition of Mackey functors.
For $G=C_2$, a Mackey functor $M$ can be described as a diagram
\[
\begin{tikzcd}
M(G/e)\ar[loop below,"w"]\ar[r,shift left=0.5ex,"\tr"]\ar[r,shift right=0.5ex,leftarrow,"\res"']&
M(G/G)
\end{tikzcd}
\]
such that $\res\circ \tr = \id + w$.
For an abelian group $M$ with involution $w$,
the associated Mackey functor $\ul{M}$ is given by
\[
\begin{tikzcd}
M\ar[loop below,"w"]\ar[r,shift left=0.5ex,"\tr"]\ar[r,shift right=0.5ex,leftarrow,"\res"']&
\{m\in M:w(m)=m\},
\end{tikzcd}
\]
where $\res$ is the inclusion, and $\tr:=1+w$.
As a special case,
for an abelian group $M$,
the \emph{constant Mackey functor} $\ul{M}$ is given by
\[
\begin{tikzcd}
M\ar[loop below,"\id"]\ar[r,shift left=0.5ex,"2"]\ar[r,shift right=0.5ex,leftarrow,"\id"']&
M.
\end{tikzcd}
\]
For a commutative ring $R$ with involution, $\ul{R}$ admits a natural commutative Green functor structure. We recall that a Green functor is a monoid object in Mackey functors with respect to the monoidal structure given by the box product $\sotimes$, and refer e.g.\ to \cite[p.\ 61]{lewis88}, \cite[section 1.3]{Bouc}, and \cite[sections A.3 and A.4]{HP23} for further details. Throughout this article, we only deal with commutative Green functors, hence we will omit the "commutative" from our terminology.

\begin{rmk}
Since $\res$ for any Mackey functor associated with an abelian group with involution is injective, there is a fully faithful embedding from abelian groups with involution to Mackey functors. Looking at the Burnside Mackey functor, which is not a $\ul{\Z}$-module, one sees that this embedding is not essentially surjective. This functor extends to an embedding from commutative rings with involution to Green functors, which is fully faithful as well. The two  notions of module coincide via this embedding:
if $M$ is an abelian group and $R$ is a commutative ring,  both with or without involution,
then $\ul{M}$ is an $\ul{R}$-module in the sense of \cite[pp.\ 62--63]{lewis88} if and only if $M$ is an $R$-module.
This easily follows from Lemma \ref{lem:iotaR-modules}
in the case without involution on $R$. 
\end{rmk}

\begin{recoll}
\label{model structures on A-modules}
For a Green functor $A$,
let $\Mod_{A}$ be the category of $A$-modules.
This is an abelian category, in which cokernel and kernel can be computed pointwise.
Let $\Ch(A):=\Ch(\Mod_A)$ be the category of chain complexes of $A$-modules.
For every integer $n$,
we have the homology functor
\[
\ul{H}_n\colon \Ch(A)\to \Mod_A
\]
sending a chain complex $\cdots \to \cF_{n+1}\xrightarrow{d_{n+1}} \cF_n \xrightarrow{d_n} \cF_{n-1}\to \cdots$ to $\ker d_n/\im d_{n+1}$.
A morphism $\cF\to \cG$ in $\Ch(A)$ is a \emph{quasi-isomorphism} if $\ul{H}_n(\cF)\to \ul{H}_n(\cG)$ is an isomorphism of $A$-modules for every integer $n$.
Observe that a quasi-isomorphism in $\Ch(A)$ is a pointwise quasi-isomorphism in the sense that $\cF(G/e)\to \cG(G/e)$ and $\cF(G/G)\to \cG(G/G)$ are quasi-isomorphisms.

Let $\rD(A):=\rD(\Mod_{A})$ be the derived $\infty$-category of $\Mod_{A}$,
which is obtained by inverting quasi-isomorphisms in $\Ch(A)$ in the $\infty$-categorical sense.
For every integer $n$,
we have the homology functor
\[
\ul{H}_n\colon \rD(A)\to \Mod_{A}.
\]
For $\cF\in \rD(A)$,
we set
\[
H_n^{G}(\cF):=\ul{H}_n(\cF)(G/G).
\]
Let $\tau_{\leq n},\tau_{\geq n}\colon \rD(A)\to \rD(A)$ denote the truncation functors in the sense of derived $\infty$-categories of abelian categories,
see \cite[Notation 1.2.1.7]{HA}.
Here we use the standard $t$-structure on the derived $\infty$-category of an abelian category.
Hence $\cF \in  \tau_{\leq n} \rD(A)$ if and only if $\ul{H}_m(\cF)=0$ for all $m>n$, and similarly for $\tau_{\geq n} \rD(A)$.
We say that $\cF$ is \emph{$(n-1)$-connected} if the induced morphism $\tau_{\geq n}\cF\to \cF$ in $\rD(A)$ is an equivalence.

Recall that the \emph{Burnside category} $\cB_G$ consists of the finite $G$-sets, with
morphisms given by
\[
\Hom_{\cB_G}(X,Y)
:=
\Hom_{\Sp^G}(\Sigma^\infty X_+,\Sigma^\infty Y_+)
\]
for finite $G$-sets $X$ and $Y$.
Here, $\Sp^G$ denote the $\infty$-category of $G$-spectra,
which is obtained by formally inverting the $G$-representation spheres in the $\infty$-category of pointed $G$-spaces,
see \cite[section 9.2]{BH21}.
We refer to \cite[section V.9]{LMS86} for an algebraic description of $\cB_G$.
Recall from \cite[Definition 2.1]{LM06} that a Mackey functor is a contravariant additive functor from $\cB_G$ to the category of abelian groups.
For a finite $G$-set $X$,
we set $\ul{\cB}^X:=\Hom_{\cB_G}(-,X)$,
which we regard as a Mackey functor.
Recall the box product $\sotimes$ of Mackey functors from above.
The set of objects $A\sotimes \ul{\cB}^{G/H}[n]$
for all subgroups $H$ of $G$ and integers $n$ generates $\rD(A)$ since
\[
\Hom_{\rD(A)}
(A\sotimes \ul{\cB}^{G/H}[n],\cF)
\simeq
\ul{H}_n(\cF)(G/H)
\]
for $\cF\in \rD(A)$.

By the non-graded versions of \cite[Propositions 4.3, 4.4]{LM06},
the abelian category $\Mod_{A}$ has enough projectives,
and every projective $A$-module is a direct summand of a product of $A$-modules of the form $A\sotimes \ul{\cB}^{G/H}$,
where $H$ is a subgroup of $G$.
Hence \cite[Theorem 2.2]{CH02} implies that $\Ch(A)$ admits the projective model structure,
where a morphism $\cF\to \cG$ in $\Ch(A)$ is a projective fibration (resp.\ weak equivalence) if and only if $\ul{H}_n(\cF)\to \ul{H}_n(\cG)$ is an epimorphism of Mackey functors for every $n\in \Z$ (resp.\ a quasi-isomorphism).
In particular,
every object in $\Ch(A)$ is projectively fibrant.
A (homologically) bounded below complex of projective $A$-modules is a cofibrant object in $\Ch(A)$,
see \cite[Lemma 2.7(b)]{CH02}.

On the other hand,
$\Ch(A)$ admits the injective model structure by \cite[Proposition 3.13]{Bek00},
where a morphism $\cF\to \cG$ in $\Ch(A)$ is an injective cofibration if and only if $\ul{H}_n(\cF)\to \ul{H}_n(\cG)$ is a monomorphism of Mackey functors for every $n\in \Z$.
In particular,
every object in $\Ch(A)$ is injectively cofibrant.
Hence the coproduct in $\rD(A)$ can be computed using the coproduct in $\Ch(A)$.
It follows that for every set of objects $\{\cF_i\}_{i\in I}$ and integer $n$,
we have a natural isomorphism
\[
\ul{H}_n(\bigoplus_{i\in I} \cF_i)
\cong
\bigoplus_{i\in I} \ul{H}_n(\cF_i).
\]
Use this to show that  $A\sotimes \ul{\cB}^{G/H}[n]$ is a compact object of $\rD(A)$ for every subgroup $H$ of $G$ and integer $n$.

Since $\Mod_{A}$ is a symmetric monoidal abelian category,
$\rD(A)$ has a natural symmetric monoidal product $\sotimes_{A}^\L$,
which is the derived tensor product.
Using \cite[Lemma 4.1.8.8]{HA},
we see that $\sotimes_{A}^\L$ preserves colimits in each variable.
In particular,
$\sotimes_{A}^\L$ is a bi-exact functor.
\end{recoll}

\begin{exm}
Let $R$ be a commutative ring,
and consider the constant Green functor $\ul{R}$.
Let $M$ be an $\ul{A}$-module.
There is a natural isomorphism
\[
(M\sotimes \ul{\cB}^X)(Y)
\cong
M(X\times Y),
\]
see \cite[p.\ 519]{LM06}.
Use this when $M=\ul{R}$ to obtain a natural isomorphism
\[
\ul{R}\sotimes \ul{\cB}^{G/H}
\cong
\ul{R^{\oplus G/H}}
\]
for every subgroup $H$ of $G$,
where $R^{\oplus G/H}$ is the $R$-module with the $G$-action that permutes the coordinates,
and $\ul{R^{\oplus G/H}}$ is the associated $\ul{R}$-module.
\end{exm}

\begin{lem}
\label{right adjoint colimits}
Let $F\colon \cC\to \cD$ be a colimit preserving functor of stable $\infty$-categories that admit colimits,
and let $S$ be a set of compact objects of $\cC$ that generates $\cC$.
If $F$ sends every object of $S$ to a compact object of $\cD$,
then a right adjoint $G$ of $F$ preserves colimits.
\end{lem}
\begin{proof}
This is well-known.
By \cite[Proposition 1.4.4.1(2)]{HA},
it suffices to show that $G$ preserves sums.
To finish the proof,
use the property that the functor corepresented by a compact object preserves coproducts.
\end{proof}

\begin{const}
\label{construction of abelian iota}
Let $R$ be a commutative ring.
There is an adjunction
\[
\iota
:
\Mod_R
\rightleftarrows
\Mod_{\ul{R}}
:
(-)^{G}
\]
such that $\iota M:=\ul{M}$ for every $A$-module $M$ and $\ul{L}^{G}:=\ul{L}(G/G)$ 
for every $\ul{R}$-module $\ul{L}$.
We obtain the induced adjunction
\[
\iota
:
\Ch(R)
\rightleftarrows
\Ch(\ul{R})
:
(-)^{G}.
\]
Using the description of the projective model structure on $\Ch(\ul{A})$ in Recollection \ref{model structures on A-modules},
we see that $(-)^{G}$ is a right Quillen functor with respect to the projective model structures.
Hence we obtain the induced adjunction
\[
\iota
:
\rD(R)
\rightleftarrows
\rD(\ul{R})
:
(-)^{G}.
\]
\end{const}

\begin{const}\label{construction of abelian isharp}
Let $R$ be a commutative ring.
Then the change of groupoids $(i_{\sharp},i^*)$ with respect to $i:* \to BG$, see e.g.\ \cite[section A.1]{HP23}, induces  an adjunction
\[
i_\sharp
:
\Mod_R
\rightleftarrows
\Mod_{\ul{R}}
:
i^*
\]
such that $i_\sharp M:=\ul{M^{\oplus G}}$ for every $R$-module $M$ and $i^*\ul{L}:=\ul{L}(G/e)$ for every $\ul{R}$-module $\ul{L}$.
We obtain the induced adjunction
\[
i_\sharp
:
\Ch(R)
\rightleftarrows
\Ch(\ul{R})
:
i^*.
\]
Using the description of the projective model structure on $\Ch(\ul{R})$ in Recollection \ref{model structures on A-modules},
we see that $i^*$ is a right Quillen functor with respect to the projective model structures.
Hence we obtain the induced adjunction
\[
i_\sharp
:
\rD(R)
\rightleftarrows
\rD(\ul{R})
:
i^*.
\]
Use Lemma \ref{right adjoint colimits} to see that $i^*$ is colimit preserving.
\end{const}

\begin{const}\label{alphaupperstar}
Let $(\Fin_{\rB \Z/2})_*$ denote the category of pointed finite $\Z/2$-sets, and let $A$ be a Green functor.
Consider the functor
\[
\alpha^*\colon (\Fin_{\rB \Z/2})_*
\to
\rD(A)
\]
sending a pointed finite $\Z/2$-set $(V,p)$ to
$(A\sotimes \ul{\cB}^V) / (A \sotimes \ul{\cB}^{\{p\}})$.
This functor sends finite coproducts to finite direct sums and is monoidal if the monoidal structure on $(\Fin_{\rB \Z/2})_*$ is given by the smash product.
Let $\Spc^{\Z/2}_*$ denote the category of pointed $\Z/2$-spaces,
which is equivalent to $\cP_{\Sigma}((\Fin_{\rB \Z/2})_*)$,
see \cite[Lemma 2.2, section 9.2]{BH21}.
Hence we have the induced symmetric monoidal functor
\[
\alpha^*\colon 
\Spc^{\Z/2}_*
\to
\rD(A).
\]
Let $S^\sigma$ be the topological $\Z/2$-space $S^1$ with the $\Z/2$-action given by $(x,y)\in S^1\subset \R^2\mapsto (x,-y)$.
We use the same notation $S^\sigma$ for the associated $\Z/2$-space,
whose value at $(\Z/2)/e$ is $S^1$ and value at $(\Z/2)/(\Z/2)$ is $(S^\sigma)^{\Z/2}=S^0$.
Observe that $S^\sigma$ is equivalent to the homotopy cofiber of the
equivariant map $(\Z/2)_+\to *_+$,
where as usual $\Z/2$ is considered with the non-trivial involution.
Hence $\alpha^*$ sends $S^\sigma$  
to the complex
\[
A[\sigma]:=[A\sotimes \ul{\cB}^{\Z/2}\xrightarrow{f} A\to 0],
\] 
where $A$ sits in degree $0$. 
It admits a monoidal inverse
\[
A[-\sigma]
:=
[0\to A\xrightarrow{g} A\sotimes \ul{\cB}^{\Z/2}],
\]
where $A$ sits in degree $0$.
If $A=\ul{R}$ for some commutative ring $R$ with the trivial involution,
then $f$ becomes the summation $\ul{R^{\oplus \Z/2}}\to \ul{R}$,  as without  involutions the two composite morphisms $R\rightrightarrows R\oplus R\to R$ with the two inclusions from the summands correspond to the
two composite maps $*\rightrightarrows \Z/2\to *$ and hence both composite morphisms $R\rightrightarrows R$ are equal to the identity. Also, $g$ becomes the diagonal morphism $\ul{R}\to \ul{R^{\oplus^{\Z/2}}}$.
Together with \cite[Proposition 2.9(1)]{Rob15} and the $\infty$-categorical construction of $\Sp^{\Z/2}$ in \cite[section 9.2]{BH21},
we obtain a colimit preserving symmetric monoidal functor
\[
\alpha^*
\colon
\Sp^{\Z/2}
\to
\rD(A)
\]
sending $\Sigma^\infty (V,p)$ to $(A\sotimes \ul{\cB}^V) / (A \sotimes \ul{\cB}^{\{p\}})$ for every pointed finite $\Z/2$-set $(V,p)$.
Let $\alpha_*$ be a right adjoint of $\alpha^*$.
\end{const}

\begin{df}\label{defshiftabelian}
Let $A$ be a Green functor.
Recall that $\sotimes_{A}^\L$ denotes the derived tensor product in $\rD(A)$.
For $\cF\in \rD(A)$ and integers $m$ and $n$,
we have the \emph{equivariant shift}
\[
\cF[m+n\sigma]
:=
\cF[m] \sotimes_{A}^\L (A[\sigma])^{\sotimes_A^\L n}.
\]
If $n=0$, then this is the usual shift.
We set
\[
\ul{H}_{m+n\sigma}(\cF)
:=
\ul{H}_m(\cF[-n\sigma])
\text{ and }
H_{m+n\sigma}^{\Z/2}(\cF)
:=
H_m^{\Z/2}(\cF[-n\sigma]).
\]
\end{df}

\begin{rmk}
\label{another generators}
Let $A$ be a Green functor.
Recall that the family of $A[n]$ and $A \sotimes \ul{\cB}^{\Z/2}[n]$ for integers $n$ generates $\rD(A)$.
Using the cofiber sequence $A \sotimes \ul{\cB}^{\Z/2}\to A \to A[\sigma]$,
we see that the family of $A[n]$ and $A[n+\sigma]$ for integers $n$ also generates $\rD(A)$. 
\end{rmk}

Recall from above that $\Sp^G$ denotes the $\infty$-category of (genuine) $G$-spectra.
We have the fixed point functor $(-)^G\colon \Sp^G\to \Sp$,
which admits a left adjoint $\iota$.
We have the forgetful functor $i^*\colon \Sp^G\to \Sp$,
which admits a left adjoint $i_\sharp$.
We refer to \cite[appendix A]{HP23} for a review of equivariant stable homotopy theory including the details about the functors $\iota$, $(-)^G$, $i_\sharp$, and $i^*$.

For the next results,
recall also the functors in Constructions \ref{construction of abelian iota} and \ref{construction of abelian isharp}.

\begin{prop}
\label{alpha isharp}
Let $R$ be a commutative ring.
Then there are commutative squares
\[
\begin{tikzcd}
\Sp\ar[d,"\alpha^*"']\ar[r,"\iota"]&
\Sp^{\Z/2}\ar[d,"\alpha^*"]
\\
\rD(R)\ar[r,"\iota"]&
\rD(\ul{R}),
\end{tikzcd}
\quad
\begin{tikzcd}
\Sp\ar[d,"\alpha^*"']\ar[r,"i_\sharp"]&
\Sp^{\Z/2}\ar[d,"\alpha^*"]
\\
\rD(R)\ar[r,"i_\sharp"]&
\rD(\ul{R}).
\end{tikzcd}
\]
\end{prop}
\begin{proof}
There are equivalences
$
\iota \alpha^* \Sphere \simeq \ul{R} \simeq \alpha^* \iota \Sphere
$
and
$
i_\sharp \alpha^* \Sphere \simeq \ul{R^{\oplus G}} \simeq \alpha^* i_\sharp \Sphere
$.
By \cite[Corollary 1.4.4.6]{HA},
we obtain the desired commutative squares.
\end{proof}

For a symmetric monoidal $\infty$-category $\cC$ and a commutative ring object $R$ of $\cC$,
let $\Mod_R:=\Mod_R(\cC)$ denote the $\infty$-category of $R$-modules.

For a Mackey functor $M$,
let $\rH M$ denote the equivariant Eilenberg-MacLane spectrum.
For a Green functor $A$,
we can regard $\rH A$ as a commutative algebra object of $\Sp^{\Z/2}$, compare \cite[section A.4]{HP23}.
For an $A$-module $M$,
we can regard $\rH M$ as an $\rH A$-module object of $\Sp^{\Z/2}$.
See \cite[sections A.3, A.4]{HP23} for a review.

The following is an equivariant refinement of the stable Dold-Kan correspondence \cite[Theorem 5.1.6]{SS03} proved by Schwede-Shipley.
In the case of the Burnside Mackey functor,
the following is due to Patchkoria-Sanders-Wimmer \cite[Theorem 5.10]{PSW22}.

\begin{prop}\label{greenequivalence}
Let $A$ be a Green functor, e.g.\ $A$ a commutative ring with possibly non-trivial involution.
Then the functor $\rH$
above induces an equivalence of symmetric monoidal stable  $\infty$-categories
\[
\Mod_{\rH A}
\simeq
\rD(A).
\]
\end{prop}
\begin{proof}
Let $\cF$ be the family of $\Sigma^n\Sphere$ and $\Sigma^n \Sigma^\infty (\Z/2)_+$ for all integers $n$.
Recall the standard fact that $\cF$ is a set of compact objects of $\Sp^{\Z/2}$ that generates $\Sp^{\Z/2}$,
see e.g.\ \cite[Proposition A.1.6]{HP23}.
As the images of all the objects in $\cF$ under $\alpha^*$ are obviously compact,
Lemma \ref{right adjoint colimits} implies that a right adjoint $\alpha_*$ to $\alpha^*
\colon
\Sp^{\Z/2}
\to
\rD(A)$
preserves colimits.

We claim that $\alpha_*$ is conservative.
For this,
it suffices to show that the family of functors $\Map_{\Sp^{\Z/2}}(X,\alpha_*(-))$ for $X\in \cF$ is conservative.
This holds by adjunction since $\alpha^*$ sends $\cF$ to a set of generators of $\rD(A)$.
The canonical morphism
\[
X\wedge \alpha_* Y\to \alpha_*(\alpha^*X\sotimes_{A}^\L Y)
\]
in $\Sp^{\Z/2}$ is an equivalence for $X\in \Sp^{\Z/2}$ and $Y\in \rD(A)$ since this holds for $X=\Sphere,\Sigma^\infty (\Z/2)_+$ and $\alpha^*$ and $\alpha_*$ preserve colimits.
We finish the proof together with \cite[Corollary 4.8.5.21]{HA}.
\end{proof}

\begin{rmk}
\label{Dold-Kan remark}
Let $A$ be a Green functor.
Due to \cite[section 3.4.1]{HA}, the $\infty$-category of commutative algebra objects of $\Sp^{\Z/2}$ equipped with a map from $\rH A$ is equivalent to the $\infty$-category of commutative algebra objects of $\rD(A)$.
Let $\cF$ be a commutative algebra object of $\rD(A)$.
Loc.\ cit.\ implies that the $\infty$-category of $\cF$-modules in $\rD(A)$ is equivalent to the $\infty$-category of $\alpha_*\cF$-modules in $\Sp^{\Z/2}$.

We also note that under the equivalence $\Mod_{\rH A}\simeq \rD(A)$,
$\wedge_{\rH A}$ corresponds to $\sotimes_A^\L$ (see \cite[Proposition A.5.7]{HP23}),
$\Sigma^{m+n\sigma}$ (defined by the smash product with $S^m \wedge (S^\sigma)^{\wedge n}$) corresponds to $[m+n\sigma]$ for integers $m$ and $n$,
$\ul{\pi}$ corresponds to $\ul{H}$.
and the slice filtration corresponds to the filtration in Definition \ref{abelian rho}.
\end{rmk}

\begin{const}
\label{const: derived functors of Green functors}
Let $f\colon A\to B$ be a morphism of Green functors.
We have an adjunction between the categories of chain complexes
\[
f^*
:
\Ch(A)
\rightleftarrows
\Ch(B)
:
f_*,
\]
where $f^*$ sends a complex $\cF$ to the base change $B\Box_{A} \cF$,
and $f_*$ is the forgetful functor.
Since $f_*$ preserves quasi-isomorphisms and projective fibrations,
$f_*$ becomes a right Quillen functor if we impose the projective model structures on $\Ch(A)$ and $\Ch(B)$.
Hence we obtain the adjunction between $\infty$-categories
\[
f^*
:
\rD(A)
\rightleftarrows
\rD(B)
:
f_*,
\]
and this $\infty$-categorical right derived functor $f_*$ is conservative and preserves colimits as in the proof of Proposition \ref{greenequivalence}.
\end{const}

\begin{prop}
\label{abelian Dold-Kan}
Let $A\to B$ be a morphism of Green functors.
We regard $B$ as a commutative algebra object of $\rD(A)$.
Then there exists an equivalence of symmetric monoidal stable $\infty$-categories
\[
\Mod_{B}(\rD(A))
\simeq
\rD(B).
\]
\end{prop}
\begin{proof}
Combine Proposition \ref{greenequivalence} with \cite[Corollary 3.4.1.9]{HA}.
\end{proof}

\begin{prop}
\label{conservativity}
Let $R$ be a commutative ring.
Then the pair of functors $i^*,(-)^{\Z/2}\colon \rD(\ul{R})\to \rD(R)$ is conservative.
\end{prop}
\begin{proof}
We observed that the functor $\alpha_*\colon \rD(\ul{R})\to \Sp^{\Z/2}$ is conservative
in the proof of Proposition \ref{greenequivalence}.
Hence it suffices to show that the pair of functors $\alpha_*i^*,\alpha_*(-)^{\Z/2}\colon \rD(\ul{R})\to\Sp$ is conservative.
Proposition \ref{alpha isharp} yields natural equivalences $\alpha_*i^*\simeq i^*\alpha_*$ and $\alpha_*(-)^{\Z/2}\simeq (-)^{\Z/2}\alpha_*$.
To conclude,
observe that the pair of functors $i^*,(-)^{\Z/2}\colon \Sp^{\Z/2}\to \Sp$ is conservative and $\alpha_*\colon \rD(R)\to \Sp$ is conservative.
\end{proof}

\begin{lem}
\label{lem:box product}
Let $M$ be a Mackey functor,
and let $N$ be an abelian group.
Then the box product $M\sotimes \ul{N}$ can be described as
\[
\begin{tikzcd}
M(C_2/e)\otimes N\ar[loop below,"w\otimes \id"]\ar[r,shift left=0.5ex,"\tr \otimes \id"]\ar[r,shift right=0.5ex,leftarrow,"\res \otimes \id"']&
(M(C_2/C_2)\otimes N
)/I,
\end{tikzcd}
\]
where $I$ is the ideal generated by
$(\tr(\res(x))-2)\otimes z$
for $x\in \ul{M}(C_2/C_2)$ and $z\in N
$.
\end{lem}
\begin{proof}
By \cite[Proposition 1.5.1]{Bouc} (see also \cite[Lemma 2.46]{Shulman} for the special case that we need),
$(M\sotimes \ul{N})(C_2/C_2)$ is the quotient of
\[
(M(C_2/C_2)\otimes N)\oplus (M(C_2/e)\otimes N)
\]
by the ideal $J$ generated by
\begin{gather*}
(x\otimes \tr(z),0)-(0,\res(x)\otimes z),
\\
(\tr(y)\otimes z,0)-(0,y\otimes \res(z)),
\\
(0,w(y) \otimes z)- (0,y \otimes z)
\end{gather*}
for $x\in M(C_2/C_2)$, $y\in M(C_2/e)$, and $z\in N$.
Recall that we have $\res(z)=z$ and $\tr(z)=2z$.
The third relation follows from the second since
\[
(0,w(y)\otimes z)
=
(\tr(w(y))\otimes z,0)
=
(\tr(y)\otimes z,0)
=
(0,y\otimes z).
\]
We also have
\[
(x\otimes 2z,0)=(0,\res(x)\otimes z)=(\tr(\res(x))\otimes z,0).
\]
Hence we obtain the isomorphism
\[
((M(C_2/C_2)\otimes N)\oplus (M(C_2/e)\otimes N))/J
\to
(M(C_2/C_2)\otimes N)/I
\]
induced by $\id \otimes \id$ on the left summand and by $\id\otimes \tr$ on the right summand.
The description of the transition, restriction, and involution for $M\sotimes \ul{N}$ is also due to loc.\ cit.
\end{proof}

\begin{lem}
\label{lem:box base change}
Let $R\to A$ be a homomorphism of commutative rings,
and let $M$ be an $\ul{R}$-module.
Then the base change $M\sotimes_{\ul{R}}\ul{A}$ is obtained by taking $\otimes_{R} A$ pointwise,
i.e., $M\sotimes_{\ul{R}}\ul{A}$ can be described as
\[
\begin{tikzcd}
M(C_2/e)\otimes_R A\ar[loop below,"w\otimes \id"]\ar[r,shift left=0.5ex,"\tr \otimes \id"]\ar[r,shift right=0.5ex,leftarrow,"\res \otimes \id"']&
M(C_2/C_2)\otimes_R A.
\end{tikzcd}
\]
\end{lem}
\begin{proof}
By Lemma \ref{lem:iotaR-modules},
we have $\tr\circ \res=2$ for $M$.
Together with Lemma \ref{lem:box product},
we see that $M\sotimes \ul{A}$ is obtained by taking $\otimes A$ pointwise to $M$,
i.e., $M\sotimes \ul{A}$ can be described as
\[
\begin{tikzcd}
M(C_2/e)\otimes  A\ar[loop below,"w\otimes \id"]\ar[r,shift left=0.5ex,"\tr \otimes \id"]\ar[r,shift right=0.5ex,leftarrow,"\res \otimes \id"']&
M(C_2/C_2)\otimes  A.
\end{tikzcd}
\]
We can similarly show that $M\sotimes \ul{R}\sotimes \ul{A}$ is obtained by taking $\otimes R \otimes A$ pointwise to $M$.
Since $M\sotimes_{\ul{R}} \ul{A}$ is the equalizer of the two induced morphisms $M\sotimes \ul{R}\sotimes \ul{A} \rightrightarrows M\sotimes \ul{A}$,
the above descriptions for $M\sotimes \ul{A}$ and $M\sotimes \ul{R} \sotimes \ul{A}$ imply the claim.
\end{proof}

We considered modules over arbitrary Green functors in Recollection \ref{model structures on A-modules}. In the special case of constant Green functors, these can be described as follows.
The next Lemma shows in particular that a Mackey functor is a module over $\ul{\Z}$ if and only if $\tr \circ \res = 2$.

\begin{lem}
\label{lem:iotaR-modules}
Let $R$ be a commutative ring.
Then the objects of the category of $\ul{R}$-modules are precisely those Mackey functors $M$ for which $M(C_2/e)$ and $M(C_2/C_2)$ are equipped with $R$-module structures, $\res$, $\tr$, and $w$ are $R$-linear, and $\tr\circ \res=2$. 
The morphisms in the category of $\ul{R}$-modules are the pointwise $R$-linear morphisms of Mackey functors.
\end{lem}
\begin{proof}
Let $M$ be an $\ul{R}$-module.
We have the module structure morphism $\ul{R}\sotimes M\to M$ of Mackey functors.
The relation
$\tr\circ \res=2$
holds for $\ul{R}\sotimes M$ by Lemma \ref{lem:box product}.
This implies that the relation
$\tr\circ \res=2$ still holds for $M$ since the structure morphism $\ul{R}\sotimes M\to M$ is an epimorphism.
With this in hand,
we see that the box product $\ul{R}\sotimes M$ can be described as
\[
\begin{tikzcd}
\ul{R}\otimes M(C_2/e)\ar[loop below,"\id\otimes w"]\ar[r,shift left=0.5ex,"\id \otimes \tr"]\ar[r,shift right=0.5ex,leftarrow,"\id \otimes \res"']&
\ul{R}\otimes M(C_2/C_2)
\end{tikzcd}
\]
by Lemma \ref{lem:box product} again.
Then we can interpret the module axioms for $\ul{R}\sotimes M\to M$ as what is written in the statement. 
\end{proof}

Taking some Mackey functor $\ul{M}$ arising from an abelian group $M$ with involution and then adding $\oplus \Z/2$ to $M(G/G)$ produces a Mackey functor which does not come from an abelian group (=  $\Z$-module) with involution, but by Lemma \ref{lem:iotaR-modules} still yields a $\ul{\Z}$-module, so that there are strictly more $\ul{\Z}$-modules than abelian groups with involution.
By Proposition \ref{inverting 2}, this is essentially the only possible difference.
Also recall \cite[Lemma 2.2.3]{HP23} that morphisms on the underlying classical modules uniquely extend to morphisms of modules over Green functors.

Recall \cite[Lemma 2.2.3]{HP23} that the functor $u$ from the category of $R$-modules with involutions to the category of $\ul{R}$-modules given by $M\mapsto \ul{M}$ is right adjoint to the evaluation functor $(C_2/e)$.
We discuss some subtleties for this adjunction at the end of Remark \ref{omegaone}.
Also note that the image of $u$ only contains objects which are fixed under the projector $\cP^0$,
which is the endofunctor of $\Mack_{\Z/2}$ killing the kernel of the restriction, and which also appears in Proposition \ref{slicesexplicit} below.
If we invert $2$, everything becomes much easier:

\begin{prop}
\label{inverting 2}
Let $R$ be a commutative ring such that $2$ is invertible in $R$.
Then the functor $u$ above is an equivalence of categories.
\end{prop}
\begin{proof}
Let $M$ be an $\ul{R}$-module.
We have $\res\circ \tr=1+w$,
where $w$ is the involution on $M(C_2/e)$.
By Lemma \ref{lem:iotaR-modules},
we also have $\tr\circ\res=2$.
Let $M'$ be the submodule of $M(C_2/e)$ consisting of the elements $a$ satisfying $w(a)=a$.
When we restrict $\res$ and $\tr$ to $M'$,
we have $\res\circ\tr=2$ and $\tr\circ \res=2$.
Together with the assumption that $2$ is invertible in $R$,
we see that $\res\colon M(C_2/C_2)\to M(C_2/e)$ is injective and its image is precisely $M'$.
This description means $M\simeq\ul{L}$, where $L$ is the $R$-module $M(C_2/e)$ with the involution $w$.
Hence $u$ is essentially surjective.
Using \cite[Lemma 2.2.3]{HP23},
we deduce that $u$ is fully faithful.
\end{proof}

Let $R$ be a commutative ring.
We refer to \cite[\href{https://stacks.math.columbia.edu/tag/091N}{Tag 091N}]{stacks} and the references given there for various equivalent descriptions of derived completion.
The \emph{derived $p$-completion of $\cF\in \rD(R)$} is
\[
\cF_p^\wedge:=\R\lim_n \cF\otimes_{R}^\L (R/p^nR),
\]
where $R/p^n R$ means $\cofib(p^n\colon R\to R)$ taken in $\rD(R)$.
We often omit the $\R$-decoration in $\R\lim$ when it is clear from the context that we are dealing with a derived limit and a derived tensor product.
If $\cF$ is concentrated in degree $0$,
then the derived $p$-completion of $\cF$ is different from the classical $p$-completion $\lim_n \cF\otimes_R (R/p^n R)$ in general.

We now study derived p-completion in more detail, following \cite{SAG}.

\begin{df}\label{def:completion}
Let $A$ be a Green functor.
The category $\Mod_{A}$ of $A$-modules is $\Z$-linear, i.e., an additive category.
Hence the stable $\infty$-category $\rD(A)$ is $\Z$-linear in the sense of \cite[Definition D.1.2.1]{SAG}.
We set
\[
A[1/p]
:=
\colim(A\xrightarrow{p} A\xrightarrow{p}\cdots)
\]
in $\rD(A)$,
which is concentrated in degree $0$ and its $\ul{H}_0$ has a natural Green functor structure.
Following \cite[Definitions 7.1.1.1, 7.2.4.1, 7.3.1.1]{SAG},
we say that $\cF\in \rD(A)$ is
\begin{enumerate}
\item[(1)] \emph{derived $p$-nilpotent} if $\cF\sotimes_{A}^\L A[1/p]\simeq 0$,
\item[(2)] \emph{derived $p$-local} if $\Map_{\rD(A)}(\cG,\cF)$ is contractible for every derived $p$-nilpotent object $\cG$ of $\rD(A)$,
\item[(3)] \emph{derived $p$-complete} if $\Map_{\rD(A)}(\cG,\cF)$ is contractible for every derived $p$-local object $\cG$ of $\rD(A)$.
\end{enumerate}
Let $\rD_{p-\comp}(A)$ denote the full subcategory of $\rD(A)$ spanned by the derived $p$-complete objects.
\end{df}

We have the localization functor
\[
(-)_p^\wedge
\colon
\rD(A)
\to
\rD_{p-\comp}(A)
\]
which by definition is left adjoint to the inclusion functor
$\rD_{p-\comp}(A)\to \rD(A)$,
and hence $(-)_p^\wedge$ preserves colimits.
Assume that $A=\ul{R}$ for some commutative ring $R$.
By \cite[Corollary 7.3.2.2]{SAG},
$\cF\in \rD(\ul{R})$ is derived $p$-complete if it is a local object in the sense of \cite[Definition 5.5.4.1]{HTT} with respect to the class of morphisms $\ul{R[1/p]}[n]\to 0$ and $\ul{R[1/p]^{\oplus \Z/2}}[n]\to 0$ for all integers $n$.

\begin{prop}
\label{completioncommuteswith}
Let $R$ be a commutative ring.
Then there are natural equivalences
\[
(\iota \cF)_p^\wedge
\simeq
\iota (\cF_p^\wedge),
\;
(i_\sharp \cF)_p^\wedge
\simeq
i_\sharp (\cF_p^\wedge),
\text{ and }
(i^*\cG)_p^\wedge
\simeq
i^*(\cG_p^\wedge)
\]
for $\cF\in \rD(R)$ and $\cG\in \rD(\ul{R})$.
\end{prop}
\begin{proof}
Observe that $\iota$ sends $R[1/p]$ to $\ul{R}[1/p]$,
$i_\sharp$ sends $R[1/p]$ to $\ul{R^{\oplus \Z/2}}[1/p]$, and $i^*$ sends $\ul{R}[1/p]$ to $R[1/p]$.
Using the universal property of the localization \cite[Proposition 5.5.4.20]{HTT},
we obtain the desired equivalences.
\end{proof}

For a Green functor $A$ and integer $m\geq 1$,
the notation $A/mA$ means the cofiber of the multiplication morphism $m\colon A\to A$ taken in $\rD(A)$.

\begin{prop}
\label{completion results}
Let $A$ be a Green functor,
and let $\cF$ be an object of $\rD(A)$.
\begin{enumerate}
\item[\textup{(1)}]
We have equivalences
\[
\cF_p^\wedge
\simeq
\cofib((\lim(\cdots \xrightarrow{p} \cF\xrightarrow{p} \cF)\to \cF)
\simeq
\lim_n
\cF \sotimes_{A}^\L A/p^n A
\]
in $\rD(A)$.
Hence the derived $p$-completion $\cF_p^\wedge$ is the pointwise derived $p$-completion if $A=\ul{R}$ for some commutative ring $R$.
\item[\textup{(2)}]
$\fib(\cF\to \cF_p^\wedge)$ is derived $p$-local.
\item[\textup{(3)}]
$\cF$ is derived $p$-local if and only if $\cF_p^\wedge\simeq 0$.
\item[\textup{(4)}]
$\cF$ is derived $p$-complete if and only if the induced morphism $\cF\to \cF_p^\wedge$ in $\rD(A)$ is an equivalence.
\item[\textup{(5)}]
If $\cF$ is derived $p$-local,
then $\cF[m+n\sigma]$ is derived $p$-local for all integers $m$ and $n$.
\item[\textup{(6)}]
If $\cF$ is derived $p$-local and $m$ is an integer coprime to $p$,
then the multiplication morphism $m\colon \cF\to \cF$ is an equivalence.
\item[\textup{(7)}]
If $A=\ul{R}$ for some commutative ring $R$,
then $\cF$ of $\rD(\ul{R})$ is derived $p$-complete if and only if $i^*\cF$ and $\cF^{\Z/2}$ in $\rD(R)$ are derived $p$-complete.
\end{enumerate}
\end{prop}
\begin{proof}
The first equivalence in (1) is a consequence of \cite[Proposition 7.3.2.1]{SAG}, and the second uses that lim and cofib commute in the stable setting.
(2) and (3) are consequences of \cite[Remark 7.2.0.2, Proposition 7.3.1.4]{SAG}.
(4) holds since $(-)_p^\wedge$ is a localization functor in the sense of \cite[Definition 5.2.7.2]{HTT}.
(5) is a consequence of (1) and (3).

(6) For every integer $n\geq 1$,
using that $m$ and $p^n$ are coprime, we see that the multiplication $m\colon A/p^n A\to A/p^n A$ in $\rD(A)$ is an equivalence.
Take $\lim_n(\cF\sotimes_A^\L (-))$ and use (1) and (4) to deduce the claim.

(7) As noted above,
$\cF$ is derived $p$-complete if and only if
\[
\Hom_{\rD(\ul{R})}
(\ul{R[1/p]^{\oplus \Z/2}}[n],\cF) = 0 =
\Hom_{\rD(\ul{R})}
 (\ul{R[1/p]}[n],\cF)
\]
for every integer $n$.
We conclude using $i_\sharp (R[1/p])\simeq \ul{R[1/p]^{\oplus \Z/2}}$ and $\iota (R[1/p])\simeq \ul{R[1/p]}$.
\end{proof}

\begin{prop}
\label{locality and monoidal}
Let $A$ be a Green functor.
If $\cF$ is a derived $p$-local object of $\rD(A)$,
then $\cF\sotimes_{A}^\L \cG$ is derived $p$-local for $\cG\in \rD(A)$.
\end{prop}
\begin{proof}
The class $\cS$ of those objects $\cG$ for which $\cF\sotimes_{A}^\L \cG$ is
derived $p$-local is closed under colimits by \cite[Proposition 7.2.4.9(2)]{SAG} and under equivariant shifts
$[m+n\sigma]$ for all integers $m$ and $n$ by Proposition \ref{completion results}(5).
Furthermore,
$\cS$ contains $A$.
Using Remark \ref{another generators}, it follows that $\cS$ contains every object of $\rD(A)$.
\end{proof}

\begin{prop}
\label{monoidal completion local}
Let $A$ be a Green functor,
and let $\cF\to \cF'$ be a morphism in $\rD(A)$.
If the induced morphism $\cF_p^\wedge \to \cF_p'^\wedge$ in $\rD(A)$ is an equivalence,
then the induced morphism
\[
(\cF\sotimes_{A}^\L \cG)_p^\wedge
\to
(\cF'\sotimes_{A}^\L \cG)_p^\wedge
\]
in $\rD(A)$ is an equivalence for $\cG\in \rD(A)$.
\end{prop}
\begin{proof}
Using Proposition \ref{completion results}(2), we see that the fiber of $\cF\to \cF'$ is derived $p$-local.
Hence the fiber of
\(
\cF\sotimes_{A}^\L \cG
\to
\cF'\sotimes_{A}^\L \cG
\)
is derived $p$-local by Proposition \ref{locality and monoidal}.
Proposition \ref{completion results}(3) finishes the proof.
\end{proof}

\begin{prop}
Let $A$ be a Green functor.
Then there exists a unique 
symmetric monoidal structure on $\rD_{p-\comp}(A)$ such that $(-)_p^\wedge$ is monoidal,
and the symmetric monoidal product is given by $(M,N)\mapsto (M\sotimes_A^\L N)_p^\wedge$ for all $M,N\in \rD_{p-\comp}(A)$.
\end{prop}
\begin{proof}
As in \cite[Corollary 7.3.5.2]{SAG},
this is a consequence of \cite[Proposition 2.2.1.9]{HA} and Proposition \ref{monoidal completion local}.
\end{proof}
The next result will be useful when proving Theorem \ref{cofibersequence}.

\begin{prop}
\label{monoidal completion algebra}
Let $A$ be a Green functor,
and let $B$ be a commutative algebra object of $\rD(A)$.
For $B$-modules $M$ and $N$,
the induced morphism
\[
(M\sotimes_B^\L N)_p^\wedge
\to
(M_p^\wedge \sotimes_{B_p^\wedge }^\L N_p^\wedge)_p^\wedge
\]
in $\Mod_B$ is an equivalence.
\end{prop}
\begin{proof}
The description of the monoidal structure on $\infty$-categories of modules in \cite[section 4.4.1, Theorem 4.5.2.1(2)]{HA} yields a natural equivalence
\[
M\sotimes_B^\L N
\simeq
\colim
\big(
\cdots
\substack{\rightarrow\\[-1em] \rightarrow \\[-1em] \rightarrow}
M\sotimes_{A}^\L B \sotimes_{A}^\L N
\substack{\rightarrow\\[-1em] \rightarrow}
M\sotimes_{A}^\L N
\big)
\]
in $\Mod_B$,
where the colimit is obtained by the two-sided bar construction.
Take $(-)_p^\wedge$ on both sides,
permute $\colim$ and $(-)_p^\wedge$,
and use Proposition \ref{monoidal completion local} to obtain a natural equivalence
\[
(M\sotimes_B^\L N)_p^\wedge
\simeq
\big(
\colim
\big(
\cdots
\substack{\rightarrow\\[-1em] \rightarrow \\[-1em] \rightarrow}
M_p^\wedge \sotimes_{A}^\L B_p^\wedge \sotimes_{A}^\L N_p^\wedge
\substack{\rightarrow\\[-1em] \rightarrow}
M_p^\wedge \sotimes_{A}^\L N_p^\wedge
\big)
\big)_p^\wedge
\]
in $\Mod_B$.
The latter one is equivalent to $(M_p^\wedge \sotimes_{B_p^\wedge }^\L N_p^\wedge)_p^\wedge$.
\end{proof}

\section{Recollection on slices in stable equivariant homotopy theory}
\label{sec3}

Hill-Hopkins-Ravenel \cite[section 4]{HHR} introduces a slice filtration.
For example, they defined when $X\in \Sp^{\Z/2}$ is slice $(-1)$-positive, written $X\geq 0$.
Ullman \cite{Ull13} suggests a slightly modified version,
also considered in \cite{HHR21}, which we now briefly recall. 

\begin{rmk}\label{compareslices}
Both \cite{HHR} and \cite{HHR21} yield the same condition for $X\geq 0$. Comparing \cite[4.48]{HHR} and \cite[11.1.18]{HHR21}, we see that $X\geq 1$ in \cite{HHR} is a weaker condition than the one in \cite{HHR21}. We also see that they have different $0$-slice by comparing \cite[4.50]{HHR} with \cite[11.1.45]{HHR21}. Hence the condition $X\leq 0$ differs as well, the condition of \cite{HHR} is stronger than the one in \cite[11.1.18 (iv)]{HHR21}.
\end{rmk}

For every integer $n$,
let $\Sp_{\geq n}^{\Z/2}$ be the localizing (in the sense of \cite[Definition 6.3.12]{HHR21}, in particular not triangulated) subcategory of $\Sp^{\Z/2}$ generated by the set
\[
\cS_n
:=
\{
\Sigma^{m+m\sigma} \Sphere
:
2m\geq n
\}
\cup
\{
\Sigma^{m} (\Sphere^{\oplus \Z/2})
:
m\geq n
\},
\]
where the involution on $\Sphere^{\oplus \Z/2}:=i_\sharp \Sphere$ permutes the two coordinates of $\Sphere^{\oplus \Z/2}$ (on $(G/e)$ in the Mackey functor description).
The inclusion $\Sp_{\geq n}^{\Z/2}\to \Sp^{\Z/2}$ admits a right adjoint $P_n\colon \Sp^{\Z/2}\to \Sp_{\geq n}^{\Z/2}$,
and we still write $P_n$ for the endofunctor on $\Sp^{\Z/2}$.
We also obtain the natural transformations
\[
\cdots \to P_n \to P_{n-1}\to \cdots,
\]
which we call the slice filtration, and which by definition has
layers $P_n^n:=\cofib(P_{n+1}\to P_n)$.
See e.g.\ \cite[section 11.1.E]{HHR21} for details.
We set $P^{n}:=\cofib(P_{n+1} \to \id)$.
Let $\Sp_{\leq n}^{\Z/2}$ be the full subcategory of $\Sp^{\Z/2}$ spanned by the objects $X$ such that $\Hom_{\Sp^{\Z/2}}(Y,X)=0$ for every $Y\in \cS_n$.

For a $\Z/2$-spectrum $X$ and integers $m$ and $n$,
we set
\begin{gather*}
\pi_{m+n\sigma}(X)
:=
\pi_0((\Sigma^{-m-n\sigma}X)^{\Z/2}),
\\
\ul{\pi}_{m+n\sigma}(X)
:=
\ul{\pi}_m(\Sigma^{-n\sigma} X)
\cong
\ul{\pi}_0(\Sigma^{-m-n\sigma} X)
\in
\Mack_{\Z/2},
\end{gather*}
compare also Definition \ref{defshiftabelian}.
Here, recall that $(-)^{\Z/2}$ denotes the fixed-point functor $\Sp^{\Z/2}\to \Sp$.
The standard notation $\ul{\pi}_n$ appears in e.g.\ \cite[section 3]{HHR}.
Note that $\Sp^{\Z/2}$ can be equivalently defined by spectral Mackey functors \cite{Bar17}, and the equivariant homotopy groups obtained this way agree with the above ones.

The following is the key computational result for the equivariant slices.

\begin{prop}\label{slicesexplicit}
Let $X$ be a $\Z/2$-spectrum.
Then we have equivalences
\begin{gather*}
P_{2n}^{2n}(X)
\simeq
\Sigma^{n+n\sigma} \rH \ul{\pi}_{n+n\sigma}(X),
\\
P_{2n+1}^{2n+1}(X)
\simeq
\Sigma^{n+1+n\sigma}  \rH \cP^0 \ul{\pi}_{n+1+n\sigma}(X)
\end{gather*}
for all integers $n$,
where $\cP^0$
is the endofunctor of $\Mack_{\Z/2}$ killing the kernel of the restriction.
\end{prop}
\begin{proof}
We refer to \cite[Theorem 17.5.25]{Hil20}.
\end{proof}

For any $\Z/2$-spectrum $X$ and integer $n$,
we set
\begin{gather*}
\rho_{2n}(X)
:=
\ul{\pi}_{n+n\sigma}(X),
\\
\rho_{2n+1}(X)
:=
\cP^0 \ul{\pi}_{n+1+n\sigma}(X)
\end{gather*}
so that we have $P_{2n}^{2n}(X)\simeq \Sigma^{n+n\sigma}\rH \rho_{2n}(X)$ and $P_{2n+1}^{2n+1}(X)\simeq \Sigma^{n+1+n\sigma} \rH \rho_{2n+1}(X)$.

We now introduce a similar filtration on the derived $\infty$-category $\rD(\ul{A})$ of $\ul{A}$-modules for a commutative ring $A$,
which under Proposition \ref{greenequivalence} is easily seen to correspond to the slice filtration.

\begin{df}
\label{abelian rho}
Let $A$ be a Green functor, e.g.\ $A=\ul{R}$ for a commutative ring $R$ with or without involution.
For every integer $n$,
let $\rD_{\geq n}(A)$ be the smallest full subcategory of $\rD_{\geq n}(A)$ closed under colimits and extensions and containing
\[
\cS_n(A)
:=
\{
A[m+m\sigma]
:
2m\geq n
\}
\cup
\{
A\sotimes \ul{\cB}^{\Z/2}[m]
:
m\geq n
\}.
\]
Let $\rD_{\leq n-1}(A)$ be the full subcategory of $\rD(A)$ spanned by the objects $\cF$ such that $\Hom_{\rD(A)}(\cG,\cF)=0$ for every $\cG\in \cS_n(A)$.
Observe that we have $\cF\in \rD_{\leq n-1}(A)$ if and only if $\Hom_{\rD(A)}(\cG,\cF)=0$ for every $\cG\in \rD_{\geq n}(A)$.

The inclusion $\rD_{\geq n}(A)\to \rD(A)$ admits a right adjoint $P_n\colon \rD(A)\to \rD_{\geq n}(A)$, and again we also write $P_n$ for the endofunctor on $\rD(A)$.
We also obtain the natural transformations
\[
\cdots \to P_n \to P_{n-1}\to \cdots,
\]
which we call the slice filtration.
We set $P^n:=\cofib(P_{n+1} \to \id)$ and $P_n^n:=\cofib(P_{n+1}\to P_n)$ for every integer $n$.
Observe that there is an equivalence $P_n^n\simeq \fib(P^n\to P^{n-1})$, 
and that $P^n(\cF) \in \rD_{\leq n}(A)$.
\end{df}

\begin{lem}
\label{shift 2}
Let $A$ be a Green functor.
For every integer $n$,
there are natural equivalences of $\infty$-categories
\[
\rD_{\geq n+2}(A)\simeq \rD_{\geq n}(A)[1+\sigma]
\text{ and }
\rD_{\leq  n+2}(A)\simeq \rD_{\leq n}(A)[1+\sigma].
\]
This notation means the full subcategory spanned by all the objects of this shifted form.
\end{lem}
\begin{proof}
We have equivalences
\[
\Sigma^\sigma i_\sharp \Sphere
\simeq
\Sigma^\sigma \Sphere \wedge i_\sharp \Sphere
\simeq
i_\sharp(i^* \Sigma^\sigma \Sphere \wedge \Sphere)
\simeq
\Sigma^1 i_\sharp \Sphere,
\]
where the second equivalence is due to \cite[Proposition A.1.10]{HP23}.
It follows that we have $\cS_{n+2}(A)=\cS_n(A)[1+\sigma]$,
whence the result follows.
\end{proof}

As for $\Sp^{\Z/2}$ above,
we define group valued functors $H^{\Z/2}_{m+n\sigma}(-)$ and Mackey functor valued functors $\ul{H}_{m+n\sigma}(-)$ for integers $m$ and $n$.

\begin{lem}
\label{abelian filtration}
Let $A$ be a Green functor.
For $\cF\in \rD(A)$,
we have the following properties.
\begin{enumerate}
\item[\textup{(1)}]
$\cF\in \rD_{\geq 0}(A)$ if and only if $\ul{H}_n(\cF)=0$ for every integer $n<0$.
\item[\textup{(2)}]
$\cF\in \rD_{\geq 1}(A)$ if and only if $\ul{H}_n(\cF)=0$ for every integer $n<1$.
\item[\textup{(3)}]
$\cF\in \rD_{\leq -1}(A)$ if and only if $\ul{H}_n(\cF)=0$ for every integer $n>-1$.
\item[\textup{(4)}]
$\cF\in \rD_{\leq 0}(A)$ if and only if $\ul{H}_n(\cF)=0$ for every integer $n>0$.
\end{enumerate}
\end{lem}
\begin{proof}
Argue as in \cite[Proposition 11.1.18]{HHR21}.
\end{proof}

\begin{prop}
\label{abelian slicesexplicit}
Let $A$ be a commutative ring, possibly with involution.
For $\cF\in \rD(\ul{A})$ and integer $n$,
there are canonical equivalences in $\rD(\ul{A})$
\begin{gather*}
P_{2n}^{2n}(\cF)
\simeq
(\ul{H}_{n+n\sigma}(\cF))[n+n\sigma],
\\
P_{2n+1}^{2n+1}(\cF)
\simeq
(\cP^0 \ul{H}_{n+1+n\sigma}(\cF))[n+1+n\sigma],
\end{gather*}
where $\cP^0$ is the endofunctor of $\Mod_{\ul{A}}$ killing the kernel of the restriction.
\end{prop}
\begin{proof}
Argue as in \cite[Theorem 17.5.25]{Hil20}.
If $A$ has a non-trivial involution, note that the functor $\cP^0$ exists as $\res$ is $A$-linear.
\end{proof}

For example, with the above notation,
we have
\begin{equation}
\label{b}
\ul{A}[n+n\sigma]
\simeq
P_{2n}^{2n}(\ul{A}[n+n\sigma])
\in \rD_{\geq 2n}(\ul{A})\cap \rD_{\leq 2n}(\ul{A}).
\end{equation}

\begin{df}
\label{rho definition}
Let $A$ be a commutative ring.
For $\cF\in \rD(\ul{A})$ and integer $n$,
we set
\begin{gather*}
\rho_{2n}(\cF)
:=
\ul{H}_{n+n\sigma}(\cF),
\\
\rho_{2n+1}(\cF)
:=
\cP^0 \ul{H}_{n+1+n\sigma}(\cF).
\end{gather*}
Observe that we have a natural isomorphism of Mackey functors $\rho_m(\cF)\cong \rho_m(\alpha_*\cF)$ for every integer $m$,
where $\alpha_*\colon \rD(\ul{A})\to \Sp^{\Z/2}$ is right adjoint to the symmetric monoidal functor $\alpha^*\colon \Sp^{\Z/2}\to \rD(\ul{A})$ in Construction \ref{alphaupperstar}.
\end{df}

For a Green functor $A$ and integers $m$, $n$, and $l$ with $n,l\geq 0$,
note that we have an inclusion
\begin{equation}
\label{a}
\rD_{\geq m}(A)[n+l\sigma]\subset \rD_{\geq m}(A).
\end{equation}

\begin{prop}
\label{conservative slices}
Let $A$ be a Green functor.
If $\cF\in \rD(A)$ satisfies $P_n^n(\cF)\simeq 0$ for every integer $n$,
then $\cF\simeq 0$ in $\rD(A)$.
\end{prop}
\begin{proof}
We have the fiber sequence $P_{n+1}(\cF)\to P_n(\cF)\to P_n^n(\cF)$ in $\rD(A)$.
Together with the assumption $P_n^n(\cF)\simeq 0$,
we have $P_n(\cF)\simeq P_{n+1}(\cF)$ for every integer $n$.
In particular, $P_1(\cF)\in \rD_{\geq n}(A)$ for every integer $n$.
Lemma \ref{shift 2} implies that
$P_1(\cF)\in \rD_{\geq 2m}(A)=\rD_{\geq 0}(A)[m+m\sigma]$ for every integer $m$.
Since by \eqref{a}, we have $\rD_{\geq 0}(A)[m+m\sigma]\subset  \rD_{\geq 0}(A)[m]$ if $m\geq 0$,
we obtain $P_1(\cF)\simeq 0$ using Lemma \ref{abelian filtration}(1).
A similar argument shows that $P^0(\cF)\simeq 0$.
Use the fiber sequence $P_1(\cF)\to \cF\to P^0(\cF)$ to finish the proof.
\end{proof}

\begin{df}\label{def:even}
Let $X$ be a $\Z/2$-spectrum.
We say that $X$ is \emph{even} if $\rho_{2n+1}(X)=0$ for every integer $n$.

We say that $X$ is \emph{very even} if $X$ is even and $\rho_{2n}(X)$ is a constant Mackey functor for every integer $n$.
In this case,
we have isomorphisms
\begin{align*}
\rho_{2n}(X)(C_2/e)
= &
\Hom_{\Sp^{\Z/2}}(\Sigma^\infty (C_2/e)_+,\Sigma^{-n-n\sigma}X)
\\
\cong &
\Hom_{\Sp^{\Z/2}}(i_\sharp \Sphere,\Sigma^{-n-n\sigma}X)
\\
\cong &
\Hom_{\Sp}(\Sphere,i^*\Sigma^{-n-n\sigma}X)
\\
\cong &
\pi_{2n}(i^*X).
\end{align*}
Hence we have a natural isomorphism
\begin{equation}
\label{eq:very even}
\rho_{2n}(X)\cong \ul{\pi_{2n}(i^* X)}.
\end{equation}

For a commutative ring $R$,
we use a similar terminology for objects of $\rD(\ul{R})$.
\end{df}

\begin{prop}\label{prop:even}
Let $X$ be a $\Z/2$-spectrum.
Then $X$ is even if and only if $i^*X$ is even in the sense that $\pi_{2n+1}(i^*X)=0$ for every integer $n$.
\end{prop}
\begin{proof}
If $\pi_{2n+1}(i^*X)=0$,
then $\ul{\pi}_{n+1+n\sigma}(X)(C_2/e)=0$.
This implies that the map $\res$ is $0$ for $\rho_{2n+1}(X)$,
so we have
$\cP^0\ul{\pi}_{n+1+n\sigma}(X)=0$.
On the other hand,
$\cP^0\ul{\pi}_{n+1+n\sigma}(X)=0$ implies $\pi_{2n+1}(i^*X)=0$.
Together with the definition $\rho_{2n+1}(X)=\cP^0\ul{\pi}_{n+1+n\sigma}(X)$,
we finish the proof.
\end{proof}

\section{Real Hochschild homology and real topological Hochschild homology}
\label{sec4}

Let $\NAlg^{\Z/2}$ denote the $\infty$-category of normed $\Z/2$-spectra.
This can be defined as the underlying $\infty$-category of a certain model category \cite[Proposition B.129]{HHR} of commutative monoids in orthogonal $\Z/2$-spectra.
Alternatively,
Bachmann-Hoyois provide a purely $\infty$-categorical formulation,
see \cite[Definition 9.14]{BH21}.
We refer to \cite[Remark A.2.3]{HP23} for a review.
A crucial property of $\NAlg^{\Z/2}$ due to Hill-Hopkins-Ravenel \cite[Proposition 2.27]{HHR} is that there exists a symmetric monoidal functor $N^{\Z/2}\colon \CAlg(\Sp) \to \NAlg^{\Z/2}$ that is left adjoint to the forgetful functor $i^*\colon \NAlg^{\Z/2}\to \CAlg$.

There are various equivalent definitions of $\THR$.
We continue to use the following form in \cite[Definition 2.1.1]{HP23} as a definition of $\THR$, which is equivalent to the original ``B\"okstedt model'' definition of $\THR$ due to Hasselholt-Madsen \cite{HM} under a mild flatness condition, see \cite[Theorem, p.\ 65]{DMPR21}.

\begin{df}\label{def:thrandhr}
For a normed $\Z/2$-spectrum $A$,
the \emph{real topological Hochschild homology of $A$} is the pushout
\[
\THR(A)
:=
A\wedge_{N^{\Z/2} i^* A} A
\]
in $\NAlg^{\Z/2}$,
where both morphisms $N^{\Z/2}i^*A\to A$ are the counit map.
We may regard $\THR(A)$ as an $A$-algebra using the morphism $A\to A\wedge_{N^{\Z/2}i^*A}A$ to the second smash factor.
On the other hand,
we have the map $\THR(A)\to A$ given by $A\wedge_{N^{\Z/2}i^*A }A\to A\wedge_A A\simeq A$.
For a map of normed $\Z/2$-spectra $A\to B$,
the \emph{real Hochschild homology of $B$ over $A$} is the pushout
\[
\HR(B/A)
:=
\THR(B)\wedge_{\THR(A)}A
\]
in $\NAlg^{\Z/2}$.
We may also regard $\HR(B/A)$ as a $B$-algebra.
\end{df}

The above definition immediately implies an equivalence of normed $\Z/2$-spectra
\[
\HR(B/\Sphere)
\simeq
\THR(B).
\]

\begin{rmk}
For rings, forgetting the involution this is equivalent to the classical $\HH$ in the flat case thanks to \cite[Lemma 2.5]{BMS19},  keeping in mind the convention on pages 211/212 of loc.\ cit.
So we take the equivariant refinement of  \cite[Lemma 2.5]{BMS19} as our definition. Establishing an equivariant HKR theorem in the nontrivial involution case with this definition will be difficult, but see our partial results further below. Also, note that this definition of $\HR$ produces an equivariant ring spectrum, whereas the traditional $\HH(A/R)$ is a simplicial ring. 
However, using the stable Dold-Kan functor $\rH$ from  Proposition \ref{greenequivalence},
$\rH(\HH(A/R))$ is indeed equivalent to $i^*\HR(A/R)$,
see Proposition \ref{i* and THR}.
\end{rmk}

\begin{rmk}\label{shortnotation}
For a homomorphism of commutative rings $R\to A$ with involutions,
we often use the abbreviated notation
\[
\THR(A)
:=
\THR(\rH \ul{A})
\text{ and }
\HR(A/R)
:=
\HR(\rH \ul{A}/\rH \ul{R}).
\]
Let $\THR(A;\Z_p)$ and $\HR(A/R;\Z_p)$ be their derived $p$-completions in $\rD(\ul{A})$.
If $R=\Z$,
then we set $\HR(A):=\HR(A/\Z)$.

For morphisms of Green functors $R\to A$, we also use the abbreviated notation $\THR(A):=\THR(\rH A)$ and $\HR(A/R):=\HR(\rH A/\rH R)$ and their $p$-completions $\THR(A;\Z_p)$ and $\THR(A/R;\Z_p)$.
\end{rmk}

\begin{prop}
\label{prop:NAlg has colimits}
The $\infty$-category $\NAlg^{\Z/2}$ has colimits and limits,
and the forgetful functor $\NAlg^{\Z/2}\to \Sp^{\Z/2}$ is conservative and preserves limits and sifted colimits.
\end{prop}
\begin{proof}
An analogous result is proved in \cite[Proposition 7.6(1),(2)]{BH21} for the motivic setting. 
One can similarly argue for the topological setting too. 
\end{proof}

\begin{const}
\label{left Kan}
Let $R$ be a commutative ring,
and let $\Poly_R$ be the category of finitely generated polynomial $R$-algebras.
A cohomology theory on the $\infty$-category $\sCRing_R$ of simplicial
commutative $R$-algebras is often
extended from a cohomology theory on $\Poly_R$ as observed in \cite[Construction  2.1]{BMS19},
which we review as follows.
According to loc.\ cit,
there is a natural equivalence of categories between the category $\CRing_R$ of $R$-algebras and the category of functors $\Poly_R\to \Set$ sending coproducts to products.
By \cite[Corollary 5.5.9.3]{HTT},
this yields an equivalence of $\infty$-categories
$\sCRing_R \simeq \cP_{\Sigma}(\Poly_R)$.
Hence for every $\infty$-category $\cD$ with sifted colimits,
there exists an equivalence of $\infty$-categories
\begin{equation}
\Fun_\Sigma(\sCRing_R,\cD)
\xrightarrow{\simeq}
\Fun(\Poly_R,\cD)
\end{equation}
using \cite[Proposition 5.5.8.15]{HTT},
where $\Fun_\Sigma(\sCRing_R,\cD)$ denotes the full subcategory of $\Fun(\sCRing_R,\cD)$ spanned by the functors preserving sifted colimits. 
Hence for every functor $f\colon \Poly_R\to \cD$,
there exists a functor $F\colon \sCRing_R\to \cD$ preserving sifted colimits that is unique.
In this case,
we say that $F$ is \emph{the left Kan extension of $f$}.
\end{const}

\begin{const}\label{extendwithinvolutions}
The approach of \cite[Construction 2.1]{BMS19} of extending $\infty$-functors from polynomial rings to simplicial rings
can be adapted to the equivariant setting as follows:
Let $R$ be a commutative ring with involution,
and let $\Poly_R^{\Z/2}$ be the category of $R$-algebras of the form $R[\N^X]$ with a finite $\Z/2$-set $X$.
Here,
$\N^X$ is the commutative monoid of functions $X\to \N$ with the involution induced by the involution on $X$,
and $R[M]$ for a commutative monoid $M$ with involution $\sigma$ is the commutative monoid ring $\{\sum_{m\in M} r_m x^m : r_m \in R\}$ with involution obtained by $x^m\mapsto x^{\sigma(m)}$.
The functor $X\mapsto R[\N^X]$ is a left adjoint,
see Example \ref{exm:THR is a left Kan extension}.
Consider the $\infty$-category $\cP_{\Sigma}(\Poly_R^{\Z/2})$,
which is the full subcategory of the $\infty$-category of presheaves of spaces $\cP(\Poly_R^{\Z/2})$ spanned by the functors $(\Poly_R^{\Z/2})^{op}
\to \Spc$ preserving finite products.
By \cite[Corollary 5.5.9.3]{HTT},
$\cP(\Poly_R^{\Z/2})$ is the underlying $\infty$-category of the simplicial model category of the simplicial presheaves on $\Poly_R^{\Z/2}$ with the projective model structure.

Let $\CRing_{R}^{\Z/2}$ be the
category of commutative $R$-algebras with involution.
The inclusion functor $\Poly_R^{\Z/2}\to \CRing_R^{\Z/2}$ preserves coproducts.
This implies that for every $A\in \CRing_R^{\Z/2}$,
the presheaf on $\Poly_R^{\Z/2}$ represented by $A$ is in $\cP_{\Sigma}(\Poly_R^{\Z/2})$.
Hence the Yoneda embedding
\[
\Poly_R^{\Z/2}
\to
\cP_\Sigma(\Poly_R^{\Z/2})
\]
factors through $\CRing_R^{\Z/2}$.
(To proceed exactly as in \cite[Construction 2.1]{BMS19}, we would have to consider simplicial rings here and show that $\cP_\Sigma(\Poly_R^{\Z/2})$ is indeed equivalent to the underlying $\infty$-category of the simplicial commutative rings with involutions.
This would require an equivariant version of \cite[Example 5.1.3]{CS}.)

Let $\cD$ be an $\infty$-category with sifted colimits.
Then \cite[Proposition 5.5.8.15]{HTT} yields an induced equivalence of $\infty$-categories
\[
\Fun_\Sigma(\cP_{\Sigma}(\Poly_R^{\Z/2}),\cD)
\xrightarrow{\simeq}
\Fun(\Poly_R^{\Z/2},\cD),
\]
where $\Fun_\Sigma(\cP_{\Sigma}(\Poly_R^{\Z/2}),\cD)$ denotes the full subcategory of $\Fun(\cP_{\Sigma}(\Poly_R^{\Z/2}),\cD)$ spanned by the functors preserving sifted colimits.
This implies that any functor $f\colon \Poly_R^{\Z/2} \to \cD$ admits a unique extension $F\colon \cP_\Sigma(\Poly_R^{\Z/2})\to \cD$ in the $\infty$-sense 
 preserving sifted colimits.
In this case,
we say that $F$ is a \emph{left Kan extension of $f$}.
Compose $F$ with the above functor $\CRing_R^{\Z/2}\to \cP_\Sigma(\Poly_R^{\Z/2})$ to obtain a functor
\[
\CRing_R^{\Z/2}
\to
\cD,
\]
which is a natural extension of $f$.
\end{const}

\begin{rmk}
In most of the article below, we will not use
this equivariant construction, the following Example \ref{equiexample} and in particular
\eqref{eqn:THR for presheaves 2} below, but only need the non-equivariant version in Construction \ref{thrandhrforsimplicialrings}.
However, we keep this example since it will be useful for future work generalizing Theorem \ref{HRfiltration} to the nontrivial involution case, compare also Remark \ref{addinginvolutions}.
\end{rmk}

\begin{exm}\label{equiexample}
\label{exm:THR is a left Kan extension}
Let $R$ be a commutative ring with involution.
We have the adjunction
\[
F:
\Set^{\Z/2}
\rightleftarrows
\CRing_R^{\Z/2}
:
U,
\]
where $F$ sends a $\Z/2$-set $X$ to $R[\N^X]$, and $U$ is the forgetful functor.
Using $F$ and $U$,
one can form the \emph{standard resolution} \cite[\href{https://stacks.math.columbia.edu/tag/08N8}{Tag 08N8}]{stacks}
\begin{equation}
\label{eqn:standard resolution}
P_\bullet
\to
A
\end{equation}
for every $R$-algebra $A$ with involution with the terms $P_n:=(FU)^{(n+1)}(A)$ for all integers $n\geq 0$.
Observe that $P_\bullet$ is a simplicial $R$-module with involution.
Consider the Green functor $\ul{A}$ associated with $A$.
Apply $\ul{(-)}$ to \eqref{eqn:standard resolution} to obtain the standard resolution $\ul{P_\bullet}\to \ul{A}$,
and in this case $\ul{P_\bullet}$ is a simplicial $\ul{A}$-module.

Consider the functor
\begin{equation}
\label{eqn:Eilenberg-MacLane functor}
\rH
\colon
\CRing_R^{\Z/2}
\to
\NAlg^{\Z/2}
\end{equation}
sending a commutative $R$-algebra $A$ with involution to the Eilenberg-MacLane spectrum $\rH \ul{A}$, see e.g.\ \cite[Definition 2.3.2]{HP23}.
This restricts to the functor
\begin{equation}
\label{eqn:Eilenberg-MacLane functor 3}
\rH
\colon
\Poly_R^{\Z/2}
\to
\NAlg^{\Z/2},
\end{equation}
which admits a left Kan extension 
\begin{equation}
\label{eqn:Eilenberg-MacLane functor 2}
\rH
\colon
\cP_\Sigma(\Poly_R^{\Z/2})
\to
\NAlg^{\Z/2}
\end{equation}
as a special case of Construction \ref{extendwithinvolutions} for $f=\rH$ and $\cD=\NAlg^{\Z/2}$.
If $A\in \CRing_R^{\Z/2}$,
then consider the standard resolution $P_\bullet\to A$.
The induced morphism $\colim \rH(P_\bullet)\to \rH(A)$ in $\NAlg^{\Z/2}$ is an equivalence.
It follows that
\eqref{eqn:Eilenberg-MacLane functor 2},
which is originally an extension of \eqref{eqn:Eilenberg-MacLane functor 3},
is indeed an extension of \eqref{eqn:Eilenberg-MacLane functor}.

We have the functor
\begin{equation}
\label{eqn:THR for presheaves}
\THR
\colon
\cP_\Sigma(\Poly_R^{\Z/2})
\to
\NAlg^{\Z/2}
\end{equation}
given by $\THR(\cF):=\rH \cF\wedge_{N^{\Z/2}i^*\rH \cF}\rH \cF$ for $\cF\in 
\cP_\Sigma(\Poly_R^{\Z/2})$.
If $\cF$ is represented by a commutative $R$-algebra $A$ with involution,
then we have $\THR(\cF)\simeq \THR(A)$
by Definition \ref{def:thrandhr}.
Since \eqref{eqn:Eilenberg-MacLane functor 2} preserves sifted colimits and $\wedge$ is a colimit,
\eqref{eqn:THR for presheaves} preserves sifted colimits too.
It follows that \eqref{eqn:THR for presheaves} is the left Kan extension of the functor $\THR\colon \Poly_R^{\Z/2}\to \NAlg^{\Z/2}$.
Composing with the map $i^*$, one easily sees that we have $\THH(i^* \cF) \simeq i^*\THR(\cF)$ which generalizes the first statement of \cite[Proposition 2.1.3]{HP23}.
We compose \eqref{eqn:Eilenberg-MacLane functor 2} with the forgetful functor $\NAlg^{\Z/2}\to \Sp^{\Z/2}$ to obtain the functor
\begin{equation}
\label{eqn:THR for presheaves 2}
\THR\colon \cP_\Sigma(\Poly_R^{\Z/2})
\to
\Sp^{\Z/2}.
\end{equation}
This preserves sifted colimits by Proposition \ref{prop:NAlg has colimits}.
It follows that 
\eqref{eqn:THR for presheaves 2} is the left Kan extension of the functor $\THR\colon \Poly_R^{\Z/2}\to \Sp^{\Z/2}$ too.
\end{exm}

We now give several examples of left Kan extension obtained by Construction \ref{left Kan}.

\begin{exm}
\label{left Kan H}
Let $R$ be a commutative ring.
We have the functor
\[
\rH
\colon
\Poly_R
\to
\CAlg(\Sp)
\]
sending $A\in \Poly_R$ to its Eilenberg-MacLane spectrum $\rH A$. 
We have a
left Kan extension
\[
\rH
\colon
\sCRing_R
\to
\CAlg(\Sp).
\]
This preserves sifted colimits by Construction \ref{left Kan},
$\rH A$ is equivalent to the usual Eilenberg-MacLane spectrum for any commutative $R$-algebra $A$ arguing
as in Example \ref{exm:THR is a left Kan extension}.

We similarly
define a functor $\rH \colon \sCRing_R\to \NAlg^{\Z/2}$ that
is obtained as a left Kan extension such that
$\rH \ul{A}$ is the equivariant Eilenberg-MacLane spectrum for any commutative $R$-algebra $A$.
(As $A$ has trivial involution here, we don't need Construction \ref{extendwithinvolutions} for this.)
Note that this $\rH \colon \sCRing_R\to \NAlg^{\Z/2}$ is the restriction of \eqref{eqn:Eilenberg-MacLane functor 2}.
\end{exm}

\begin{exm}
\label{left Kan Omega}
Let $R$ be a commutative ring.
We have the functor
\[
\Omega_{/R}^1
\colon
\Poly_R \to \Mod(\rD(R))
\]
sending $A\in \Poly_R$ to the $A$-module $\Omega_{A/R}^1$,
where
$\Mod(\cC)$ denotes the underlying $\infty$-category of the generalized $\infty$-operad $\Mod(\cC)^{\otimes}$ in \cite[Definition 4.5.1.1]{HA} for every symmetric monoidal $\infty$-category $\cC$.
Recall from \cite[Remark 4.2.1.15, Corollary 4.5.1.6]{HA}
that an object of the $\infty$-category $\Mod(\rD(R))$ consists of a pair of a commutative algebra object
$B$ and a $B$-module.
By Proposition \ref{module sifted colimit} and \cite[Corollary 4.2.3.2]{HA}
$\Mod(\rD(R))$ admits sifted colimits, and the forgetful functor $\Mod(\rD(R))\to \CAlg(\rD(R))\times \rD(R)$ is conservative and preserves sifted colimits.
Let
\[
\L_{/R}
\colon
\sCRing_R \to \Mod(\rD(R))
\]
be a 
left Kan extension of $\Omega_{/R}^1$.
For an $R$-algebra $A$,
$\L_{A/R}$ is an $\rH A$-module and hence an object of $\rD(A)$ due to Example \ref{left Kan H}.
We wish to keep track of this $\rH A$-module structure,
and this is the reason why we work with $\Mod(\rD(R))$ rather than $\rD(R)$.
Together with the equivalence of $\infty$-categories $\Mod_A(\rD(R))\simeq \rD(A)$ (a non-equivariant version of Proposition \ref{abelian Dold-Kan}),
we may regard $\L_{A/R}$ as an object of $\rD(A)$.

We later might wish to extend the last example to commutative $R$-algebras with involution using Construction \ref{extendwithinvolutions}. This could be very useful below, see Remark \ref{addinginvolutions} below.
\end{exm}

\begin{const}\label{thrandhrforsimplicialrings}
Let $R$ be a commutative ring.
We have the functor
\begin{equation}
\label{eqn:thrandhrforsimplicialrings}
\THR
\colon
\sCRing_R
\to
\Mod(\Sp^{\Z/2})
\end{equation}
sending $A\in \sCRing_R$ to the $\rH \ul{A}$-module $\THR(\rH \ul{A})$,
where $\rH$ is the left Kan extension of the equivariant Eilenberg-MacLane spectrum functor in Example \ref{left Kan H}.
This preserves sifted colimits since $\rH$ preserves sifted colimits by Example \ref{left Kan H},
the coproduct $\wedge$ in $\NAlg^{\Z/2}$ preserves colimits,
the forgetful functor $\NAlg^{\Z/2}\to \Sp^{\Z/2}$ preserves sifted colimits by Proposition \ref{prop:NAlg has colimits},
and the functors $N^{\Z/2}\colon \CAlg(\Sp)\to \NAlg^{\Z/2}$ and $i^*\colon \NAlg^{\Z/2}\to \CAlg(\Sp)$ preserve colimits.
Arguing as in Construction \ref{left Kan},
we see that $\THR$ is the left Kan extension of its restriction
\[
\THR\colon \Poly_R\to \Mod(\Sp^{\Z/2}).
\]
Note that \eqref{eqn:THR for presheaves 2} is an extension of \eqref{eqn:thrandhrforsimplicialrings} after forgetting the module structure,
and we only need \eqref{eqn:thrandhrforsimplicialrings} for the later pats.

Similarly,
the functor
\[
\HR(-/\ul{R})
\colon
\sCRing_R
\to
\Mod(\rD(\ul{R}))
\]
preserves sifted colimits.
Hence $\HR(-/\ul{R})$ is the left Kan extension of its restriction
\[
\HR(-/\ul{R})
\colon
\Poly_R
\to
\Mod(\rD(\ul{R})).
\]
\end{const}

\begin{prop}
\label{i* and THR}
Let $R\to A$ be a homomorphism of commutative rings.
Then there are equivalences in $\rD(A)$
\[
i^*\THR(A)
\simeq
\THH(A)
\text{ and }
i^*\HR(A/R)
\simeq
\HH(A/R)
\]
which are natural in $D(R)$.
\end{prop}
\begin{proof}
Consider the forgetful conservative functors $\alpha_*\colon \rD(A)\to \Sp$ and $\alpha_*\colon \rD(\ul{A})\to \Sp^{\Z/2}$.
By Proposition \ref{alpha isharp},
we have a natural isomorphism $\alpha_* i^*\simeq i^*\alpha_*$.
Hence it suffices to provide equivalences
\[
i^*\alpha_*\THR(A)
\simeq
\alpha_*\THH(A)
\text{ and }
i^*\alpha_* \HR(A/R)
\simeq
\alpha_*\HH(A/R),
\]
i.e.,
we can forget the $\rH A$-module and $\rH \ul{A}$-module structures.
Then the first equivalence is a special case of \cite[Proposition 2.1.3]{HP23}, and the second equivalence holds as $i^*$ commutes with push-outs in $\NAlg^{\Z/2}$.
\end{proof}

\begin{prop}
\label{connective}
Let $B$ be a normed $\Z/2$-spectrum such that its underlying $\Z/2$-spectrum is $(-1)$-connected.
Then $\THR(B)$ is $(-1)$-connected.
Let $A$ be a commutative ring with involution.
Then $\THR(A;\Z_p)$ is $(-1)$-connected.
\end{prop}
\begin{proof}
By \cite[Proposition A.3.7]{HP23},
$N^{\Z/2} i^*B$ is $(-1)$-connected.
Together with \cite[Corollary 6.8.1]{BGS20},
we see that $\THR(B)$ is $(-1)$-connected.

We have the equivalence $\THR(A;\Z_p)\simeq \lim_n \THR(A)\sotimes_{\ul{A}}^\L (\ul{A}/p^n\ul{A})$.
Since the $\lim^1$ terms add at most $-1$ homological degree in the limit,
we see that $\THR(A;\Z_p)$ is $(-2)$-connected 
and (using \cite[Corollary 6.8.1]{BGS20} again) there is an isomorphism
\[
\ul{\pi}_{-1}\THR(A;\Z_p)
\cong
\mathop{\mathrm{lim}^1}_n \ul{\pi}_0\THR(A) \sotimes_{\ul{A}}\ul{\pi}_0 (\ul{A}/p^n \ul{A}).
\]
Using the description of Lemma \ref{lem:box base change}, we see that the tower on the right hand side consists of pointwise epimorphisms, hence satisfies the Mittag-Leffler condition in the abelian category of $\ul{\Z}$-modules.
We deduce that $\ul{\pi}_{-1}\THR(A;\Z_p)$ vanishes.
\end{proof}

\begin{lem}
\label{lem:finiteness of THR}
For every integer $n>0$,
$\ul{\pi}_n(\THR(\Z))$ is finite, i.e., $\ul{\pi}_n(\THR(\Z))(C_2/e)$ and $\ul{\pi}_n(\THR(\Z))(C_2/C_2)$ are finite.
\end{lem}
\begin{proof}
By the finiteness result \cite[Lemma 2.5]{BMS19},
it remains to show that $\pi_n(\THR(\Z)^{\Z/2})$ is finite.
Using the homotopy orbit spectral sequence, \cite[Lemma 2.5]{BMS19} implies that
$\pi_n(\THR(\Z)_{h \Z/2})$ is finite for $n>0$. 
By \cite[Theorem 5.20]{DMPR21},
$\pi_n(\Phi^{\Z/2} \THR(\Z))$ is finite.
The isotropy cofiber sequence
\[
\THR(\Z)_{h \Z/2}
\to
\THR(\Z)^{\Z/2}
\to
\Phi^{\Z/2} \THR(\Z)
\]
implies that $\pi_n(\THR(\Z)^{\Z/2})$ is finite too.
\end{proof}

\begin{rmk}
\label{DMPR conjecture}
This important finiteness result refines the last part of \cite[Lemma 2.5]{BMS19}.
It is also related to \cite[Theorem 5.24]{DMPR21}. In fact,
Dotto-Moi-Patchkoria-Reeh \cite[Remark 5.23]{DMPR21} conjectured that there is an equivalence of $\Z/2$-spectra
\[
\THR(\Z)
\simeq
\rH \ul{\Z} \oplus \bigoplus_{i=1}^\infty \Sigma^{i-1+i\sigma} \rH \ul{\Z/i}.
\]
Hahn and Wilson \cite[Theorem B.0.1]{AKF} have now proved this conjecture,
which also implies Lemma \ref{lem:finiteness of THR}.

As noted in \cite[Remark 1.5]{BMS19},
the construction of the Bhatt-Morrow-Scholze filtrations mostly relies on the ``formal'' aspect of algebraic topology, with the important exception of B\"okstedt's computation of $\pi_*\THH(\F_p)$.
In particular, they do not use B\"okstedt's computation of $\pi_*\THH(\Z)$ \cite{Boekstedt}.
In our case,
we rely on the computation of $\ul{\pi}_*\THR(\F_p)$ \cite[Theorem 5.15]{DMPR21} and the computation of $\pi_n(\Phi^{\Z/2}\THR(\Z))$  \cite[Theorem 5.20]{DMPR21} that uses the computation of $\pi_*\Phi^{\Z/2}\rH \Z$ and the K\"unneth spectral sequence.
We do not need the computation of $\ul{\pi}_*\THR(\Z)$.
\end{rmk}

\begin{const}
\label{const:finite Mackey}
Let $M$ be an $\ul{\Z}$-module such that $M$ is finite, i.e.,
$A:=M(C_2/e)$ and $B:=M(C_2/C_2)$ are finite abelian groups.
We have the morphism of $\ul{\Z}$-modules $\ul{B}\to M$ given by the diagram
\[
\begin{tikzcd}
B
\ar[r,"\id"]
\ar[d,"\id"',shift right=0.5ex]
\ar[d,"2",shift left=0.5ex,leftarrow]
&
B
\ar[d,"\res"',shift right=0.5ex]\ar[d,"\tr",shift left=0.5ex,leftarrow]
\\
B
\ar[r,"\res"]
\ar[loop,out=190, in=170,looseness=9]
&
A
\ar[loop,out=350, in=10,looseness=9]
\end{tikzcd}
\]
since $\tr\circ \res=2$ for $M$ by Lemma \ref{lem:iotaR-modules}
and the involution on $B$ is trivial by definition.
Both the kernel and cokernel of $\ul{B}\to M$ have the form
\[
\begin{tikzcd}
C\ar[loop below,"w"]\ar[r,shift left=0.5ex,"\tr"]\ar[r,shift right=0.5ex,leftarrow,"\res"']&
0
\end{tikzcd}
\]
for some $\Z$-module $C$.
The identity $1+w=\res\circ \tr$ implies $w=-1$.
(Note that in the case of the kernel, we have $w=1$, and $C$ is $2$-torsion.)
Hence both the kernel and cokernel Mackey functors are the cokernel of the morphism of $\ul{\Z}$-modules $\ul{C}\to \ul{C^{\oplus \Z/2}}$ given by the diagram
\[
\begin{tikzcd}
C
\ar[r,"\id"]
\ar[d,"\id"',shift right=0.5ex]
\ar[d,"2",shift left=0.5ex,leftarrow]
&
C
\ar[d,"\res"',shift right=0.5ex]\ar[d,"\tr",shift left=0.5ex,leftarrow]
\\
C
\ar[r,"\res"]
\ar[loop,out=190, in=170,looseness=9]
&
C\oplus C,
\ar[loop,out=355, in=5,looseness=7]
\end{tikzcd}
\]
where $\res$ is the diagonal embedding,
the involution on $C$ is trivial,
and the involution on $C\oplus C$ changes the summands.
Hence there exists a finite number of finite abelian groups $L_1,\ldots,L_n$ such that we can also construct $M$ from $\ul{L_i}$ and $\ul{L_i^{\oplus \Z/2}}$ using kernels of surjections, cokernels of injections, and extensions finitely many times.
It follows that we can construct $M$ from $\ul{\Z}$ and $\ul{\Z^{\oplus \Z/2}}$ using the above operations finitely many times.
This implies that we can construct $\rH M$ from $\rH \ul{\Z}$ and $\rH \ul{\Z}[\sigma]$ using fibers, cofibers, and extensions only  finitely many times.
\end{const}

\begin{lem}
\label{finite derived tensor}
Let $L$ and $M$ be $\ul{\Z}$-modules.
If $L$ and $M$ are finite,
then the $\ul{\Z}$-module $\ul{H}_n(L\sotimes_{\ul{\Z}}^\L M)$ is finite for every integer $n$.
\end{lem}
\begin{proof}
Using Construction \ref{const:finite Mackey},
it suffices to show that
\[
\ul{H}_n(L\sotimes_{\ul{\Z}}^\L \ul{\Z})
\text{ and }
\ul{H}_n(L\sotimes_{\ul{\Z}}^\L \ul{\Z^{\oplus \Z/2}})
\]
are finite for every integer $n$.
The assumption that $L$ is finite implies this.
\end{proof}

\begin{df}\label{def:filtration}
Let $\Z^{op}$ denote the category whose object set is $\Z$ and whose morphism $m\to n$ in $\Z$ is uniquely given if and only if $m\geq n$.
Let $\cC$ be a stable $\infty$-category.
A \emph{filtration in $\cC$} is a functor
\(
F
\colon
\Z^{op}\to \cC.
\)
The \emph{underlying object of $F$} is $F(-\infty):=\colim_n F(n)$.
We say that $F$ is \emph{complete}
if $\lim_n F(n) \simeq 0$.
For every integer $n$,
we use the notation
\[
\Fil_nF:=F(n),
\;
\Fil^n F:=\cofib(\Fil_{n+1}F\to F),
\;
\gr^nF:=\cofib(\Fil_{n+1}F\to \Fil_nF).
\]
\end{df}

\begin{exm}
For $\cF\in \rD(\ul{A})$,
the slice filtration on $\cF$ in Definition \ref{abelian rho} is a filtration.
For every integer $n$,
we have $P_n\cF=\Fil_n \cF$,
$P^n\cF=\Fil^n \cF$, and $P_n^n\cF=\gr^n \cF$.
\end{exm}

The following is an equivariant refinement of \cite[section 5.1]{BMS19}.

\begin{df}
Let $A$ be a Green functor.
The \emph{filtered derived $\infty$-category of $A$-modules} is
\[
\DF(A)
:=
\Fun(\Z^{op},\rD(A)).
\]
There is a natural symmetric monoidal structure on $\DF(A)$,
see \cite[Theorem 1.13]{GP18}.
It is given by
\[
(F\sotimes^\L_A G)(i)
:=
\colim_{j+k\geq i}
F(j)\sotimes^\L_A G(k)
\]
for $F,G\in \DF(A)$ and integers $i$.
Let $\widehat{\DF}(A)$ denote the full subcategory of $\DF(A)$ spanned by complete filtered complexes.
\end{df}

\begin{lem}
\label{Properties of DF}
Let $A$ be a Green functor.
\begin{enumerate}
\item[\textup{(1)}]
The symmetric monoidal structure on $\DF(A)$ restricts to a symmetric monoidal structure on $\widehat{\DF}(A)$.
\item[\textup{(2)}]
The inclusion functor $\widehat{\DF}(A)\to \DF(A)$ admits a left adjoint,
which we call the \emph{completion}.
\end{enumerate}
\end{lem}
\begin{proof}
We refer to \cite{GP18},
which is written for a general setting of symmetric monoidal $\infty$-categories.
Compare also \cite[Lemma 5.2]{BMS19},
which restricts to the case of derived categories of $\E_\infty$-rings.
\end{proof}

\begin{lem}
\label{limit and smash}
Let $A \in \CAlg^{\Z/2}$ be a commutative ring object of $\Sp^{\Z/2}$,
let $L$ and $M$ an $A$-modules,
and let $\cdots \to \Fil_1 M\to \Fil_0 M:=M$ be a complete filtration on $M$.
Assume that $A$, $L$, and $M$ are $(-1)$-connected and $\Fil_n M$ is $(n-1)$-connected for every integer $n$.
Then there are natural equivalence of $A$-modules
\[
\lim_n (L\wedge_A \Fil_n M)
\simeq
0
\text{ and }
\lim_n (L\wedge_A \Fil^n M)
\simeq
L\wedge_A \lim_n \Fil^n M.
\]
\end{lem}
\begin{proof}
Using \cite[Corollary 6.8.1]{BGS20},
we see that $L\wedge_A \Fil_n M$ is $(n-1)$-connected.
It follows that $\lim_n (L\wedge_A \Fil_n M)$ vanishes.
Use the cofiber sequence $\Fil_{n+1}M\to M\to \Fil^n M$ to finish the proof.
\end{proof}

\begin{lem}
\label{p-complete and filtration}
Let $A$ be a commutative ring,
and let $\cdots \to \Fil_1\cF\to \Fil_0\cF:=\cF$ be a complete filtration on $\cF\in \rD(\ul{A})$.
Then we have the following properties:
\begin{enumerate}
\item[\textup{(1)}]
$\Fil_\bullet (\cF_p^\wedge) :=(\Fil_\bullet \cF)_p^\wedge$ is a complete filtration on $\cF_p^\wedge$.
\item[\textup{(2)}]
We have an equivalence
\(
(\gr^n \cF)_p^\wedge
\simeq
\gr^n (\cF_p^\wedge)
\)
in $\rD(\ul{A})$ for every integer $n$.
\item[\textup{(3)}]
If $\gr^n\cF$ is derived $p$-complete for every integer $n$,
then $\cF$ is derived $p$-complete. 
\end{enumerate}
\end{lem}
\begin{proof}
(1) We have equivalences
\[
\lim_n \Fil_n (\cF_p^\wedge)
\simeq
\lim_n \lim_i \Fil_n (\cF \sotimes_{\ul{A}}^\L (\ul{A/p^i A}))
\simeq
\lim_i \cofib( \lim_n \Fil_n \cF \xrightarrow{p^i}   \lim_n \Fil_n \cF  )
\]
in $\rD(\ul{A})$,
and the last one vanishes since $\lim_{n} \Fil_n \cF$ vanishes.
Hence $\Fil_\bullet (\cF_p^\wedge)$ gives a complete filtration on $\cF_p^\wedge$.

(2) Use the cofiber sequences $\Fil_{n+1}\cF\to \Fil_n\cF\to \gr^n \cF$ and $\Fil_{n+1} (\cF_p^\wedge) \to \Fil_n (\cF_p^\wedge)\to \gr^n (\cF_p^\wedge)$.

(3) By induction and (2),
we have an equivalence $\Fil^n \cF \simeq \Fil^n(\cF_p^\wedge)$ for every integer $n$.
Take $\lim_n$ on both sides, and use the completeness of $\Fil_n\cF$ and (1) to conclude.
\end{proof}

The following result seems well-known.
\begin{lem}
\label{completed cotangent}
Let $R$ be a commutative $\Z_p$-algebra.
Then there is a natural equivalence $(\L_{R/\Z})_p^\wedge \simeq (\L_{R/\Z_p})_p^\wedge$ in $\rD(R)$.
\end{lem}
\begin{proof}
We have the transitivity fiber sequence
\(
\L_{\Z_p/\Z}\otimes_{\Z_p}^\L R
\to
\L_{R/\Z}
\to
\L_{R/\Z_p}.
\)
Since $\Z_p/p^n \Z_p$ admits a finite filtration whose graded pieces are $\F_p$, it suffices to show $\L_{\Z_p/\Z}\otimes_{\Z_p}^\L \F_p\simeq 0$. 
The induced morphism $\L_{\F_p/\Z}
\to
\L_{\F_p/\Z_p}$ in $\rD(\F_p)$
is an equivalence, as a standard computation, see e.g.\  \cite[\href{https://stacks.math.columbia.edu/tag/08SJ}{Tag 08SJ}]{stacks} shows  $\L_{\F_p/\Z}\simeq (p)/(p)^2[1]$ and  similarly for $\L_{\F_p/\Z_p}$.
Use the transitivity fiber sequence
\(
\L_{\Z_p/\Z}\otimes_{\Z_p}^\L \F_p
\to
\L_{\F_p/\Z}
\to
\L_{\F_p/\Z_p}
\)
to conclude.
\end{proof}

The following result is an equivariant refinement of \cite[a part of Lemma 2.5]{BMS19},
which will be used later for relating statements about $\THR$
and $\HR$.

\begin{prop}
\label{prop:hrcompleted}
Let $A$ be a commutative ring.
Then the natural morphism
\[
\THR(A;\Z_p)
\sotimes_{\THR(\Z)}^\L
\ul{\Z}
\to
\HR(A;\Z_p)
\]
in $\rD(\ul{A})$ is an equivalence.
\end{prop}
\begin{proof}
By Proposition \ref{monoidal completion algebra},
we have equivalences
\begin{align*}
(\THR(A)\sotimes_{\THR(\Z)}^\L \ul{\Z})_p^\wedge
\simeq &
(\THR(A;\Z_p)\sotimes_{\THR(\Z;\Z_p)}^\L \ul{\Z}_p^\wedge)_p^\wedge
\\
\simeq &
(\THR(A;\Z_p)\sotimes_{\THR(\Z)}^\L \ul{\Z})_p^\wedge.
\end{align*}
Hence
it suffices to show that
\(
\THR(A;\Z_p)
\sotimes_{\THR(\Z)}^\L
\ul{\Z}
\)
is derived $p$-complete.

Recall that $\ul{H}_i(\THR(\Z))$ is finite for $i>0$ by Lemma \ref{lem:finiteness of THR},
$\ul{H}_0(\THR(\Z))\simeq \ul{\Z}$ by
\cite[Theorem 5.1]{DMPR21} (see also 
\cite[Proposition 2.3.5]{HP23} and \cite{HP23erratum}),
and $\ul{H}_i(\THR(\Z))=0$ for $i<0$ by Proposition \ref{connective}.
As $\THR(A;\Z_p)$ is derived $p$-complete,
it suffices to show that
\[
\THR(A;\Z_p)
\sotimes_{\THR(\Z)}^\L
\tau_{\geq 1}\THR(\Z)
\]
is derived $p$-complete.
Recall that $\THR(A;\Z_p)$ is $(-1)$-connected by Proposition \ref{connective}.
The filtration
\(
\THR(A;\Z_p)\sotimes_{\THR(\Z)}^\L \tau_{\geq \bullet} \THR(\Z)
\)
is complete by a $\rD(\ul{A})$-variant of Lemma \ref{limit and smash}.
Using Lemma \ref{p-complete and filtration}(3),
we reduce to showing that
\[
\THR(A;\Z_p)\sotimes_{\THR(\Z)}^\L \ul{H}_i \THR(\Z)
\]
is derived $p$-complete for every integer $i\geq 1$.
Since $\ul{H}_i \THR(\Z)$ is finite for $i\geq 1$,
Construction \ref{const:finite Mackey} allows us to reduce to showing that
\[
\alpha(M):=
\THR(A;\Z_p)\sotimes_{\THR(\Z)}^\L  \ul{M}
\text{ and }
\beta(M):=
\THR(A;\Z_p)\sotimes_{\THR(\Z)}^\L  \ul{M^{\oplus \Z/2}}
\]
are derived $p$-complete for every finite abelian group $M$.
We further reduce to the case when $M=\Z/\ell^m$ for some prime $\ell$ and integer $m\geq 1$.

The cofiber of the multiplication $\ell^m\colon \ul{\Z}\to \ul{\Z}$ is $\ul{\Z/\ell^m}$.
If $\ell\neq p$,
then we have $\alpha(M)\simeq 0$ by Proposition \ref{completion results}(6). A similar argument yields $\beta(M)\simeq 0$.

Now assume $\ell=p$.
For every integer $k\geq 1$,
the quotient homomorphism $\Z/p^{m(k+1)}\to \Z/p^{mk}$ induces the following commutative diagram in $\rD(\ul{A})$
\[
\begin{tikzcd}
\ul{\Z/p^m} \ar[d,"0"']\ar[r,"0"]&
\ul{\Z/p^m} \ar[d,"\id"]\ar[r]&
\ul{\Z/p^m}\sotimes_{\ul{\Z}}^\L \ul{\Z/p^{m(k+1)}}\ar[d]
\\
\ul{\Z/p^m}\ar[r,"0"]&
\ul{\Z/p^m}\ar[r]&
\ul{\Z/p^m}\sotimes_{\ul{\Z}}^\L \ul{\Z/p^{mk}}
\end{tikzcd}
\]
whose rows are
obtained by applying $\ul{\Z/p^m} \sotimes_{\ul{\Z}}^\L$ to the cofiber sequences $\ul{\Z} \xrightarrow{p^{m(k+1)}} \ul{\Z} \to  \ul{\Z/p^{m(k+1)}}$ and $\ul{\Z}\xrightarrow{p^{mk}} \ul{\Z} \to \ul{\Z/p^{mk}}$.
Apply $\THR(A;\Z_p)\sotimes_{\THR(\Z)}^\L(-)$ to this diagram,
and take $\lim_k$ as $k$ varies.
Using $\lim(\cdots \xrightarrow{0}\ul{\Z/p^m}\xrightarrow{0} \ul{\Z/p^m})\simeq 0$,
we obtain
\begin{align*}
\alpha(M)
\simeq  &
\lim(\cdots \xrightarrow{\id} \alpha(M)\xrightarrow{\id}  \alpha(M))
\\
\simeq &
\lim(\cdots \to \alpha(M)\sotimes_{\ul{\Z}}^\L \ul{\Z/p^{2k}}
\to
\alpha(M)\sotimes_{\ul{\Z}}^\L \ul{\Z/p^{k}})
\simeq 
\alpha(M)_p^\wedge.
\end{align*}
This implies that $\alpha(M)$ is derived $p$-complete by Proposition \ref{completion results}(4).
Again, a variant of this argument shows that $\beta(M)$ is derived $p$-complete as well:
We have a commutative diagram in $\rD(\ul{A})$
\[
\begin{tikzcd}
\ul{(\Z/p^m)^{\oplus \Z/2}} \ar[d,"0"']\ar[r,"0"]&
\ul{(\Z/p^m)^{\oplus \Z/2}} \ar[d,"\id"]\ar[r]&
\ul{(\Z/p^m)^{\oplus \Z/2}}\sotimes_{\ul{\Z}}^\L \ul{\Z/p^{m(k+1)}}\ar[d]
\\
\ul{(\Z/p^m)^{\oplus \Z/2}}\ar[r,"0"]&
\ul{(\Z/p^m)^{\oplus \Z/2}}\ar[r]&
\ul{(\Z/p^m)^{\oplus \Z/2}}\sotimes_{\ul{\Z}}^\L \ul{\Z/p^{mk}}
\end{tikzcd}
\]
whose rows are cofiber sequences, and then apply $\THR(A;\Z_p)\sotimes_{\THR(\Z)}^\L(-)$ to this diagram and take $\lim_k$ as $k$ varies.
\end{proof}

\begin{df}
For a commutative ring $R$,
let $\rD(\ul{R})_{\sigma-\mathrm{sums}}$ be the full subcategory of the $\infty$-category $\rD(\ul{R})$
spanned by the objects of the form
\[
\cF:=
\bigoplus_{n=0}^\infty \ul{\cF_n}[n\sigma]
\]
such that $\cF_n$ is an $R$-module for all integers $n\geq 0$.
In this case,
we have an isomorphism $\cF_n\cong H_n(i^*\cF)$ of $R$-modules. 
This rather technical definition will be considered in the following Lemma, which will be used in the proof of 
Theorem \ref{HRfiltration}.
\end{df}

\begin{lem}
\label{sigma filtration}
Let $R$ be a commutative ring.
Then there exists a natural filtration $\Fil_n \cF$ with $n\geq 0$ for all $\cF\in  \rD(\ul{R})_{\sigma-\mathrm{sums}}$ such that the $n$th graded piece $\gr^n \cF$ is equivalent to $\ul{H_n(i^*\cF)}[n\sigma]$ for all integers $n\geq 0$.
\end{lem}
\begin{proof}
We set $\Fil_0\cF:=\cF$.
For every integer $n\geq 0$,
we recursively construct a \emph{natural} filtration
\[
\Fil_{n+1} \cF
:=
(P_{2n+2}(\Fil_n \cF[n+1]))[-n-1].
\]
We have the natural morphism
\[
\Fil_{n+1}\cF
\to
\Fil_n\cF
\]
induced by the natural morphism $P_{2n+2}(\Fil_n \cF[n+1])\to \Fil_n \cF[n+1]$.
We claim that there exists an equivalence
\[
\Fil_n \cF
\simeq
\bigoplus_{j=n}^\infty \ul{\cF_j}[j\sigma]
\]
for all integers $n\geq 0$,
where $\cF_j:=H_j(i^*\cF)$ for every integer $j\geq 0$.
This holds if $n=0$.
Assume that the claim holds for $n$.
We have
\[
\ul{\cF_j}[j\sigma+n+1]\in \rD_{\geq 2n+2}(\ul{R})
\]
for $j>n$ by Lemmas \ref{shift 2} and \ref{abelian filtration}(1) and
\[
P_{2n+1}^{2n+1} \ul{\cF_n}[n\sigma + n+1]
\simeq
\ul{\cF_n}[n\sigma + n+1]
\]
by Proposition \ref{abelian slicesexplicit},
see Definition \ref{abelian rho} for the notation $\rD_{\geq 2n+2}(\ul{R})$.
Together with $P_{2n+2}P_{2n+1}^{2n+1}\simeq 0$, we have $P_{2n+2}\ul{\cF_n}[n\sigma+n+1]\simeq 0$.
This shows the claim for $n+1$,
which completes the induction.
Then $\gr^n \cF$ can be identified with $\ul{\cF_n}[n\sigma]$.
\end{proof}

The following Lemma is the starting point of all computations, ultimately leading to Theorem \ref{perfectoid THR}. We currently can't generalize this lemma to $R$ or $\N$ having a non-trivial involution, but see Remark \ref{addinginvolutions} below.

\begin{lem}
\label{HR(R[N])}
Let $R$ be a commutative ring.
Then there is an equivalence
\[
\HR(R[\N]/R)
\simeq
\ul{\Omega_{R[\N]/R}^0} \oplus
\ul{\Omega_{R[\N]/R}^1} [\sigma].
\]
in $\rD(\ul{R[\N]})$.
Note that by definition the left-hand side is an algebra and hence a module over $\rH \ul{R[\N]}$,
which we then consider as an object in  $\rD(\ul{R[\N]})$ using \textup{Proposition \ref{greenequivalence}}.
In particular,
we have $
\HR(R[\N]/\R)\in \rD(\ul{R[\N]})_{\sigma-\mathrm{sums}}.
$
\end{lem}
\begin{proof}
We have the maps $\Sphere[\N]\to \THR(\Sphere[\N])\to \Sphere[\N]$ in $\NAlg^{\Z/2}$ from Definition \ref{def:thrandhr}.
Hence $\THR(\Sphere[\N])$ admits a decomposition $\Sphere[\N]\oplus X$ as an $\Sphere[\N]$-module.
By \cite[Propositions 4.2.11, 4.2.13]{HP23},
we obtain an equivalence of $\Z/2$-spectra
\[
\THR(\Sphere[\N])
\simeq
\bigoplus_{i=0}^\infty \Sphere
\oplus
\bigoplus_{i=0}^\infty \Sigma^\sigma \Sphere.
\]
using that $\Sphere[S^{\sigma}] \simeq \Sphere \oplus \Sigma^{\sigma}\Sphere$ and an index shift.
From this,
we obtain an equivalence of $\Z/2$-spectra
\(
X
\simeq
\bigoplus_{i=0}^\infty \Sigma^\sigma \Sphere.
\)
We have equivalences of $\rH \ul{R}$-modules
\begin{align*}
\HR(R[\N]/R)
= &
\THR(R[\N])
\wedge_{\THR(R)}
\rH \ul{R}
\\
\simeq &
\THR(\Sphere[\N])
\wedge_{\THR(\Sphere)}
\rH \ul{R}
\\
\simeq &
\THR(\Sphere[\N]) \wedge \rH \ul{R}
\end{align*}
using \cite[Proposition 4.2.15]{HP23} and the equivalence $\THR(\Sphere)\simeq \Sphere$.
Hence setting $Y\simeq X\wedge \rH \ul{R}$, we obtain an equivalence
\[
\HR(R[\N]/R)
\simeq
\ul{R[\N]} \oplus Y
\]
in $\rD(\ul{R[\N]})$.
Observe that we have an equivalence $Y\simeq \bigoplus_{i=0}^\infty \ul{R}[\sigma]$ in $\rD(\ul{R})$ by the equivalence of Proposition \ref{greenequivalence},
under which $[\sigma]$ corresponds to $\Sigma^{\sigma}$.

Recall the conservative forgetful functor $\alpha_*\colon \rD(\ul{R[\N]})\to \Sp^{\Z/2}$.
For every $\cF\in \rD(\ul{R[\N]})$,
we have a natural isomorphism of Mackey functors
\[
\ul{H}_{m+n\sigma}(\cF)\cong \ul{\pi}_{m+n\sigma}(\alpha_*\cF)
\]
by adjunction since $\alpha^*(\Sigma^{m+n\sigma}\Sphere)
\simeq
\ul{R[\N]}[m+n\sigma]$.
Using this,
we have an equivalence $Y\simeq \ul{M}[\sigma]$ in $\rD(\ul{R[\N]})$ for some
$R[\N]$-module $M$.
Together with the Hochschild-Kostant-Rosenberg theorem (see e.g.\ \cite[Theorem 3.4.4]{Lo})
and Proposition \ref{i* and THR},
we have isomorphisms of $R[\N]$-modules
\[
M\simeq H_1(i^* \HR(R[\N]/R))\cong H_1(\HH(R[\N]/R))\cong \Omega_{R[\N]/R}^1.
\]
On the other hand, the direct summand $\ul{R[\N]}$ canonically corresponds to $\ul{\Omega_{R[\N]/R}^0}$.
\end{proof}

\begin{lem}
\label{sigmaofRxy}
Let $R$ be a commutative ring, $M$ an $R$-module with involution $\tau$.
\begin{enumerate}
\item[\textup{(1)}] If 2 is invertible in $R$, 
we have an equivalence $\ul{M}[\sigma] \simeq \ul{(M,w)}[1]$ in $\rD(\ul{R})$, where $w(m)=-\tau(m)$ for all $m \in M$.
\item[\textup{(2)}] 
The $R$-module $M \oplus M$ with involution $(a,b) \mapsto (b,a)$ is isomorphic as an $R$-module with involution to $M \oplus M$ with involution  $(a,b) \mapsto (\tau(b),\tau(a))$.
\end{enumerate}
\end{lem}
\begin{proof}
For the second statement, one uses the isomorphism of $R$-modules with involution given by $(a,b) \mapsto (a,\tau(b))$.
For the first part,
we compute the kernel and cokernel of the morphism $\ul{M^{\oplus \Z/2}}\to \ul{M}$ induced by the summation $M\oplus M\to M$, where by the definition of $[\sigma]$ in Construction \ref{alphaupperstar} the kernel is $\ul{M}[\sigma -1]$, and the cokernel is $0$ as 2 is invertible.
Now one easily verifies that $\ul{M}[\sigma-1]$
is quasi-isomorphic to the $\ul{R}$-module given by the diagram
\[
\begin{tikzcd}
M\ar[loop below,"-\tau"]\ar[r,shift left=0.5ex,"1-\tau"]\ar[r,shift right=0.5ex,leftarrow,"\res"']&
\{m\in M:m=-\tau(m)\},
\end{tikzcd}
\]
where $\res$ is the inclusion.
This $\ul{R}$-module is isomorphic to $\ul{(M,w)}$.
\end{proof}

\begin{rmk}\label{omegaone}
We now assume that $2$ is invertible in $R$.
As before, the following computation is carried out in the abelian category of modules over the Green functor $\ul{R[\N]}=\ul{R[x]}$.
\[
\begin{tikzcd}
R[\N]\ar[loop below,"w"]\ar[r,shift left=0.5ex,"0"]\ar[r,shift right=0.5ex,leftarrow,"0"']&
0,
\end{tikzcd}
\]
Using the isomorphism of $R[\N]$-modules $R[\N] \cong \Omega^1_{R[\N]/R}$ and Lemma \ref{sigmaofRxy}(1),
we obtain an equivalence 
in $\rD(\ul{R[\N]})$
\[
\HR(R[\N]/R)
\simeq
\ul{\Omega^0_{R[\N]/R}}
\oplus
\ul{(\Omega^1_{R[\N]/R},w)}[1],
\]
where $w$ is the $R[\N]$-linear involution on $R[\N]$ given by $a\mapsto -a$ for $a\in R[\N]$.
So in this description the involution on $\Omega^1_{R[x]/R}$ is induced by $fdx \mapsto -fdx$.
Note that via the standard isomorphism 
$\HH_1(R[x]/R) \cong \Omega_{R[x]/R}^{1}$ given by $1 \otimes x - x \otimes 1 \mapsto dx$, this involution corresponds to changing the two tensor factors on $\HH_1$, as expected.
Let us also point out a possible source of confusion here: the cokernel above in modules over Green functors is $0$ only if $2$ is invertible in $R$, but the underlying map of $R$-modules with involution is surjective even if $2$ is not invertible.
\end{rmk}

In the next,
we use the notion of flat modules of Green functors \cite[Definition A.5.1]{HP23}.

\begin{lem}
\label{monoidal HR}
Let $R\to A$, $R \to B$ be morphisms of Green functors.
If either $A$ or $B$ is flat over $R$,
then there is a natural equivalence
in $\rD(A\sotimes_R B)$
\[
\HR(A\sotimes_R B)
\simeq
\HR(A/R)
\sotimes_{R}^\L \HR(B/R).
\]
\end{lem}
\begin{proof}
We set $C:=A\sotimes_R B$, which is equivalent to $A\sotimes_R^\L B$ using the assumption of flatness.
We have the equivalences
\begin{align*}
\HR(C/R)
= &
\THR(C)
\sotimes_{\THR(R)}^\L
R
\\
\simeq &
(\THR(A)\sotimes_{\THR(R)}^\L\THR(B))
\sotimes_{\THR(R)}^\L
R
\\
\simeq &
(\THR(A)\sotimes_{\THR(R)}^\L R)
\sotimes_{R}^\L
(\THR(B)\sotimes^\L_{\THR(R)} R)
\\
\simeq &
\HR(A/R)\sotimes_{R}^\L\HR(B/R)
\end{align*}
where the second one is due to \cite[Proposition 2.1.5]{HP23}. 
\end{proof}

The following is an equivariant refinement of the Hochschild-Kostant-Rosenberg theorem \cite{HKR},
see also \cite[p.\ 215]{BMS19}.
We will relate this filtration to slice filtrations after derived $p$-completion in Theorem \ref{perfectoid THR} below,
refining \cite[Theorem 6.1]{BMS19}, in the case where $A$ is a perfectoid ring and $R=\Z_p$.
Recall that Example \ref{left Kan H} extended the equivariant Eilenberg-MacLane spectrum $\rH(-)$ to simplicial commutative $R$-algebras.

\begin{thm}\label{HRfiltration}
Let $R$ be a commutative ring,
and let $A$ be a simplicial commutative $R$-algebra.
Then there exists a natural filtration $\Fil_\bullet \HR(A/R)$ on $\HR(A/R)$ whose $n$th graded piece admits a natural identification
\[
\gr^n \HR(A/R)
\simeq
(\iota \wedge_{A}^n  \L_{A/R})[n\sigma]
\]
for every integer $n$.
If $2$ is invertible in $R$,
then this filtration is complete.
\end{thm}
\begin{proof}
If $A=R[x_1,\ldots,x_d]$ with an integer $d\geq 0$,
then Lemmas \ref{HR(R[N])} and \ref{monoidal HR} imply $\HR(A/R)\in \rD_{\sigma-\mathrm{sums}}(\ul{A})$,
 i.e.,
$\HR(A/R)\simeq \bigoplus_{n=0}^d\ul{\cF_n}[n\sigma]$ for some $A$-modules $\cF_n$.
Using the Hochschild-Kostant-Rosenberg theorem
and Proposition \ref{i* and THR},
we have isomorphisms of $A$-modules
\[
\cF_n
\cong
H_n(i^* \HR(A/R))
\cong
H_n(\HH(A/R))
\cong
\Omega_{A/R}^n,
\]
where the second and third isomorphisms are natural. We will crucially rely on this naturality in the Hochschild-Kostant-Rosenberg theorem when we identify the graded pieces of $\HR(A/R)$ below.
Hence we obtain a (non-natural)
equivalence
\begin{equation}
\label{HRomegasum}
\HR(A/R)
\simeq
\bigoplus_{n=0}^d
\ul{\Omega_{A/R}^n}[n\sigma]
\end{equation}
in $\rD(\ul{A})$.
In particular, we have $\HR(A/R)\in \rD_{\sigma-\mathrm{sums}}(\ul{A})$.
Together with the natural filtration in Lemma \ref{sigma filtration},
we can regard $\HR(A/R)$ as an object of $\Fun(\Z^{op},\rD(\ul{A}))$.
Proposition \ref{abelian Dold-Kan}
allows us to regard $\HR(A/R)$ as an object $\Fun(\Z^{op},\Mod(\rD(\ul{R})))$,
where $\Mod(-)$ denotes the $\infty$-category of the module objects of a symmetric monoidal category.
Hence we obtain a functor
\[
\HR(-/R)
\colon
\Poly_R
\to
\Fun(\Z^{op},\Mod(\rD(\ul{R})))
\]
(where again we work with $\Mod$ to keep track the $A$-module structure on $\HR(A/R)$) 
such that its $n$th graded piece $\gr^n\HR(A/R)$ is naturally equivalent to $\underline{\Omega_{A/R}^n} [n\sigma]$ as $A$-modules in $\rD(\ul{R})$ for every integer $n$.

For general $A\in \sCRing_R$,
we get the desired filtration by left Kan extension,
using the description of $\L_{A/R}$ from Example \ref{left Kan Omega}.
We also use that the $\wedge_A^n$ in $\Omega_{A/R}^n=\wedge_A^n \Omega_{A/R}^1$ has left Kan extension $\wedge_A^n$, and $\ul{(-)}$ has left Kan extension $\iota$.
To check that this filtration is complete if $2$ is invertible in $R$,
use the forgetful functor $\Mod(\rD(\ul{R}))\to \rD(\ul{R})$ to obtain a functor
\[
\HR(-/R)\colon \sCRing_R\to \Fun(\Z^{op},\rD(\ul{R}))=\DF(\ul{R}).
\]
By Lemma \ref{sigmaofRxy}(1) below,
we see that $\Fil_n\HR(A/R)$ is $(n-1)$-connected for $A\in \Poly_R$ and integer $n$.
As any sifted colimit of $(n-1)$-connected objects is $(n-1)$-connected, 
we deduce that $\Fil_n\HR(A/R)$ is $(n-1)$-connected for
$A\in \sCRing_R$ and integer $n$.
Hence the filtration on $\HR(A/R)$ is complete.
\end{proof}

The following result deserves to be called the \emph{real
Hochschild-Kostant-Rosenberg theorem}.

\begin{thm}
\label{smoothHKR}
Let $R$ be a commutative ring,
and let $A$ be a smooth $R$-algebra.
Then the natural filtration on $\HR(A/R)$ in Theorem \ref{HRfiltration} is complete, and its $n$th graded piece in $\rD(A)$ given by
\[
\gr^n\HR(A/R)
\simeq
\ul{\Omega_{A/R}^n}[n\sigma]
\]
for every integer $n$.
\end{thm}
\begin{proof}
The last claim follows from $\L_{A/R}\simeq \Omega_{A/R}^1$ and Theorem \ref{HRfiltration}.

By \cite[Theorem 3.4.3]{HP23} (see also \cite{HP23erratum}) 
$\THR(-)$ is an \'etale hypersheaf on the affine isovariant site, and hence on the subsite of affine schemes with trivial $C_2$-action over a given base ring $R$. Definition \ref{def:thrandhr} immediately implies that $\HR(-/R)$ is also an \'etale sheaf.
Furthermore,
every limit of complete filtered complexes is complete.
Hence it suffices to work Zariski locally on $A$ to show that the filtration on $\HR(A/R)$ is complete.
Together with \cite[Corollaire IV.17.11.4]{EGA},
we may assume that $R\to A$ factors through $R\to R[x_1,\ldots,x_d]\to A$,
where the first homomorphism is the obvious inclusion and the second homomorphism is \'etale.

We set $R':=R[x_1,\ldots,x_d]$.
For $B\in \Poly_{R'}$,
we have
\[
\HR(R'/R)\wedge_{\rH \ul{R'}}\rH \ul{B},
\;
\HR(B/R)
\in
\rD(\ul{R})_{\sigma-\mathrm{sums}}
\]
by \eqref{HRomegasum}.
Using the naturality of the filtration in Lemma \ref{sigma filtration},
we see that the base change map
\begin{equation}
\label{basechange}
\HR(R'/R)\wedge_{\rH \ul{R'}}\rH \ul{B}
\to
\HR(B/R)
\end{equation}
obtained from \cite[(3.1)]{HP23} is compatible with the filtrations for $B\in \Poly_{R'}$.
By left Kan extension,
we see that \eqref{basechange} is compatible with the filtrations for $B\in \sCRing_{R'}$.
In particular, the base change map
\[
\HR(R'/R)\wedge_{\rH \ul{R'}}\rH \ul{A}
\to
\HR(A/R)
\]
is compatible with the filtrations.
This map is an equivalence by \cite[Theorem 3.2.3]{HP23}.
To conclude,
observe that the filtration on $\HR(R'/R)$ is finite and hence complete as $\Omega^n_{R'/R}$ vanishes for $n>d$.
\end{proof}

\begin{rmk}
Let $R$ be a commutative ring,
and let $A$ be a smooth $R$-algebra.
If $2$ is invertible in $R$, then $R[\sigma-1]$ is equivalent to $\ul{(R,w)}$ by Lemma \ref{sigmaofRxy}(1) below, where $w$ is the involution $R\to R$ sending $x$ to $-x$. In particular, $R[\sigma-1]$ is  then concentrated in degree $0$.
This implies that $\ul{H}_{i\sigma} \ul{\Omega_{A/R}^n}[n\sigma]=0$ whenever $i\neq n$.
Hence for every integer $n$,
the equivalence of Theorem \ref{smoothHKR} induces a natural isomorphism
\[
\ul{H}_{n\sigma}\HR(A/R)
\simeq
\ul{\Omega_{A/R}^n}.
\]
\end{rmk}

We now discuss some partial results towards a generalization to commutative $R$-algebras with a non-trivial involution.

\medskip

\medskip

\begin{df}
\label{equivariant Omega}
For any commutative ring $R$ and for any commutative $R$-algebra $A$, we considered the graded $A$-module $\Omega^*_{A/R}$.
Observe that if $A$ comes with an $R$-linear involution, then for any $n \geq 0$, this induces an involution on $\Omega^n_{A/R}$.
Restricting to $\Poly_R^{\Z/2}$ and using Construction \ref
{extendwithinvolutions}, we obtain the definition of the {\em real cotangent complex} in degree $n$, denoted $\ul{\L^n_{A/R}}$, in $\rD(\ul{A})$, proceeding as in Example \ref{left Kan Omega} and also using Proposition \ref{abelian Dold-Kan}.
\end{df}

\begin{prop}\label{thra2}
Let $R$ be a commutative ring in which 2 is invertible, and let $A:=R[\N^{\oplus\Z/2}]$. Then we have an equivalence
\[
\HR(A/R)
\simeq  \bigoplus _{i=0}^2 \ul{\Omega_{A/R}^i}[i \sigma]
\]
in $\rD(\ul{A})$.
For all commutative rings $R$,
with $2$ invertible or not invertible,
we have an equivalence
\[
\HR(R[\N^{\oplus \Z/2}]/R)
\simeq  \ul{R[\N^{\oplus \Z/2}]} \oplus  (\ul{R[v,w] \oplus R[x,y]})[1] \oplus \ul{R[\N^{\oplus \Z/2}]}[1 + \sigma]
\]
in $\rD(\ul{A})$,
where the middle summand has the involution changing $v$ and $x$ as well as $w$ and $y$, but no involution on the individual summands $R[v,w]$ and $R[x,y]$.
We also have an equivalence of $\Sphere[\N^{\oplus \Z/2}]$-modules
\[
\THR(\Sphere[\N^{\oplus \Z/2}])
\simeq
\Sphere[\N^{\oplus \Z/2}] \oplus \Sigma^1 i_*\Sphere[\N^{\oplus \Z/2}]
\oplus
\Sigma^{1+\sigma} \Sphere[\N^{\oplus \Z/2}].
\]
\end{prop}
\begin{proof}
There are equivalences
\[
\THR(\Sphere[\N^{\oplus \Z/2}])
\simeq 
\THR(N^{\Z/2}\Sphere[\N])
\simeq
N^{\Z/2}
\THH(\Sphere[\N])
\]
in $\NAlg^{\Z/2}$,
where we need \cite[Proposition 2.1.3]{HP23} for the second one.
Recall that $\THR(\Sphere[\N^{\oplus \Z/2}])$ is an $\Sphere[\N^{\oplus \Z/2}]$-algebra via the inclusion on the second smash factor. Hence the above chain of equivalences induces an $\Sphere[\N^{\oplus \Z/2}]$-algebra structure on $N^{\Z/2}
\THH(\Sphere[\N])$, where the action is induced by applying $N^{\Z/2}$ to the usual action (via the second smash factor) of $\Sphere[\N]$ on $\THH(\Sphere[\N])$.
There is an equivalence of spectra $\THH(\Sphere[\N])\simeq \Sphere\oplus \bigoplus_{j\geq 1}\Sphere[S^1]$ by \cite[Proposition 3.20]{Rog09},
which is equivalent to $X\oplus \Sigma^1 X$ with $X:=\bigoplus_{j\geq 0}\Sphere\in \Sp$.
Analyzing this computation which reduces to the cyclic nerve of $\N$, compare also \cite[Proposition 4.2.11]{HP23}, we see that the $\Sphere[\N]$-module structure of this normed spectrum is induced by the action of $\N$ shifting the index in $\bigoplus$.
Hence we have an equivalence of $\Sphere[\N]$-modules $\THH(\Sphere[\N])\simeq \Sphere[\N]\oplus \Sigma^1\Sphere[\N]$.
Since $N^{\Z/2}\Sigma^1 \Sphere
\simeq
\Sigma^{1+\sigma}\Sphere$ by \cite[Proposition A.59]{HHR},
we have an equivalence
\[
N^{\Z/2}\THH(\Sphere[\N])
\simeq
\Sphere[\N^{\oplus \Z/2}] \oplus \Sigma^1 i_*\Sphere[\N^{\oplus \Z/2}]
\oplus
\Sigma^{1+\sigma} \Sphere[\N^{\oplus \Z/2}]
\]
of $\Sphere[\N^{\oplus \Z/2}]$-modules
using the usual behavior on the smash product on direct sums and the model categorical description of the norm functor in \cite[section 2.2.3]{HHR}:
for an orthogonal $\Z/2$-spectrum $X$,
we have $N^{\Z/2} X = X\wedge X$ with the involution switching the two factors.

Now, let $R$ be a commutative ring.
By \cite[Proposition 4.2.15]{HP23} and the definition of $\HR$,
we have equivalences
\begin{align*}
\HR(R[\N^{\oplus \Z/2}]/R)
= &
\THR(R[\N^{\oplus \Z/2}])\wedge_{\THR(R)}\rH \ul{R}
\\
\simeq &
(\THR(\Sphere[\N^{\oplus \Z/2}])\wedge \THR(R))\wedge_{\THR(R)}\rH \ul{R}
\\
\simeq &
\THR(\Sphere[\N^{\oplus \Z/2}])\wedge \rH \ul{R}.
\end{align*}
Consequently, we obtain an equivalence in $\rD(\ul{R[\N^{\oplus \Z/2}]})$
\[
\HR(R[\N^{\oplus \Z/2}]/R)
\simeq  \ul{R[\N^{\oplus \Z/2}]} \oplus  (\ul{R[v,w] \oplus R[x,y]})[1] \oplus \ul{R[\N^{\oplus \Z/2}]}[1 + \sigma].
\]

From now on, we assume that $2$ is invertible in $R$.
Using the above description of the middle summand via $i_*(X \wedge X)$,
it has the involution changing $v$ and $x$ as well as $w$ and $y$,
but no involution on the individual summands.
However, we can use both parts of Lemma  \ref{sigmaofRxy}, namely the second for $\tau$ the sign involution, to deduce that the middle summand is equivalent to $\ul{\Omega_{A/R}^1}[\sigma]$,
where the involution on $\Omega_{A/R}^1\cong R[x,y]dx \oplus R[x,y]dy$ is given by
\[
f(x,y)dx+g(x,y)dy \mapsto f(y,x)dy+g(y,x)dx.
\]
For the last summand, we use the Lemma \ref{sigmaofRxy}(1) and $dxdy=-dydx$ to see that $\ul{\Omega_{A/R}^2}[2\sigma]$ is equivalent to $\ul{R[\N^{\oplus \Z/2}]}[1+\sigma]$.
\end{proof}

Note that parts of this computation work for $i_*$ and arbitrary monoids $M$,  e.g.\ $\THR(\Sphere[M^{\oplus \Z/2}])\simeq N^{\Z/2}\THR(\Sphere[M])$. One should also observe the following: even if $\rho_n \THR(R)\simeq 0$ for every odd $n$, e.g.\ for $R$ a perfectoid ring in the completed variant, see Theorem \ref{perfectoid THR}), we no longer have $\rho_n \THR(R[\N^{\oplus \Z/2}])\simeq 0$ for every odd $n$.
This follows from Proposition \ref{prop:even} and \cite[Proposition 2.1.3]{HP23}, because  $i^*\THR(R[\N^{\oplus \Z/2}])=\THH(R[x,y])$ has homotopy groups in odd degrees
using
\[
\THH(R[x,y])
\simeq
\THH(R)\wedge \THH(\Sphere[\N])\wedge \THH(\Sphere[\N])
\]
and the explicit description of $\THH(\Sphere[\N])$ in \cite[Proposition 3.20]{Rog09}.
However, if we proceed as for Proposition \ref{HRsigma}, Example \ref{thrgm}, and  Remark \ref{perfectoid ring with nontrivial involution}, i.e.\ use colimit perfection to construct a (quasiregular semi-)perfectoid ring $S$ with involution from $R[\N^{\oplus \Z/2}]$, then applying $i^*$ and arguing as above we see that $\THR(S)$ is even.

\begin{rmk}\label{addinginvolutions}
It would be nice to generalize Theorem \ref{HRfiltration} to commutative rings $A$ and even $R$ with involution, thus establishing the most general version of a real Hochschild-Kostant-Rosenberg theorem.
We expect that if $2$ is invertible in $R$, then there exists a natural complete filtration $\Fil_\bullet \HR(A/R)$ on $\HR(A/R)$ whose $n$th graded piece is
\[
\gr^n \HR(A/R)
\simeq
\iota \wedge_A^n \L_{A/R}[n\sigma]
\]
for every integer $n$.
If $A$ is a smooth $R$-algebra,
then this would induce
an equivalence
\[
\gr^n \HR(A/R)
\simeq
\ul{\Omega_{A/R}^n}[n\sigma]
\]
in $\rD(\ul{A})$ for every integer $n$.
Combining Lemmas \ref{HR(R[N])} and \ref{monoidal HR} and Proposition
\ref{thra2} as well as a formula for $\Omega^1$ of finitely generated polynomial algebras, it is possible to show that this expected real Hochschild-Kostant-Rosenberg theorem for $\HR(A/R)$ and individual $A \in \Poly_R^{\Z/2}$ is true if $2$ is invertible in $R$.
However, we cannot use Construction \ref{extendwithinvolutions} to extend the filtration from $\Poly_R^{\Z/2}$ to $\CRing_R$ from the above  results unless we show that the expected filtration is natural in $A\in \Poly_R^{\Z/2}$.
Without involution, the naturality was established in the proof of Theorem \ref{HRfiltration} using the naturality of the classical Hochschild-Kostant-Rosenberg  theorem as an ingredient. In the more general case with involutions, one has to 
write down a map  $\ul{\Omega^1_{A/R}} \to \ul{H}_{\sigma}\HR(A/R)$ and show that it is natural on $\Poly_R^{\Z/2}$. There are at least two possible strategies for this, both assuming 2 is invertible in $R$.  First, one could try to describe $\HR(A/R)$ in terms of real simplicial sets to write down this map for general $A$.
We expect that after the Segal subdivision, $\HR(A/R)$ can be written as a certain chain complex
\[
\cdots
\to
A\otimes_R A\otimes_R A\otimes_R A
\to
A\otimes_R A.
\]
Second, one could try to define this map using Lemma \ref{HR(R[N])} and  Proposition \ref{thra2}, Lemma \ref{monoidal HR} and a similar formula for $\Omega^1$ for tensor products of $R$-algebras. This second approach might be easier for defining the map $\ul{\Omega^1}\to \ul{H}_{\sigma}\HR$, but  verifying that this is natural will presumably be more complicated.

The Example \ref{thrgm} below discusses functoriality in a very special case. 

If $2$ is not invertible in $R$, then the situation is actually worse:
Lemma \ref{sigma filtration} is not generalizable because $\HR(A/R)$ will not be contained in $\rD_{\sigma-\mathrm{sums}}(\ul{A})$.
Once functoriality is established, we can proceed
in combination with Construction 
\ref{extendwithinvolutions} to extend Theorem \ref{HRfiltration} to the case of nontrivial involutions.
Of course, extending results of later chapters, e.g.\ Theorem \ref{perfectoid THR}, to (perfectoid) rings $R$ with involution would require further additional work.
Finally, recall the evaluation functor $ev=(-)(C_2/e)$ considered e.g. before Proposition \ref{inverting 2}. This functor is not conservative. Still, we may consider its derived functor from $\rD(\ul{R})$ to the derived $\infty$-category of the abelian category of $R$-modules with involution. This commutes with $[\sigma]$ in an obvious sense, and computations as in Proposition \ref{thra2} suggest that the expected real Hochschild-Kostant-Rosenberg theorem for $R$-algebras with involution might hold in the setting of modules with involution (rather than Green modules) even if $2$ is not invertible.
\end{rmk}

Let us conclude the discussion about $R$-algebras with involutions by studying the other standard example. Observe that this result is compatible with the above conjecture on $\HR$ of $R$-algebras with involution.

\begin{prop}
\label{HRsigma}
Let $R$ be a commutative ring in which $2$ is invertible, and consider $A:=R[\Z^\sigma]=R[x,x^{-1}]$,
with $\Z^\sigma$ the abelian group $\Z$ with the sign involution and $R[x,x^{-1}]$ with the corresponding involution $x \mapsto x^{-1}$.
Then we have an equivalence
$$\HR(A/R)
\simeq  \ul{\Omega_{A/R}^0} \oplus \ul{\Omega_{A/R}^1}[\sigma]$$
in $\rD(\ul{A})$.
For all commutative rings $R$,
with $2$ invertible or not invertible,
we have an equivalence
\[
\HR(A/R)
\simeq  \ul{A} \oplus  \ul{A}[1]
\]
in $\rD(\ul{A})$,
We also have an equivalence of $\Sphere[\Z^\sigma]$-modules
\[
\THR(\Sphere[\Z^\sigma])
\simeq
\Sphere[\Z^\sigma] \oplus \Sigma^1\Sphere[\Z^\sigma].
\]
\end{prop}
\begin{proof}
By \cite[Proposition 5.9]{DMPR21} and \cite[(4.13)]{HP23},
we have equivalences of $\Z/2$-spectra
\[
\THR(\Sphere[\Z^{\sigma}])
\simeq
\Sphere[\Bdi \Z^\sigma]
\simeq \bigoplus_{j\in \Z^\sigma} \Sphere[S^1],
\]
where the involution acts on the index $\Z^\sigma$.
This is equivalent to $X\oplus \Sigma^1 X$ with $X:=\bigoplus_{j\in \Z^\sigma} \Sphere$ in $\Sp^{\Z/2}$.
Analyzing this computation which reduces to the dihedral nerve of $\Z^\sigma$, compare also \cite[Proposition 4.2.6]{HP23}, we see that the $\Sphere[\Z^\sigma]$-module structure of this normed spectrum is induced by the action of $\Z$ shifting the index in $\bigoplus$.
Hence we have an equivalence of $\Sphere[\Z^\sigma]$-modules $\THH(\Sphere[\Z^\sigma])\simeq \Sphere[\Z^\sigma]\oplus \Sigma^1\Sphere[\Z^\sigma]$.

We also have the equivalence of $\Z/2$-spectra $\THR(R[M]) \simeq \THR(R) \wedge \Sphere [\Bdi M]$ for arbitrary monoids $M$ with involution by \cite[Proposition 4.2.15]{HP23}.
Arguing as in Proposition \ref{thra2}, we deduce
an equivalence
\(
\HR(A/R)
\simeq  \ul{A} \oplus  \ul{A}[1]
\)
in $\rD(\ul{A})$.

From now on,
we assume that $2$ is invertible in $R$.
By Lemma \ref{sigmaofRxy}(1),
we have $\ul{A}[1]\simeq \ul{(A,w)}[\sigma]$, where $w$ is the involution on $A$ given by $f(x)\mapsto -f(1/x)$.
The involution on $\Omega_{A/R}^1$ is given by $f(x)dx\mapsto f(1/x)d(1/x)$.
As $A$-modules with involution,
we have the isomorphism $\Omega_{A/R}^1\to (A,w)$ sending $f(x)dx$ to $xf(x)$.
Hence we have an equivalence
\(
\HR(A/R)
\simeq  \ul{\Omega_{A/R}^0} \oplus \ul{\Omega_{A/R}^1} [\sigma]\)
in $\rD(\ul{A})$,
\end{proof}

We continue to study the $R$-algebra with involution from Proposition \ref{HRsigma} in the next example, illustrating  functoriality in one of the easiest possible cases.

\begin{exm}\label{thrgm}
Consider the real simplicial set $\rN^\sigma \Z^\sigma$ and its real geometric realization $\rB^\sigma \Z^\sigma$, see \cite[Definition 4.2.1]{HP23}.
For an integer $n$,
the map $n\colon \rN^\sigma \Z^\sigma\to \rN^\sigma \Z^\sigma$ induced by the multiplication $n\colon \Z^\sigma \to \Z^\sigma$ sends the element $1\in (\rN^\sigma \Z^\sigma)_1$ in simplicial degree $1$ to the element $n\in (\rN^\sigma \Z^\sigma)_1$ in simplicial degree $1$.
Together with the equivalence of $\Z/2$-spaces $S^1\simeq \rB^\sigma \Z^\sigma$ in \cite[Example 5.13]{DMPR21},
we obtain a commutative square of $\Z/2$-spaces
\[
\begin{tikzcd}
\rB^\sigma \Z^\sigma\ar[d,"n"']\ar[r,"\simeq"]&
S^1\ar[d,"n"]&
\\
\rB^\sigma \Z^\sigma\ar[r,"\simeq"]&
S^1
\end{tikzcd}
\]
with horizontal equivalences,
where $n\colon S^1\to S^1$ denotes the multiplication by $n$.
Now, assume $n\neq 0$.
There is an isomorphism of real simplicial sets $\Ndi \Z^\sigma \cong \Z^\sigma \times \N^\sigma \Z^\sigma$ which was already used implicitly in the proof of Proposition \ref{HRsigma} above,
see \cite[proof of Proposition 4.2.6]{HP23} for this description.
Using this,
we see that the morphism $n\colon \THR(\Sphere[\Z^\sigma])\to \THR(\Sphere[\Z^\sigma])$ induced by $n\colon \Z^\sigma\to \Z^\sigma$ can be identified with the map
\[
\bigoplus_{j \in \Z^\sigma} \Sphere[S^1]
\to
 \bigoplus_{j \in \Z^\sigma} \Sphere[S^1]
\]
sending the $j$th factor to $nj$th factor for $j\in \Z^\sigma$ whose underlying maps $\Sphere[S^1]\to \Sphere[S^1]$ is induced by $n\colon S^1\to S^1$.
Using the description of $\HR(R[\Z^\sigma]/R)$ in Proposition \ref{HRsigma},
we deduce that the morphism $\HR(R[\Z^\sigma]/R)\to \HR(R[\Z^\sigma]/R)$ in $\rD(\ul{R})$ induced by $n\colon \Z^\sigma\to \Z^\sigma$ can be identified with
\[
\alpha \oplus \beta\colon \ul{R[\Z^\sigma]}\oplus \ul{R[\Z^\sigma]}[1]\to \ul{R[\Z^\sigma]}\oplus \ul{R[\Z^\sigma]}[1],
\]
where
$\alpha(f(x)):=f(x^n)$ and $\beta(f(x)):=n f(x^n)x^{n-1}$ for $f(x)\in R[\Z^\sigma]$.
If we assume that $2$ is invertible in $R$,
then this can be also identified with the endomorphism on $\ul{\Omega_{R[\Z^\sigma]/R}^0}\oplus \ul{\Omega_{R[\Z^\sigma]/R}^1}[\sigma]$ induced by $f\colon R[\Z^\sigma]\to R[\Z^\sigma]$ given by $f(x)=x^n$, which is a morphism in $\Poly_R^{\Z/2}$.
Hence the isomorphism 
$$\HR(R[\Z^\sigma]/R)
\simeq  \ul{\Omega_{R[\Z^\sigma]/R}^0} \oplus \ul{\Omega_{R[\Z^\sigma]/R}^1}[\sigma]$$
is functorial with respect to $f$.
We will continue to study $R[\Z^{\sigma}]$ in Remark \ref{perfectoid ring with nontrivial involution} below, which will be modified to obtain a quasi-
regular semiperfectoid – and in fact even perfectoid – $\F_p$-algebra with involution, and whose $\THR$ will be shown to have the expected form.
\end{exm}

\begin{cor}
\label{complete HRfiltration}
Let $R$ be a commutative ring,
and let $A$ be a simplicial commutative $R$-algebra.
Then there exists a natural filtration $\Fil_\bullet \HR(A/R;\Z_p)$ on $\HR(A/R;\Z_p)$ whose $n$th graded piece is
\[
\gr^n \HR(A/R;\Z_p)
\simeq
\iota (\wedge_{A}^n  \L_{A/R})_p^\wedge [n\sigma]
\]
for every integer $n$.
If $2$ is invertible or if $A$ is a smooth $R$-algebra, 
then this filtration is complete.
\end{cor}
\begin{proof}
Note that derived $p$-completion obviously
commutes with $\gr$ and $[n\sigma]$, and also with $\iota$ by Proposition \ref{completioncommuteswith}.
Hence Lemmas \ref{p-complete and filtration}(1),(2) and
\ref{completed cotangent} together with Theorems \ref{HRfiltration} and \ref{smoothHKR} finish the proof.
\end{proof}

For a commutative ring $R$ in which $2$ is not invertible,
the difficulty of proving that the filtration in Theorem \ref{HRfiltration} is complete is caused by the fact that the intersection $\bigcap_{n\geq 0} (\rD_{\geq 0}(\ul{R})[n\sigma])$ is nonempty.
Indeed, if $\cF\in \rD_{\geq 0}(\ul{R})$ satisfies $i^*\cF\simeq 0$,
then we have $\cF\sotimes \ul{R^{\oplus \Z/2}}\simeq 0$,
so we have $\cF\simeq \cF[\sigma]$.
It follows that we have $\cF\in \rD_{\geq 0}(\ul{R})[n\sigma]$ for every integer $n$.
The following theorem is proved in \cite[Theorem 6.20]{Par23}, 
see also Remark 6.21 of loc.\ cit.

\begin{thm}
\label{conjHKR}
Let $R\to A$ be a map of commutative rings.
If the Frobenius $\varphi\colon A/2\to A/2$ (i.e., the squaring map) is surjective,
then the filtrations in Theorem \ref{HRfiltration} and Corollary \ref{complete HRfiltration} are complete.
\end{thm}

Observe that the condition of Theorem \ref{conjHKR} holds if $A$ is perfectoid and $R=\Z_p$ by definition.
Of course, 
if $2$ is invertible or if $A$ is a smooth $R$-algebra,
then the filtration of Theorem \ref{HRfiltration} is complete for other reasons (see Theorem \ref{smoothHKR}), and the surjectivity of $\varphi$ will not even hold in general in the smooth case.

\section{THR of perfectoid rings}
\label{sec5}

Here is the outline of the proof of Theorem \ref{perfectoid THR},
which is an equivariant refinement of \cite[Theorem 6.1]{BMS19}.
Let $R$ be a perfectoid ring.
\begin{enumerate}
\item[(1)]
We compute the slices of $\HR(R;\Z_p)$ in Proposition \ref{prop:perfectoid HR} using Theorem \ref{HRfiltration}, which is an equivariant refinement of the Hochschild-Kostant-Rosenberg theorem.
\item[(2)]
We define pseudo-coherent objects of $\rD(\ul{R})$ in Definition \ref{pseudo coherent}.
Lemmas \ref{lem:pseudo-coherence and exact sequence}--\ref{lem:pseudo-coherence and finitely generated} establish some useful facts about pseudo-coherent objects.
Lemma \ref{lem:pseudo-coherence} shows that $\HR(R;\Z_p)$ and $\THR(R;\Z_p)$ are pseudo-coherent.
\item[(3)]
In Lemmas \ref{perfectoid basechange HR} and \ref{lem:perfectoid basechange}, we establish a certain base change property for $\HR$ and $\THR$.
Lemma \ref{lem:Nakayama} deals with an application of the Nakayama lemma.
\item[(4)]
We do not establish a real refinement of \cite[Proposition IV.4.2]{NS} in general.
Instead,
we first compare the zeroth slices of $\THR(R;\Z_p)$ and $\HR(R;\Z_p)$ in Lemma \ref{lem:zeroth slice}.
Lemma \ref{lem:induction} provides an equivariant refinement of the induction argument in \cite[proof of Theorem 6.1]{BMS19},
and we use Lemma \ref{lem:induction} in Lemmas \ref{lem:first slice} and \ref{lem:second slice} to compare the first and second slices.
\item[(5)]
To finish the proof of Theorem \ref{perfectoid THR},
we combine the above arguments as in \cite[proof of Theorem 6.1]{BMS19},
and use Lemma \ref{lem:induction} again for the induction argument.
\end{enumerate}

Hence in general, several of the following results are real respectively Mackey refinements of statements established or used in the proof of \cite[Theorem 6.1]{BMS19}.
\begin{prop}
\label{prop:perfectoid HR}
Let $R$ be a perfectoid ring.
Then using notation introduced after Proposition \ref{slicesexplicit},
there is a natural equivalence of $\ul{R}$-modules
\[
\rho_n(\HR(R;\Z_p))
\simeq
\left\{
\begin{array}{ll}
\ul{R}& \text{if $n$ is even and nonnegative},
\\
0 & \text{otherwise}.
\end{array}
\right.
\]
\end{prop}
\begin{proof}
Consider the filtration $\Fil_\bullet\HR(R;\Z_p)$ on $\HR(R;\Z_p)$ in Corollary \ref{complete HRfiltration} whose $n$th graded piece is $\iota (\wedge_R^n \L_{R/\Z_p})_p^\wedge[n\sigma]$ for every integer $n$.
By \cite[Proposition 4.19(2)]{BMS19},
we have $(\wedge_R^n \L_{R/\Z_p})_p^\wedge \simeq R[n]$ if $n\geq 0$.
Hence we have
$\iota (\wedge_R^n \L_{R/\Z_p})_p^\wedge[n\sigma]\simeq \ul{R}[n+n\sigma]$,
which
is an object of $\rD_{\geq 2n}(\ul{R})\cap \rD_{\leq 2n}(\ul{R})$ by
\eqref{b}.
We claim that
\begin{equation}
\label{eqn:perfectoid HR}
\fib(\Fil_{m}\HR(R;\Z_p)\to \Fil_n\HR(R;\Z_p))\in \rD_{\geq 2n+2}(\ul{R})\cap \rD_{\leq 2m}(\ul{R})
\end{equation}
for every integer $m\geq n$, as we now show by induction on $m$. 
The claim is obvious if $m=n$.
Assume that the claim holds for $m$.
In the cofiber sequence
\begin{align*}
\gr^{m+1} \HR(R;\Z_p)
\to &
\fib(\Fil_{m+1} \HR(R;\Z_p)\to \Fil_n\HR(R;\Z_p))
\\
\to &
\fib(\Fil_m \HR(R;\Z_p)\to \Fil_n\HR(R;\Z_p)),
\end{align*}
the first and third terms are contained in $D_{\geq 2n+2}(\ul{R})\cap \rD_{\leq 2m+2}(\ul{R})$, so the same holds for the second one too.
This completes the induction argument for \eqref{eqn:perfectoid HR}.
Now apply $\lim_m$ to \eqref{eqn:perfectoid HR}
and use the completeness of the filtration established in \cite{Par23} (see also Theorem \ref{conjHKR}) to deduce $\Fil_n\HR(R;\Z_p)\in \rD_{\geq 2n}(\ul{R})$.
In particular,
we have $\Fil_{n+1}\HR(R;\Z_p)\in \rD_{\geq 2n+2}(\ul{R})$.
We also have $\cofib(\Fil_{n}\HR(R;\Z_p)\to \HR(R;\Z_p))\in \rD_{\leq 2n-2}(\ul{R})$.
Together with the cofiber sequences
\[
\Fil_{n+1}\HR(R;\Z_p)
\to
\HR(R;\Z_p)
\to
\cofib(\Fil_{n+1}\HR(R;\Z_p)\to \HR(R;\Z_p))
\]
and
\begin{align*}
\gr^n\HR(R;\Z_p)
\to &
\cofib(\Fil_{n+1}\HR(R;\Z_p)\to \HR(R;\Z_p))
\\
\to &
\cofib(\Fil_{n}\HR(R;\Z_p)\to \HR(R;\Z_p)),
\end{align*}
we have
\[
P_{a}^{a} \HR(R;\Z_p)\simeq P_a^a \cofib(\Fil_{n+1}\HR(R;\Z_p)\to \HR(R;\Z_p))
\simeq
P_{a}^{a}\gr^n \HR(R;\Z_p)
\]
for $a=2n,2n+1$.
Combine this with the above computation of $\gr^n \HR(R;\Z_p)$ to conclude.
\end{proof}

If $2$ is invertible in $R$ or equivalently if the fixed prime $p$ is different from $2$,
then we can use Corollary \ref{complete HRfiltration} instead of Theorem \ref{conjHKR} in the above proof.

A complex of $R$-modules is \emph{pseudo-coherent} if it is quasi-isomorphic to a (homologically) bounded below complex of finitely generated free $R$-modules.
Here, a bounded below complex means that $H_n(-)$ vanishes for $n\ll 0$, and which is called ``bounded to the right'' \cite[p.\ 243]{BMS19}.

\begin{df}
\label{pseudo coherent}
Let $R$ be a commutative ring.
An $\ul{R}$-module is \emph{finitely generated} if it is pointwise finitely generated.
An $\ul{R}$-module is \emph{free} 
if it is isomorphic to the sum of copies of $\ul{R}$ and $\ul{R^{\oplus \Z/2}}$. 
A complex of $\ul{R}$-modules is \emph{pseudo-coherent} if it is
quasi-isomorphic
to a bounded below complex of finitely generated free
$\ul{R}$-modules.
\end{df}

\begin{lem}
\label{lem:pseudo-coherence and exact sequence}
Let $\cE\to \cF\to \cG$ be a fiber sequence
in $\rD(\ul{R})$,
where $R$ is a commutative ring.
If two of $\cE$, $\cF$, and $\cG$ are pseudo-coherent,
then the remaining one is pseudo-coherent too.
\end{lem}
\begin{proof}
Without loss of generality
we may assume that $\cF$ and $\cG$ are pseudo-coherent and hence given by bounded below complexes of finitely generated free $\ul{R}$-modules.
Then $\cF$ and $\cG$ are cofibrant-fibrant objects with respect to the projective model structure on $\Ch(\ul{R})$,
so a morphism $\cF\to \cG$ in $\rD(\ul{R})$ is given by a morphism of bounded below complexes of finitely generated free $\ul{R}$-modules in $\Ch(\ul{R})$.
To conclude, observe that the mapping cone construction in $\Ch(\ul{R})$ for $\cF\to \cG$ yields a bounded below complex of finitely generated free $\ul{R}$-modules.
\end{proof}

\begin{lem}
\label{bounded pseudo-coherent}
Let $R$ be a commutative ring,
and let $\cF$ be a pseudo-coherent complex of $\ul{R}$-modules.
If there exists an integer $n$ such that $\ul{H}_m(\cF)=0$ for every integer $m<n$,
then there exists a quasi-isomorphism $\cE\to \cF$ such that $\cE$ is a bounded below complex of finitely generated free $\ul{R}$-modules and the entries of $\cE$ in degrees less than $n$ are $0$.
\end{lem}
\begin{proof}
Argue as in \cite[\href{https://stacks.math.columbia.edu/tag/064U}{Tag 064U}]{stacks}.
\end{proof}

Recall the full subcategory $\rD_{\geq n}(\ul{R})$ of $\rD(\ul{R})$ introduced in Definition \ref{abelian rho}.

\begin{lem}
\label{lem:pseudo-coherence and finitely generated}
Let $R$ be a commutative ring,
and let $\cF$ be a pseudo-coherent complex of $\ul{R}$-modules.
If $\cF\in \rD_{\geq n}(\ul{R})$ for some integer $n$,
then $\rho_n(\cF)$ is a finitely generated $\ul{R}$-module.
\end{lem}
\begin{proof}
After shifting and using Lemma \ref{shift 2},
we may assume $n=0$ or $n=1$.
By Lemma \ref{abelian filtration},
$\ul{H}_m(\cF)=0$ for every integer $m<n$.
Consider the quasi-isomorphism $\cE\to \cF$ in Lemma \ref{bounded pseudo-coherent}.
If $n=0$, then $\rho_0(\cE)=\ul{H}_0(\cE)$ is finitely generated.
If $n=1$, then $\ul{H}_1(\cE)$ is finitely generated, which implies that $\rho_1(\cE)$ is finitely generated.
\end{proof}

Consider the homotopy $t$-structure on $\Sp^{\Z/2}$ in \cite[Proposition A.3.4]{HP23}.
This is different from the slice filtrations in \cite{HHR} and \cite{HHR21}.
Let
\[
\tau_{\leq n},\tau_{\geq n}\colon \Sp^{\Z/2}\to \Sp^{\Z/2}
\]
be the associated truncation functors,
which are compatible with those of Recollection
\ref{model structures on A-modules} under Proposition \ref{greenequivalence} after forgetting the module structure.

\begin{lem}
\label{lem:pseudo-coherence}
Let $R$ be a perfectoid ring.
Then $\HR(R;\Z_p)$ and $\THR(R;\Z_p)$ are pseudo-coherent complexes of $\ul{R}$-modules.
\end{lem}
\begin{proof}
By Proposition \ref{prop:perfectoid HR},
we obtain a morphism $\ul{R}[n+n\sigma]\to \HR(R;\Z_p)$ in $\rD(\ul{R})$ corresponding to a generator of $R\simeq H_{n+n\sigma}^{\Z/2}\HR(R;\Z_p)$ for every integer $n\geq 0$.
Hence we obtain a morphism
\[
\bigoplus_{n\geq 0} \ul{R}[n+n\sigma]
\to
\HR(R;\Z_p)
\]
in $\rD(\ul{R})$,
which is a quasi-isomorphism.
In particular,
$\HR(R;\Z_p)$ is pseudo-coherent.
This implies that
\[
\THR(R;\Z_p)
\wedge_{\THR(\Z)}
\rH \ul{\Z}
\]
is pseudo-coherent by Proposition \ref{prop:hrcompleted}.
We can view $\ul{\pi}_n(\THR(\Z))$ as a $\ul{\pi}_0(\THR(\Z))$-module for every integer $n$.
Since $\ul{\pi}_0(\THR(\Z))\simeq \ul{\Z}$ by \cite[Theorem 5.1]{DMPR21} (see also \cite[Proposition 2.3.5]{HP23} which holds for $A=\Z$ by \cite{HP23erratum}) and $\ul{\pi}_n(\THR(\Z))$ is finite for $n>0$ by Lemma \ref{lem:finiteness of THR},
use Construction \ref{const:finite Mackey} to see that
\[
\THR(R;\Z_p)
\wedge_{\THR(\Z)}
\rH \ul{\pi}_{n} \THR(\Z)
\]
is pseudo-coherent
for every integer $n$,
i.e.,
it is quasi-isomorphic to a bounded below complex $\cF_n$ of finitely generated free $\ul{R}$-modules.
By \cite[Corollary 6.8.1]{BGS20} and Proposition \ref{connective}, $\cF_n$ is $(-1)$-connected.
Hence Lemma \ref{bounded pseudo-coherent} is applicable to $\cF_n$,
so
we may assume that $\cF_n$ is vanishing in negative degrees.
Let us inductively construct a quasi-isomorphism
\[
\cE_n
\xrightarrow{\simeq}
\THR(R;\Z_p)
\wedge_{\THR(\Z)}
\tau_{\leq n} \THR(\Z)
\]
with a bounded below complex $\cE_n$ of finitely generated free $\ul{R}$-modules.
If $n=0$,
then take $\cE_0:=\cF_n$.
Assume that we have constructed a quasi-isomorphism for $n$.
Then there exists a morphism $\cE_n[-1]\to \cF_{n+1}$ of chain complexes with a commutative square
\[
\begin{tikzcd}
\cE_{n}[-1]
\ar[d]\ar[r,"\simeq"]&
\THR(R;\Z_p)
\wedge_{\THR(\Z)}
\tau_{\leq n} \THR(\Z)[-1]\ar[d]
\\
\cF_{n+1}
\ar[r,"\simeq"]&
\THR(R;\Z_p)
\wedge_{\THR(\Z)}
\rH \ul{\pi}_{n+1} \THR(\Z)
\end{tikzcd}
\]
in $\rD(\ul{R})$.
Let $\cE_{n+1}$ be the mapping cone of $\cE_n[-1]\to \cF_{n+1}$ so that we have a quasi-isomorphism
\[
\cE_{n+1}
\xrightarrow{\simeq}
\THR(R;\Z_p)
\wedge_{\THR(\Z)}
\tau_{\leq n+1} \THR(\Z).
\]
When $n$ is sufficiently large,
the construction implies that $\cE_n$ becomes stable in a given degree.
This implies that $\lim_n \cE_n$ is a bounded below complex of finitely generated free $\ul{R}$-modules.
Lemma \ref{limit and smash} finishes the proof.
\end{proof}

\begin{lem}
\label{perfectoid basechange HR}
Let $R\to R'$ be a homomorphism of perfectoid rings.
Then the natural morphism
\[
\HR(R;\Z_p)
\sotimes_{\ul{R}}^\L
\ul{R'}
\to
\HR(R';\Z_p)
\]
in $\rD(\ul{R'})$ is an equivalence.
\end{lem}
\begin{proof}
The functor $\sotimes_{\ul{R}}^\L \ul{R'}\colon \rD(\ul{R})\to \rD(\ul{R'})$ sends the family $\cS_n(R)$ in Definition \ref{abelian rho} into $\rD_{\geq n}(\ul{R'})$ for every integer $n$.
This implies that $\sotimes_{\ul{R}}^\L \ul{R'}$ sends $\rD_{\geq n}(\ul{R})$ into $\rD_{\geq n}(\ul{R'})$.
In particular,
we have
\[
(P_n \HR(R;\Z_p))\sotimes_{\ul{R}}^\L \ul{R'}
\in
\rD_{\geq n}(\ul{R'})
\]
We claim
\begin{equation}
\label{eqn:perfectoid basechange HR}
(P^{m} \HR(R;\Z_p))\sotimes_{\ul{R}}^\L \ul{R'}
\in
\rD_{\leq m}(\ul{R'}).
\end{equation}
for every integer $m$.
This is obvious if $m<0$.
Assume that the claim holds for $m-1$.
In the cofiber sequence
\[
(P_m^m \HR(R;\Z_p))\sotimes_{\ul{R}}^\L \ul{R'}
\to
(P^m \HR(R;\Z_p))\sotimes_{\ul{R}}^\L \ul{R'}
\to
(P^{m-1} \HR(R;\Z_p))\sotimes_{\ul{R}}^\L \ul{R'},
\]
the first and third term are contained in $\rD_{\leq m}(\ul{R'})$ by Proposition \ref{prop:perfectoid HR} and the induction hypothesis, hence so is the second one.
This completes the induction argument for \eqref{eqn:perfectoid basechange HR}.
We will use \eqref{eqn:perfectoid basechange HR} in the case $m=n-1$.
Now we have equivalences
\begin{gather*}
(P_n \HR(R;\Z_p))\sotimes_{\ul{R}}^\L \ul{R'}
\simeq
P_n(\HR(R;\Z_p)\sotimes_{\ul{R}}^\L \ul{R'}),
\\
(P^{n-1} \HR(R;\Z_p))\sotimes_{\ul{R}}^\L \ul{R'}
\simeq
P^{n-1} (\HR(R;\Z_p)\sotimes_{\ul{R}}^\L \ul{R'})
\end{gather*}
in $\rD(\ul{R'})$.
Using these,
we have an equivalence
\[
(\rho_n \HR(R;\Z_p))\sotimes_{\ul{R}}^\L \ul{R'}
\simeq
\rho_n(\HR(R;\Z_p)\sotimes_{\ul{R}}^\L \ul{R'})
\]
in $\rD(\ul{R'})$.
Hence it suffices to show that the induced morphism
\[
(\rho_n\HR(R;\Z_p))
\sotimes_{\ul{R}}^\L \ul{R'}
\to
\rho_n\HR(R';\Z_p)
\]
in $\rD(\ul{R'})$ is an equivalence for every integer $n$.
By Proposition \ref{prop:perfectoid HR},
the claim is clear if $n<0$ or $n$ is odd,
Furthermore,
if $n\geq 0$ and $n$ is even,
then we have isomorphisms $\rho_n\HR(R;\Z_p)\simeq \ul{R}$ and $\rho_n\HR(R';\Z_p)\simeq \ul{R'}$.
Hence it suffices to show that the induced morphism
\[
(\pi_n\HH(R;\Z_p))
\sotimes_{R}^\L R'
\to
\pi_n\HH(R';\Z_p)
\]
in $\rD(R')$ is an equivalence.
This is observed in the proof of \cite[Theorem 6.1]{BMS19}.
\end{proof}

Combined with the computation of $\THR(\F_p)$ \cite[Theorem 5.15]{DMPR21},
the following result yields a new proof for the known computation  of $\THR(R;\Z_p)$ for perfect $\F_p$-algebra $R$ \cite[Remark 5.15]{DMP}, \cite[Proposition 6.26]{QS22}, since every perfect $\F_p$-algebra is a perfectoid ring.

\begin{lem}
\label{lem:perfectoid basechange}
Let $R\to R'$ be a homomorphism of perfectoid rings.
Then the natural morphism
\[
\THR(R;\Z_p)
\sotimes_{\ul{R}}^\L
\ul{R'}
\to
\THR(R';\Z_p)
\]
in $\rD(\ul{R'})$ is an equivalence.
\end{lem}
\begin{proof}
It suffices to show 
that the induced map of $\Z/2$-spectra
\[
\THR(R;\Z_p)
\wedge_{\rH \ul{R}}
\rH \ul{R'}
\to
\THR(R';\Z_p)
\]
is an equivalence.
By Proposition \ref{prop:hrcompleted} and Lemma \ref{perfectoid basechange HR},
the induced map of $\Z/2$-spectra
\[
\rH \ul{\Z}
\wedge_{\THR(\Z)}
\THR(R;\Z_p)
\wedge_{\rH \ul{R}}
\rH \ul{R'}
\to
\rH \ul{\Z}
\wedge_{\THR(\Z)}
\THR(R';\Z_p)
\]
is an equivalence.
As in the proof of Lemma \ref{lem:pseudo-coherence},
this implies that the induced map of $\Z/2$-spectra
\begin{align*}
&
\ul{\pi}_n(\THR(\Z))
\wedge_{\THR(\Z)}
\THR(R;\Z_p)
\wedge_{\rH \ul{R}}
\rH \ul{R'}
\\
\to
&
\ul{\pi}_n(\THR(\Z))
\wedge_{\THR(\Z)}
\THR(R';\Z_p)
\end{align*}
is an equivalence for every integer $n$.
By induction on $n$,
we see that the induced map of $\Z/2$-spectra
\begin{align*}
&
\tau_{\leq n}(\THR(\Z))
\wedge_{\THR(\Z)}
\THR(R;\Z_p)
\wedge_{\rH \ul{R}}
\rH \ul{R'}
\\
\to
&
\tau_{\leq n}(\THR(\Z))
\wedge_{\THR(\Z)}
\THR(R';\Z_p)
\end{align*}
is an equivalence for every integer $n$.
Take $\lim_n$ on both sides,
and use Lemma \ref{limit and smash} to conclude.
\end{proof}

In \cite[Theorem 3.4.3]{HP23}, we proved a descent theorem for $\THR$ with respect to the isovariant \'etale topology. A key ingredient for this was the \'etale base change Theorem 3.2.3 in loc.\ cit.. It is an obvious question to ask if $\THR$ satisfies descent even for the isovariant flat topology. When restricted to sites with trivial $C_2$-action, the above Lemma \ref{lem:perfectoid basechange} yields base change for the special case of perfectoid rings.
In \cite{Par23},
fpqc and syntomic descent for $\THR$ will be investigated.

\begin{lem}
\label{lem:rational THR(Z)}
The induced morphism
\[
\ul{\Q}
\to
\THR(\Z)
\sotimes_{\ul{\Z}}^\L
\ul{\Q}
\]
in $\rD(\ul{\Q})$ is an equivalence.
\end{lem}
\begin{proof}
It suffices to show that the induced map of $\Z/2$-spectra
\[
\rH \ul{\Q}
\to
\THR(\Z)
\wedge_{\rH \ul{\Z}}
\rH \ul{\Q}
\]
is an equivalence.
Both sides vanish after applying $\Phi^{\Z/2}$ since $\Phi^{\Z/2}\rH \ul{\Q}$ vanishes and $\Phi^{\Z/2}$ is monoidal.
Hence it suffices to show the claim after applying $i^*$,
i.e.,
it suffices to show that the induced map of spectra
\[
\rH \Q \to \THH(\Z) \wedge_{\rH \Z} \rH \Q
\]
is an equivalence.
This is a consequence of the finiteness of $\pi_*\THH(\Z)$ in \cite[Lemma 2.5]{BMS19}.
\end{proof}

The computation of $\ul{\pi}_*\THR(\Z)$ in \cite[Theorem B.0.1]{AKF} may be used to simplify the above proof,
but again, compare Remark \ref{DMPR conjecture}, we do not rely on this.

We are still working on the real refinement of \cite[Theorem 6.1]{BMS19}, which will be achieved in Theorem \ref{perfectoid THR} below. Some of the following results should also be compared with \cite[Proposition IV.4.2]{NS}.

\medskip

For a commutative ring $R$,
we have the map to the first smash factor
\(
\THR(R)
\to
\THR(R)\wedge_{\THR(\Z)}\rH \Z
=
\HR(R).
\)
After taking $(-)_p^\wedge$ on both sides, we obtain the morphism $\THR(R;\Z_p)\to \HR(R;\Z_p)$ in $\rD(\ul{R})$.

\begin{lem}
\label{lem:zeroth slice}
Let $R$ be a perfectoid ring.
Then the induced morphisms of $\ul{R}$-modules
\[
\ul{R}
\to
\rho_0 \THR(R;\Z_p)
\to
\rho_0 \HR(R;\Z_p)
\]
are isomorphisms.
\end{lem}
\begin{proof}
The Frobenius $\varphi\colon R/2\to R/2$ is surjective by assumption if $p=2$, and this holds for odd $p$ too since $R/2=0$ in this case. Now
combine \cite{HP23erratum} and Proposition \ref{prop:perfectoid HR}.
\end{proof}

For an $\F_p$-algebra $R$,
the \emph{colimit perfection} is $\colim(R \xrightarrow{\varphi} R \xrightarrow{\varphi} \cdots)$,
where $\varphi$ is the Frobenius, see e.g.\ \cite[Remark 8.15]{BMS19}.
\begin{lem}
\label{lem:induction}
Let $R$ be a perfectoid ring,
let $R'$ be the colimit perfection of $R/p$,
and let $n$ be an integer.
If $\rho_i\THR(R;\Z_p)$ is a finitely generated free $\ul{R}$-module and the induced morphism of $\ul{R}$-modules 
\[
\rho_i\THR(R;\Z_p)
\sotimes_{\ul{R}}
\ul{R'}
\to
\rho_i\THR(R';\Z_p)
\]
is an isomorphism for every integer $i<n$,
then $\rho_n\THR(R;\Z_p)$ is a finitely generated $\ul{R}$-module,
and the induced morphism of $\ul{R}$-modules
\[
\rho_n\THR(R;\Z_p)
\sotimes_{\ul{R}}
\ul{R'}
\to
\rho_n\THR(R';\Z_p)
\]
is an isomorphism.
A similar result holds for $\HR$ too.
\end{lem}
\begin{proof}
We focus on the case of $\THR$ since the proofs are similar.
By Lemmas \ref{lem:pseudo-coherence and exact sequence},
\ref{lem:pseudo-coherence and finitely generated},
and \ref{lem:pseudo-coherence},
$\rho_n\THR(\ul{R};\Z_p)$ is finitely generated.

A projective $\ul{R}$-module is flat due to \cite[Proposition A.5.3]{HP23}.
It follows that the induced morphism of $\Z/2$-spectra
\[
P_i^i\THR(R;\Z_p)
\wedge_{\rH \ul{R}}
\rH \ul{R'}
\to
P_i^i\THR(R';\Z_p)
\]
is an equivalence for every integer $i<n$.
This implies that the induced map of $\Z/2$-spectra
\[
P^{n-1}\THR(R;\Z_p)
\wedge_{\rH \ul{R}}
\rH \ul{R'}
\to
P^{n-1}\THR(R';\Z_p)
\]
is an equivalence.
Together with Lemma \ref{lem:perfectoid basechange},
we see that the induced map of $\Z/2$-spectra
\[
P_n\THR(R;\Z_p)
\wedge_{\rH \ul{R}}
\rH \ul{R'}
\to
P_n\THR(R';\Z_p)
\]
is an equivalence.
Take $\ul{\pi}_{m+m\sigma}$ (resp.\ $\ul{\pi}_{m+1+m\sigma}$) on both sides if $n=2m$ (resp.\ $n=2m+1$) to obtain the desired isomorphism.
\end{proof}

\begin{lem}
\label{lem:Nakayama}
Let $R$ be a perfectoid ring,
let $R'$ be the colimit perfection of $R/p$,
and let $f\colon M'\to M$ be a morphism of finitely generated $\ul{R}$-modules.
If the induced morphism $g\colon M\sotimes_{\ul{R}}\ul{R'}
\to
M'\sotimes_{\ul{R}} \ul{R'}$ is an epimorphism,
then $f$ is an epimorphism.
\end{lem}
\begin{proof}
Lemma \ref{lem:box base change} implies that the induced homomorphisms of $R'$-modules
\[
M'(C_2/e)\otimes_R R'
\to
M(C_2/e)\otimes_R R'
\text{ and }
M'(C_2/C_2)\otimes_R R'
\to
M(C_2/C_2)\otimes_R R'
\]
are epimorphisms.
Nakayama's lemma finishes the proof since the kernel of $R\to R'$ is contained in the Jacobson radical of $R$.
\end{proof}

\begin{lem}
\label{lem:first slice}
Let $R$ be a perfectoid ring.
Then we have $\rho_1\THR(R;\Z_p)\simeq 0$.
\end{lem}
\begin{proof}
Consider the map $\rH\ul{R} \to \THR(R)$ to the second smash factor.
Let $R'$ be the colimit perfection of $R/p$.
Observe that $R'$ is a perfect $\F_p$-algebra.
By Lemma \ref{lem:perfectoid basechange} for $\F_p\to R'$ and the computation of $\THR(\F_p)$ obtained by Dotto-Moi-Patchkoria-Reeh \cite[Theorem 5.15]{DMPR21},
we have $\rho_1\THR(R')\simeq 0$.
Together with Lemmas \ref{lem:pseudo-coherence and finitely generated}, \ref{lem:pseudo-coherence}, \ref{lem:zeroth slice}, and \ref{lem:induction},
we have an isomorphism
\[
\rho_1\THR(R;\Z_p)\sotimes_{\ul{R}}\ul{R'}
\cong
0.
\]
Lemma \ref{lem:Nakayama} finishes the proof.
\end{proof}

\begin{lem}
\label{lem:second slice}
Let $R$ be a perfectoid ring.
Then the induced morphism of $\ul{R}$-modules
\[
\rho_2 \THR(R;\Z_p)
\to
\rho_2 \HR(R;\Z_p)
\]
is an epimorphism. 
\end{lem}
\begin{proof}
Assume first that $R$ is a perfect $\F_p$-algebra.
Then $R$ is a perfectoid ring.
By Lemma \ref{lem:perfectoid basechange},
we reduce to the case when $R=\F_p$.
Consider the induced commutative square
\begin{equation}
\label{eqn:second slice}
\begin{tikzcd}
\pi_{1+\sigma}\THR(\F_p)\ar[d]\ar[r]&
\pi_{1+\sigma}\HR(\F_p)\ar[d,"\simeq"]
\\
\pi_2\THH(\F_p)\ar[r,"\simeq"]&
\pi_2\HH(\F_p)
\end{tikzcd}
\end{equation}
whose vertical homomorphisms are the restriction homomorphisms.
The left vertical homomorphism is an epimorphism by \cite[Lemma 5.18(ii)]{DMPR21}.
The right vertical homomorphism can be identified with $\pi_{1+\sigma}(\ul{R}[1+\sigma])\to \pi_2(R[2])$ by Proposition \ref{prop:perfectoid HR},
which is an isomorphism.
The lower horizontal homomorphism is an isomorphism by \cite[Proposition IV.4.2]{NS}.
Hence the upper horizontal homomorphism is an epimorphism.

For general $R$,
let $R'$ be the colimit perfection of $R/p$.
Since $R'$ is a perfect $\F_p$-algebra,
we know that the induced morphism of $\ul{R'}$-modules
\[
\rho_2 \THR(R';\Z_p)
\to
\rho_2 \HR(R';\Z_p)
\]
is an epimorphism.
By Lemmas \ref{lem:zeroth slice}, \ref{lem:induction}, and \ref{lem:first slice},
$\rho_2\THR(R;\Z_p)$ is a finitely generated $\ul{R}$-module,
and the induced morphism of $\ul{R'}$-modules
\[
\rho_2 \THR(R;\Z_p)\sotimes_{\ul{R}}\ul{R'}
\to
\rho_2 \THR(R';\Z_p)
\]
is an isomorphism.
On the other hand,
Proposition \ref{prop:perfectoid HR} and
Lemma \ref{lem:induction} imply that $\rho_2(\HR(R;\Z_p))$ is a finitely generated $\ul{R}$-module and the induced morphism of $\ul{R'}$-modules
\[
\rho_2 \HR(R;\Z_p)\sotimes_{\ul{R}} \ul{R'}
\to
\rho_2\HR(R';\Z_p)
\]
is an isomorphism.
Lemma \ref{lem:Nakayama} finishes the proof.
\end{proof}

For $A\in \NAlg^{\Z/2}$,
let
\[
T_A(S^{1+\sigma})
:=
\bigoplus_{n=0}^\infty
\Sigma^{n+\sigma n} A
\]
denote the free associative $A$-algebra on $S^{1+\sigma}$.
Dotto-Moi-Patchkoria-Reeh \cite[Theorem 5.15]{DMPR21} show that there is an equivalence of associative $\Z/2$-equivariant ring spectra
\begin{equation}
T_{\rH \ul{\F_p}}(S^{1+\sigma})
\simeq
\THR(\F_p)
\end{equation}
for every prime $p$.
This is a crucial ingredient for our computation of $\THR$ of perfectoid rings below, similar to  B{\"o}kstedt's computation of $\THH(\F_p)$ being crucial for \cite{BMS19}.

In the case that $R$ is a perfect field of characteristic $2$, Dotto-Moi-Patchkoria compute $\THR(R;\Z_p)$ in \cite[Remark 5.15]{DMP}.
It is observed in \cite[Proposition 6.26]{QS22} that the proof holds for a perfect $\F_p$-algebra $R$ and any prime $p$.
We now generalize this computation to the mixed characteristic
case.
Note that if 2 is not invertible, the proof of the following result relies on Theorem \ref{conjHKR},
whose proof appears in \cite[Theorem 6.20]{Par23}.

\begin{thm}
\label{perfectoid THR}
Let $R$ be a perfectoid ring.
Then there is a natural equivalence of 
associative $\Z/2$-equivariant ring spectra
\[
\THR(R;\Z_p)
\simeq
T_{\rH \ul{R}} (S^{1+\sigma}),
\]
In particular,
we have
\[
\rho_n \THR(R;\Z_p)
\cong
\left\{
\begin{array}{ll}
\ul{R}  & \text{if $n$ is even and nonnegative},
\\
0 & \text{otherwise}.
\end{array}
\right.
\]
\end{thm}
\begin{proof}
Lemma \ref{lem:second slice} gives an epimorphism of $R$-modules
\[
\pi_{1+\sigma}\THR(R;\Z_p)
\to
\pi_{1+\sigma}\HR(R;\Z_p).
\]
Choose $\tilde{u}\in \pi_{1+\sigma}\THR(R;\Z_p)$ that maps to a generator of $\pi_{1+\sigma}\HR(R;\Z_p)\cong R$,
where the isomorphism follows from Proposition \ref{prop:perfectoid HR}.
Consider the corresponding map of $\Z/2$-spectra $x\colon \Sigma^{1+\sigma}\Sphere\to \THR(R;\Z_p)$.
Since $\THR(R;\Z_p)$ is an $\rH \ul{R}$-algebra,
we obtain the induced map
\[
T_{\rH \ul{R}}(S^{1+\sigma})
\to
\THR(R;\Z_p).
\]

We claim that the induced morphism
\[
\rho_n T_{\rH \ul{R}}(S^{1+\sigma})
\to
\rho_n\THR(R;\Z_p).
\]
is an isomorphism for every integer $n$,
which implies the theorem due to Proposition \ref{conservative slices}.
The claim holds when $R=\F_p$ by the proof of \cite[Theorem 5.15]{DMPR21} and the commutativity of \eqref{eqn:second slice}.
Lemmas \ref{lem:box base change} and \ref{lem:perfectoid basechange} imply
that the claim holds when $R$ has characteristic $p$.

Assume that $R$ has mixed characteristic.
Note that we have
\[
\rho_n T_{\rH \ul{R}}(S^{1+\sigma})
\cong
\left\{
\begin{array}{ll}
\ul{R}  & \text{if $n$ is even and nonnegative},
\\
0 & \text{otherwise}.
\end{array}
\right.
\]
We proceed by induction on $n$.
The claim holds for $n<0$ since $\THR(R;\Z_p)$ is $(-1)$-connected by Proposition \ref{connective}.

Assume $n\geq 0$.
Let $R'$ be the colimit perfection of $R/p$.
By the induction hypothesis
and Lemma \ref{lem:induction},
the induced morphism of $\ul{R'}$-modules
\[
\rho_n\THR(R;\Z_p)\sotimes_{\ul{R}} \ul{R'}
\to
\rho_n\THR(R';\Z_p)
\]
is an isomorphism.
Since the claim holds for the perfect $\F_p$-algebra $R'$,
the induced morphism of $\ul{R'}$-modules
\[
\rho_n T_{\rH \ul{R}}(S^{1+\sigma})\sotimes_{\ul{R}}\ul{R'}
\to
\rho_n \THR(R;\Z_p)\sotimes_{\ul{R}} \ul{R'}
\]
is an isomorphism using the obvious isomorphism $\rho_n T_{\rH \ul{R}}(S^{1+\sigma}) \sotimes_{\ul{R}}\ul{R'}\simeq \rho_n T_{\rH \ul{R'}}(S^{1+\sigma})$.
Together with Lemma \ref{lem:Nakayama},
we see that the induced morphism of $\ul{R}$-modules
\[
\rho_n T_{\rH \ul{R}}(S^{1+\sigma}) \to \rho_n\THR(R;\Z_p)
\]
is an epimorphism.
Since $R$ is reduced,
it suffices to show that the ranks of $\rho_n\THR(R;\Z_p)(C_2/e)$ and $\rho_n\THR(R;\Z_p)(C_2/C_2)$ at
any point of $\Spec(R)$ are at least one.
This claim holds for any point of $\Spec(R')$, which is a subset of $\Spec(R)$ as a topological space.
The complement of the inclusion $\Spec(R')\subset \Spec(R)$ consists of the characteristic zero points.
We have an equivalence
\[
\rho_n\THR(R;\Z_p)\wedge_{\rH \ul{\Z}} \rH \ul{\Q}
\simeq
\rho_n\HR(R;\Z_p)\wedge_{\rH \ul{\Z}} \rH \ul{\Q}
\]
using Lemma \ref{lem:rational THR(Z)}.
Proposition \ref{prop:perfectoid HR} finishes the proof.
\end{proof}

In particular, we have an isomorphism of
graded Green functors,
$\underline{\pi}_{(1+\sigma)*}(\THR(R)) \cong \ul{R} [\tilde{u}]$
with $\tilde{u} \in \pi_{1 + \sigma}(\THR(R))$,
where the Green functor structure follows e.g.\ from \cite[section A.4]{HP23}.
Also note that the result shows that for any perfectoid ring $R$, the $\Z/2$-spectrum $\THR(R)$ is very even in the sense of Definition \ref{def:even}.

The following result refines \cite[Theorem 6.7]{BMS19}. 

\begin{thm}\label{cofibersequence}
Let $R$ be a perfectoid ring,
and let $A$ be an $R$-algebra.
Then there is an equivalence of $\rH \ul{A}$-modules
\[
\HR(A/R;\Z_p)
\simeq
\THR(A;\Z_p)
\wedge_{\THR(R;\Z_p)}
\rH \ul{R}.
\]
Furthermore,
there is a cofiber sequence
$$ \THR(A;\Z_p)[1+ \sigma] \xrightarrow{\tilde{u}}  \THR(A;\Z_p) \to \HR(A/R;\Z_p).$$
\end{thm}
\begin{proof}
Thanks to Theorem \ref{perfectoid THR},
we have a cofiber sequence
\[
\THR(R;\Z_p)[1+\sigma] \xrightarrow{\tilde{u}}  \THR(R;\Z_p) \to \HR(R/R;\Z_p).
\]
After taking $\THR(A;\Z_p)\wedge_{\THR(R;\Z_p)}$,
we obtain a cofiber sequence
\[
\THR(A;\Z_p)[1+\sigma] \xrightarrow{\tilde{u}}  \THR(A;\Z_p) \to \THR(A;\Z_p)\wedge_{\THR(R;\Z_p)}\rH \ul{R}.
\]
Since $\THR(A;\Z_p)$ and $\THR(A;\Z_p)[1+\sigma]$ are derived $p$-complete,
we see that
$\THR(A;\Z_p)\wedge_{\THR(R;\Z_p)}\rH \ul{R}$ is derived $p$-complete.
Together with Proposition \ref{monoidal completion algebra},
we obtain a natural equivalence of $\rH \ul{A}$-modules
\[
\THR(A;\Z_p)\wedge_{\THR(R;\Z_p)}\rH \ul{R}
\simeq
(\THR(A)\wedge_{\THR(R)}\rH \ul{R})_p^\wedge.
\]
The latter one is identified with $\HR(A/R;\Z_p)$,
which finishes the proof.
\end{proof}

The proof shows that this cofiber sequence is $S^{\sigma}$-equivariant, which will be important in \cite{Par23} when considering $\TCR$.

\medskip

\begin{rmk}
\label{perfectoid ring with nontrivial involution}
We continue the discussion of Proposition \ref{HRsigma} and Example \ref{thrgm}. Now, let $R$ be an $\F_p$-algebra.
Then the map $p\colon \alpha^*(\ul{R}[S^1])\to \alpha^*(\ul{R}[S^1])$ in $\rD(\ul{R})$ induced by $p\colon S^1\to S^1$ can be identified with $\id\oplus 0\colon R\oplus R[1]\to R\oplus R[1]$.
On the other hand,
using $\THR(R[\Z^\sigma])\simeq \THR(R)\wedge \THR(\Sphere[\Z^\sigma])$ from \cite[Proposition 2.1.5]{HP23} and the above description of $\THR(\Sphere[\Z^\sigma])$,
we have an equivalence
\[
\THR(R[\Z^\sigma])
\simeq
\THR(R)\sotimes_{\ul{R}}^\L (\ul{R[\Z^\sigma]}\oplus \cF[1])
\]
in $\rD(\ul{R[\Z^\sigma]})$ for some $\cF\in \rD(\ul{R[\Z^\sigma]})$ that forgets to $\ul{R}\oplus \bigoplus_{j\geq 1} \ul{R^{\oplus \Z/2}}\in \rD(\ul{R})$.
The above description of $p$ on $\alpha^*(\Sphere[S^1])$ implies that the map $p\colon \THR(R[\Z^\sigma])\to \THR(R[\Z^\sigma])$
sends the direct summand $\THR(R)\sotimes_{\ul{R}}^\L \cF[1]$ to $0$.
Let $\alpha\colon \THR(R[\Z^\sigma])\to \THR(R[\Z^\sigma])$ be the morphism in $\rD(\ul{R[\Z^\sigma]})$ induced by $p\colon \Z^\sigma\to \Z^\sigma$.
We obtain an equivalence
\[
\THR(R)\sotimes_{\ul{R}}^\L \ul{R[\Z[\tfrac{1}{p}]^\sigma]}
\simeq
\colim(\THR(R[\Z^\sigma]) \xrightarrow{\alpha} \THR(R[\Z^\sigma])\xrightarrow{\alpha}\cdots)
\]
in $\rD(\ul{R[\Z^\sigma]})$,
where $\Z[\tfrac{1}{p}]^\sigma$ is the abelian group $\Z[\tfrac{1}{p}]$ with the involution $x\mapsto -x$ for $x\in \Z[\tfrac{1}{p}]$,
which satisfies
\[
\Z[\tfrac{1}{p}]^\sigma
:=
\colim(\Z^\sigma \xrightarrow{p} \Z^\sigma\xrightarrow{p}\cdots).
\]
The functor $\THR\colon \NAlg^{\Z/2}\to \NAlg^{\Z/2}$ preserves colimits,
and the forgetful functor $\NAlg^{\Z/2}\to \Sp^{\Z/2}$ preserves sifted colimits by Proposition \ref{prop:NAlg has colimits}.
Hence we obtain an equivalence
\[
\THR(R[\Z[\tfrac{1}{p}]^\sigma])
\simeq
\THR(R)\sotimes_{\ul{R}}^\L \ul{R[\Z[\tfrac{1}{p}]^\sigma]}
\]
in $\rD(\ul{R[\Z^\sigma]})$.
Taking the induced maps of simplicial sets $\Z[\tfrac{1}{p}]^\sigma\to \Ndi \Z[\tfrac{1}{p}]^\sigma \to \Z[\tfrac{1}{p}]^\sigma$ into account,
we see that the above equivalence for $\THR(R[\Z[\tfrac{1}{p}]^\sigma])$ is indeed an equivalence in $\rD(\ul{R[\Z[\tfrac{1}{p}]^\sigma]})$.
The group ring of a uniquely $p$-divisible group over a perfect $\F_p$-algebra is a perfect $\F_p$-algebra and thus  perfectoid rings by \cite[Example 3.6]{BMS18}.
Hence if $R$ is a perfect $\F_p$-algebra, and $A:=R[\Z[\tfrac{1}{p}]^\sigma]$ the associated perfect $\F_p$-algebra with a nontrivial involution, then in particular the underlying ring $R[\Z[\tfrac{1}{p}]]$ is perfectoid.
In Example \ref{thrgm},
we show an equivalence $\THR(A)\simeq \THR(R)\sotimes_{\ul{R}}^\L \ul{A}$ in $\rD(\ul{A})$.
Theorem \ref{perfectoid THR} gives an equivalence of associative $\Z/2$-equivariant ring spectra
\(
T_{\rH \ul{R}}(S^{1+\sigma})
\simeq
\THR(R;\Z_p).
\)
Together with
\(
T_{\rH \ul{R}}(S^{1+\sigma})
\sotimes_{\ul{R}}^\L \ul{A}
\simeq
T_{\rH \ul{A}}(S^{1+\sigma}),
\)
we obtain an equivalence of associative $\Z/2$-equivariant ring spectra
\[
T_{\rH \ul{A}}(S^{1+\sigma})
\simeq
\THR(A;\Z_p).
\]
Hence we expect that Theorem \ref{perfectoid THR} is satisfied by a more general class of perfectoid rings $R$ with certain (or all) nontrivial involutions.
Note however that this spectrum is not necessarily very even, as $\rH \ul{A}$ is not very even for rings $A$ with a non-trivial involution.
We can argue similarly in the setting of Proposition \ref{thra2} to obtain an equivalence of associative $\Z/2$-equivariant ring spectra
\[
T_{\rH \ul{B}}(S^{1+\sigma})
\simeq
\THR(B;\Z_p)
\]
with $B:=R[\N^{\oplus \Z/2}[\tfrac{1}{p}]]$ if $R$ is a perfect $\F_p$-algebra.
Observe that Theorem \ref{cofibersequence} extends to $R[\Z[\tfrac{1}{p}]^\sigma]$ and $R[\N^{\oplus \Z/2}[\tfrac{1}{p}]]$.
\end{rmk}

\appendix

\section{\texorpdfstring{$\infty$-}{Infinity }category of modules and sifted colimits}
\label{secA}

The purpose of this appendix is to show that the $\infty$-category of modules admits sifted colimits under reasonable hypotheses by arguing as in \cite[Proposition 3.2.3.1]{HA}.

\begin{lem}
\label{module lemma}
Let $p\colon \cM\to \Delta^1$ be a cocartesian fibration of $\infty$-categories,
which classifies
a functor $f\colon \cC\to \cD$ in the sense of \cite[Definition 3.3.2.2]{HTT},
where $\cC$ and $\cD$ are the fibers of $p$ at $\{0\}$ and $\{1\}$.
Let $\cC'$ be the full subcategory of $\Fun_{\Delta^1}(\Delta^1,\cM)$ spanned by the functors $\Delta^1\to \cM$
which are $p$-cocartesian edges
(defined in \cite[Definition 2.4.11 and before Proposition 2.4.1.8]{HTT}).
If $K$ is a simplicial set and $f$ preserves $K$-indexed colimits,
then the inclusion functor $\cC'\to \Fun_{\Delta^1}(\Delta^1,\cM)$ preserves $K$-indexed colimits.
\end{lem}
\begin{proof}
Let $i_0\colon \{0\}\to \Delta^1$ and $i_1\colon \{1\}\to \Delta^1$ be the inclusion functors.
By \cite[Proof of Lemma 5.4.7.15]{HTT},
there exists a commutative diagram of $\infty$-categories
\[
\begin{tikzcd}
\cC\ar[r,"g"']\ar[rrd,"f"',bend right=15]\ar[rrr,bend left=15,"\id"]&
\cC'\ar[r,hookrightarrow,"a"']&
\Fun(\Delta^1,\cM)\ar[d,"i_1^*"]\ar[r,"i_0^*"']&
\cC
\\
&
&
\cD
\end{tikzcd}
\]
for some equivalence of $\infty$-categories $g\colon \cC\to \cC'$,
where $a$ is the inclusion functor.
The functors $i_0^*$ and $i_1^*$ preserve $K$-indexed colimits,
and the pair of functors $(i_0^*,i_1^*)$ is conservative.
Hence the functor $ag$ preserves $K$-indexed colimits since $f$ and $\id$ preserve $K$-indexed colimits.
Since $g$ is an equivalence of $\infty$-categories,
$a$ preserves $K$-indexed colimits.
\end{proof}

For a monoidal $\infty$-category $\cC^\otimes$,
let $\LMod(\cC)$ denote the $\infty$-category of left module objects of $\cC$ in \cite[Definition 4.2.1.13]{HA}.
To form this,
we need the $\infty$-operads $\Assoc^\otimes$ and $\LM^\otimes$ in \cite[Definitions 4.1.1.3, 4.2.1.7]{HA}.

\begin{prop}
\label{left module sifted colimit}
Let $\cC^{\otimes}$ be a monoidal $\infty$-category such that $\cC$ admits sifted colimits and the monoidal product $\otimes$ preserves sifted colimits in each variable.
Then $\LMod(\cC)$ admits sifted colimits,
and the forgetful functor
\[
p
\colon
\LMod(\cC)
\to
\Alg(\cC) \times \cC
\]
sending a left $A$-module $M$ to $(A,M)$ preserves sifted colimits.
\end{prop}
\begin{proof}
We have the cofibration of $\infty$-operads $q\colon \cC^\otimes \to \Assoc^\otimes$.
The assumption that $\cC$ has sifted colimits and $\otimes$ preserves sifted colimits in each variable implies that $q$ is compatible with sifted colimits in the sense of \cite[Definition 3.1.1.18]{HA}.
Let $u\colon \LM^\otimes\to \Assoc^\otimes$ be the functor in \cite[Remark 4.2.1.9]{HA}.

By \cite[Lemma 3.2.3.7]{HA},
the associated functor $f_!\colon \cC_X^{\otimes} \to \cC_Y^{\otimes}$ preserves sifted colimits for every morphism $f\colon X\to Y$ in $\Assoc^\otimes$.
Let $K$ be a simplicial set.
Apply \cite[Corollary 4.3.1.11]{HTT} to $q$ and \cite[Lemma 3.2.2.9]{HA} to $q^{op}$ and $(\LM^{\otimes})^{op}$ to deduce the following results:
\begin{enumerate}
\item[(i)]
The induced functor
$
r\colon
\Fun(\LM^\otimes,\cC^\otimes)\to \Fun(\LM^\otimes,\Assoc^\otimes)
$
admits relative sifted colimits in the sense of \cite[Definition 4.3.1.1]{HTT}.
\item[(ii)]
A map $K^\rhd\to \Fun(\LM^\otimes,\cC^\otimes)$ is an $r$-colimit diagram if and only if the induced map $K^\rhd\to \Fun(\{X\},\cC^\otimes)$ is a colimit diagram for every $X\in \LM^\otimes$.
\end{enumerate}
By restricting to $u\in \Fun(\LM^\otimes,\Assoc^\otimes)$,
we deduce the following results:
\begin{enumerate}
\item[(i')]
The $\infty$-category $\Fun_{\Assoc^\otimes}(\LM^\otimes,\cC^\otimes)$ admits sifted colimits.
\item[(ii')]
A map $K^\rhd\to \Fun_{\Assoc^\otimes}(\LM^\otimes,\cC^\otimes)$ is a colimit diagram if and only if the induced map $K^\rhd\to \cC_{u(X)}^\otimes$ is a colimit diagram for every $X\in \LM^\otimes$.
\end{enumerate}
The full subcategory $\LMod(\cC)$ of $\Fun_{\Assoc^\otimes}(\LM^\otimes,\cC^\otimes)$ is spanned by the maps of $\infty$-operads,
i.e.,
those functors $\LM^\otimes\to\cC^\otimes$ sending inert morphisms in $\LM^\otimes$ to inert morphisms in $\cC^\otimes$.
We now claim the following:
\begin{enumerate}
\item[(i'')]
The $\infty$-category $\LMod(\cC)$ admits sifted colimits,
and the inclusion functor $\LMod(\cC)\to \Fun_{\Assoc^\otimes}(\LM^\otimes,\Assoc^\otimes)$ preserves sifted colimits.
\item[(ii'')]
A map $K^\rhd\to \LMod(\cC)$ is a colimit diagram if and only if the induced map $K^\rhd\to \cC_{u(X)}^\otimes$ is a colimit diagram for every $X\in \LM^\otimes$.
\end{enumerate}
Let $a\colon K\to \LMod(\cC)$ be a functor.
To show (i''),
it suffices to show that the colimit of $a$ computed in $\Fun_{\Assoc^\otimes}(\LM^\otimes,\cC^\otimes)$ is contained in $\LMod(\cC)$.
This amounts to show that for every inert morphism $e$ in $\LM^\otimes$,
the colimit of the restriction $K\to \Fun_{\Delta^1}(\Delta^1,\cC_{u(e)}^\otimes)$ corresponds to a morphism that is an inert morphism in $\cC^\otimes$,
i.e., a $q$-cartesian morphism in $\cC^\otimes$ by \cite[Proposition 2.1.2.22]{HA}.
Lemma \ref{module lemma} finishes the proof of (i'').
As a consequence of (i'') and (ii'),
we obtain (ii'').
\end{proof}

For a symmetric monoidal $\infty$-category $\cC^\otimes$,
let $\Mod(\cC)$ denote the underlying $\infty$-category of the generalized $\infty$-operad $\Mod(\cC)^{\otimes}$ in \cite[Definition 4.5.1.1]{HA}.
For a commutative algebra object $A$ of $\cC$,
let $\Mod_A(\cC)$ denote the restriction of $\Mod(\cC)$ to $A$-module objects.
We often use the simpler notation $\Mod_A$ instead of $\Mod_A(\cC)$.
By \cite[Corollary 4.5.1.6]{HA},
there is a natural equivalence of $\infty$-categories
\[
\Mod(\cC)\simeq \LMod(\cC)\times_{\Alg(\cC)}\CAlg(\cC).
\]

\begin{prop}
\label{module sifted colimit}
Let $\cC^{\otimes}$ be a symmetric monoidal $\infty$-category such that $\cC$ admits sifted colimits and the monoidal product $\otimes$ preserves sifted colimits in each variable.
Then $\Mod(\cC)$ admits sifted colimits,
and the forgetful functor
\[
p
\colon
\Mod(\cC)
\to
\CAlg(\cC) \times \cC
\]
sending an $A$-module $M$ to $(A,M)$ and preserves sifted colimits.
\end{prop}
\begin{proof}
This is an immediate consequence of Proposition \ref{left module sifted colimit}.
\end{proof}

\bibliography{THRPerfd}
\bibliographystyle{siam}

\end{document}